\documentclass[letterpaper, 11pt]{article}

\usepackage[margin=1in]{geometry}
\usepackage[bookmarks, colorlinks=true, plainpages = false, colorlinks=true,
            citecolor=red,
            linkcolor=blue,
            anchorcolor=red,
            urlcolor=blue]{hyperref}
\usepackage{url}
\usepackage{amssymb, amsmath, mathtools, amsfonts, amsthm,  dsfont} 
\usepackage{float, xcolor, xspace, graphicx, subfigure, appendix, enumitem, cleveref, booktabs, multirow, comment} 
\usepackage{algorithm, algpseudocode} 
\usepackage[normalem]{ulem}

\makeatletter
\algrenewcommand\alglinenumber[1]{\footnotesize #1. }  
\algrenewcommand\algorithmiccomment[1]{\hfill\textcolor{blue}{\(\triangleright\) #1}} 
\makeatother

\usepackage{prettyref}

\newrefformat{eq}{(\ref{#1})}
\newrefformat{sec}{Section~\ref{#1}}
\newrefformat{alg}{Algorithm~\ref{#1}}
\newrefformat{fig}{Figure~\ref{#1}}
\newrefformat{tab}{Table~\ref{#1}}
\newrefformat{rmk}{Remark~\ref{#1}}
\newrefformat{clm}{Claim~\ref{#1}}
\newrefformat{claim}{Claim~\ref{#1}}
\newrefformat{def}{Definition~\ref{#1}}
\newrefformat{assump}{Assumption~\ref{#1}}
\newrefformat{cor}{Corollary~\ref{#1}}
\newrefformat{lmm}{Lemma~\ref{#1}}
\newrefformat{prop}{Proposition~\ref{#1}}
\newrefformat{app}{Appendix~\ref{#1}}
\newrefformat{thm}{Theorem~\ref{#1}}

\usepackage[size=tiny]{todonotes}

\theoremstyle{plain}
\newtheorem{theorem}{Theorem}
\newtheorem{lemma}{Lemma}
\newtheorem{proposition}{Proposition}
\newtheorem{corollary}{Corollary}

\theoremstyle{definition}
\newtheorem{definition}{Definition}

\newtheorem{remark}{Remark}
\newtheorem{claim}{Claim}

\newif\ifhideproofs
\ifhideproofs
    \usepackage{environ}
    \NewEnviron{hide}{}

\fi
\newif\ifdraft
\drafttrue
\newenvironment{proofsketch}{%
  \noindent\textit{Proof Sketch.}}{\hfill$\square$ }

\newcommand{\stepa}[1]{\overset{\rm (a)}{#1}}
\newcommand{\stepb}[1]{\overset{\rm (b)}{#1}}
\newcommand{\stepc}[1]{\overset{\rm (c)}{#1}}
\newcommand{\stepd}[1]{\overset{\rm (d)}{#1}}

\newcommand{\diff}{{\rm d}}

\newcommand{\Pois}{{\rm Pois}}

\newcommand{\E}{\mathbb{E}}

\makeatletter

\@tfor\@tempa:=ABCDEFGHIJKLMNOPQRSTUVWXYZ\do{
  \expandafter\edef\csname t\@tempa\endcsname{
    {\noexpand\widetilde{\@tempa}}
  }
}

\@tfor\@tempa:=ABCDEFGHIJKLMNOPQRSTUVWXYZ\do{
  \expandafter\edef\csname cal\@tempa\endcsname{
    \noexpand\mathcal{\@tempa}
  }
}

\@tfor\@tempa:=abcdefghijklmnopqrstuvwxyz\do{
  \expandafter\edef\csname sf\@tempa\endcsname{
    \noexpand\mathsf{\@tempa}
  }
}
\makeatother

\newcommand{\bfp}{\pL}
\newcommand{\bfq}{\pR}
\newcommand{\bfC}{\mathbf{C}}

\newcommand{\psihat}{\widehat{\psi}}

\renewcommand{\P}{\mathbb{P}}

\newcommand{\bL}{b^{\mathrm{L}}}
\newcommand{\bR}{b^{\mathrm{R}}}

\newcommand{\pL}{\mathbf{p}}
\newcommand{\pR}{\mathbf{q}}

\newcommand{\aL}{a^{\mathrm{L}}}
\newcommand{\aR}{a^{\mathrm{R}}}

\newcommand{\piL}{\pi^{\mathrm{L}}}
\newcommand{\piR}{\pi^{\mathrm{R}}}

\newcommand{\pihatL}{\widehat{\pi}^{\mathrm{L}}}
\newcommand{\pihatR}{\widehat{\pi}^{\mathrm{R}}}

\newcommand{\OS}{\mathrm{OS}}
\newcommand{\TS}{\mathrm{TS}}

\newcommand{\match}{\eta}
\newcommand{\matchnum}{M_{\star}}

\newcommand{\Adv}{\mathsf{Adv}_{\OS}}

\newcommand{\tmax}{t^\dagger}
\newcommand{\smax}{s^\dagger}

\title{Flexibility allocation in random bipartite matching markets: exact matching rates and dominance regimes}
\date{}

\author{
    Taha Ameen\thanks{
        T. Ameen is with the Department of Electrical and Computer Engineering and the Coordinated Science Lab, University of Illinois, Urbana IL, USA, \texttt{tahaa3@illinois.edu}. T. Ameen is supported by NSF Grant CCF 19-00636.
    },~
    Flore Sentenac\thanks{
        F. Sentenac is with the Department of Information Systems and Operations Management, HEC Paris Business School, France, \texttt{sentenac@hec.fr}.
    } ~and 
    Sophie H.\ Yu\thanks{
        S.\ H.\ Yu is with the Operations, Information and Decisions Department,   the Wharton School of Business, University of Pennsylvania, Philadelphia PA, USA,  \texttt{hysophie@wharton.upenn.edu}.
    }
}
\begin{document}
\maketitle

\begin{abstract}
This paper studies how a fixed flexibility budget should be allocated  across the two sides of a balanced bipartite matching market. We model compatibilities via a sparse bipartite stochastic block model in which flexible agents are more likely to connect with agents on the opposite side, and derive an exact variational formula for the asymptotic matching rate under any flexibility allocation. The derivation extends  the local weak convergence framework of~\cite{bordenave2013matchings} from single-type to multi-type unimodular Galton--Watson trees, reducing the matching rate to an explicit low-dimensional optimization problem. Using this formula, we analytically investigate when the one-sided allocation, which concentrates all flexibility on one side, dominates the two-sided allocation and vice versa, sharpening and extending the comparisons of~\cite{freund2024twov3} which relied on approximate algorithmic bounds rather than an exact characterization of the matching rate. 
\end{abstract}

\tableofcontents

\section{Introduction}
 
In many applications, flexibility can be introduced on either side of a market: supply agents may be trained to handle a broader range of tasks, demand agents may be made more accommodating, or both. This raises a natural and practically important design question: \emph{for a fixed flexibility budget, is it better to concentrate flexibility on one side of the market, or to spread it across both sides?}
 
Inspired by~\cite{freund2024twov3}, we study this question through a bipartite compatibility graph $G_n$ with $n$ supply nodes and $n$ demand nodes. On each side, agents independently belong to one of two types: \emph{flexible} (type~$2$) or \emph{regular} (type~$1$). Supply nodes are each independently flexible with probability $\bL$ and demand nodes with probability $\bR$, defining the flexibility allocation $(\bL, \bR) \in [0,1]^2$.
 
Conditional on types, edges in $G_n$ form independently: a supply node of type $x$ and a demand node of type $y$ are connected with probability $c_{xy}/n$. The connection rates are collected in the matrix $\mathbf{C} \triangleq [c_{xy}]$ for $x, y \in \{1,2\}$. In our 
flexibility design problem, $\mathbf{C}$ is parametrized by a 
baseline rate $\alpha\geq 0$ and a flexibility premium 
$\alpha_f \geq \alpha$:
\[
    \mathbf{C}
    =
    \begin{bmatrix}
        2\alpha & \alpha+\alpha_f\\
        \alpha+\alpha_f & 2\alpha_f
    \end{bmatrix},
\]
so that a regular--regular pair connects at rate $2\alpha$, a 
mixed pair at rate $\alpha+\alpha_f$, and a flexible--flexible 
pair at rate $2\alpha_f$. 
A key feature of this parametrization is that the expected number of compatibility edges depends only on the total flexibility budget
\[
    B \triangleq \bL + \bR,
\]
and not on how it is split between the two sides. This makes $B$ a natural measure of the platform's investment in flexibility, holding constant the overall graph density while varying its allocation.
 
The platform's goal is to maximize the number of successful matches. Let $M_*(G_n)$ denote the size of a maximum matching in $G_n$. Our primary object of interest is the asymptotic matching rate
\[
    \eta(\bL, \bR; \mathbf{C})
    \triangleq
    \lim_{n\to\infty} \frac{\mathbb{E}[M_*(G_n)]}{n},
\]
where the dependence on $(\alpha, \alpha_f)$ is captured through $\mathbf{C}$.

For a fixed flexibility budget $B \in (0,1]$, two canonical allocations play a central role: the \emph{one-sided allocation} $(\bL,\bR) = (B, 0)$, which places all flexibility on the supply side, and the \emph{two-sided allocation} $(\bL,\bR) = (B/2, B/2)$, which distributes it equally across both sides.

The comparison between these two allocations is governed by two competing effects. \emph{Flexibility cannibalization} favors the one-sided allocation: when flexibility is present on both sides, flexible agents match disproportionately with each other, wasting compatibility that could have benefited regular agents. \emph{Flexibility asymmetry} favors the two-sided allocation: concentrating all flexibility on one side leaves regular agents on that side at a disadvantage, competing against flexible peers without receiving any flexibility benefit themselves. \cite{freund2024twov3} identify these two effects and show that neither allocation uniformly dominates the other: depending on $B$, $\alpha$, and $\alpha_f$, either can yield the larger matching size. Depending on the parameter regime, they rely on coupling arguments, Karp--Sipser analysis, or computer-aided proofs to establish which effect prevails. While these varied algorithmic and computational tools provide valuable bounds, they do not yield a unified, exact matching rate.

\paragraph{Contributions.} In this paper, we overcome these limitations by moving from algorithmic bounds to an exact asymptotic characterization. Our first contribution is an explicit variational formula for the asymptotic matching rate $\match(\bL,\bR;\bfC)$. Specifically, we show that
\[
    \match(\bL,\bR;\bfC) = 1 - \max_{t_1,\, t_2 \,\in\, [0,1]} F(t_1, t_2),
\]
where $F$ is a closed-form function of the model parameters (see~\eqref{eq:F-2x2-general}). This formula yields exact variational characterizations for both the one-sided matching rate $\match_{\OS}(B,\alpha,\alpha_f) \triangleq \match(B, 0;\bfC)$ and the two-sided matching rate $\match_{\TS}(B,\alpha,\alpha_f) \triangleq \match(B/2, B/2;\bfC)$, enabling direct comparison at the level of the asymptotic maximum matching size. Our approach extends the local weak convergence framework of~\cite{bordenave2013matchings} from single-type to multi-type trees, since the compatibility graph sequence $(G_n)$ converges locally weakly in probability to a $2$-type Poisson Galton--Watson tree encoding the flexibility levels of the nodes.

Our second contribution uses these variational formulas to map the dominance landscape between one-sided and two-sided allocations, significantly broadening the findings of~\cite{freund2024twov3}. For one-sided dominance, we prove that the one-sided allocation is optimal for the full budget $B = 1$ across the entire $(\alpha, \alpha_f)$ space, which significantly relaxes the conditions of prior work. 
For the case of no baseline connectivity ($\alpha = 0$), we show that the one-sided allocation continues to dominate for all sufficiently large $\alpha_f$.
Conversely, we establish two-sided dominance for all $\alpha > 0$ whenever $B \in (B^*,1)$ where $B^{\star} \triangleq 2 - 2e/3 \approx 0.188$ and $\alpha_f$ is sufficiently large. Furthermore, for any $B \in (0,1)$, we identify two-sided dominance regimes at both sufficiently small and large $\alpha$, with a small-$\alpha$ threshold that, unlike previous bounds, does not vanish as $B \to 0$. By reducing the dominance comparison to an explicit low-dimensional optimization problem, our approach yields unified proofs across the full parameter space, avoiding the regime-specific analyses required by
algorithmic bounds. We refer the reader to Remarks~\ref{rmk:full_budget}
and~\ref{rmk:Thm3} for a detailed comparison with the results
of~\cite{freund2024twov3}.
Simulations in~\Cref{sec:simulation} complement the theoretical findings. Mapping the full dominance boundary across parameter regimes, they demonstrate a monotonic shift with $B$: the two-sided allocation dominates almost everywhere at low budgets, whereas one-sided dominance progressively takes over until it covers the entire space at $B = 1$.

\paragraph{Organization.} The rest of the paper is organized as follows. Our main results on the matching rate formula and dominance regimes are stated in~\Cref{sec:matching-rate-formula} and~\Cref{sec:dominance-theorems} respectively, and related literature is reviewed in~\Cref{sec:related}. \Cref{sec:matching} develops the local weak convergence analysis and proves the variational formula for the matching rate. \Cref{sec:dominance} uses this formula to establish various theorems on one-sided versus two-sided dominance. Finally,~\Cref{sec:simulation} validates these results via simulation. A notation table and some supplementary proofs and discussions are provided in the appendices.

\subsection{Main Results} \label{sec:main-results}
 
\subsubsection{Matching rate formula} \label{sec:matching-rate-formula}
 
Our first contribution is a precise characterization of the asymptotic matching rate for the bipartite random graph model. The variational formula holds for any $2 \times 2$ connection matrix $\mathbf{C}$, and is not restricted to the specific flexibility parametrization introduced above.
 
\begin{theorem}\label{thm:matching-rate}
    Let $\mathbf{p} \triangleq (1-\bL, \bL)$,
    $\mathbf{q} \triangleq (1-\bR, \bR)$, and $\mathbf{C} = [c_{xy}]$. 
    As $n\to \infty$, the fraction of matched nodes in $G_n$ converges in probability as
    \begin{align}\label{eq:1-F}
        \frac{\matchnum(G_n)}{n}
        \, \xrightarrow[n\rightarrow \infty]{~\mathrm{p.}~} \, 
        1 - \max_{t_1,\, t_2 \,\in\, [0,1]} F(t_1, t_2),
    \end{align}
    where, with $M_y \triangleq \sum_{x=1}^2 c_{xy}\, p_x$,
    \begin{align}\label{eq:F-2x2-general}
        F(t_1, t_2)
        \triangleq
        \sum_{x=1}^{2} p_x
            \exp  \bigg({-\sum_{y=1}^{2} c_{xy}\, q_y\, e^{-M_y t_y}} \bigg)
        + \sum_{y=1}^{2} q_y\, e^{-M_y t_y}(1 + M_y t_y)
        - 1.
    \end{align}
\end{theorem}
 
\begin{proofsketch}
    The proof proceeds in two steps. First, we observe that the compatibility graph sequence $G_n$ converges locally weakly in probability to a $2$-type hierarchical Poisson Galton--Watson tree. While~\cite{bordenave2013matchings} characterizes the asymptotic matching rate for convergent graph sequences, their explicit variational formula is derived only for single-type trees. We generalize their approach to the multi-type setting by establishing, for each truncation level $d$, an explicit formula for the matching rate of a bounded-degree version of the model (\Cref{thm:bounded-degree}). This involves characterizing the root exposure probability via a largest solution to a recursive distributional equation. We then show that the resulting variational objective $F^{(d)}$ converges uniformly to $F$ as $d \to \infty$, recovering the matching rate for the original Poisson Galton--Watson tree.
\end{proofsketch}
 
The full proof of~\Cref{thm:matching-rate} is given in~\Cref{sec:matching}. An immediate consequence of the theorem is the convergence of the asymptotic matching rate: since $0 \leq \matchnum(G_n)/n \leq 1$, we have the following corollary by Dominated Convergence Theorem.

\begin{corollary}
    Let $F$ be as defined in~\eqref{eq:F-2x2-general}. The asymptotic matching rate satisfies
    \begin{align}
        \match(\bL,\bR; \bfC) = \lim_{n\to \infty} \frac{\mathbb{E}[\matchnum(G_n)]}{n} = 1 - \max_{t_1,t_2\in[0,1]} F(t_1,t_2) \,.
    \end{align}
\end{corollary}

Under the flexibility parametrization, we write $\match_{\OS}(B, \alpha, \alpha_f) \triangleq \match(B, 0;\bfC)$ and $\match_{\TS}(B, \alpha, \alpha_f) \triangleq \match(B/2, B/2;\bfC)$ for the asymptotic expected matching rates under the one-sided and two-sided allocations, respectively. The following two corollaries give explicit formulas for each.
 
\begin{corollary}\label{cor:matching_size_one}
    The matching rate for the one-sided allocation is
    \[
        \match_{\OS}(B, \alpha, \alpha_f)
        = 1 - \max_{t_1 \in [0,1]} F_{\OS}(t_1),
    \]
    where
    \begin{align}
        \begin{aligned}\label{eq:FOSdef}
        F_{\OS}(t_1)
        &= (1-B)\exp \left(-2\alpha\exp \left(-(2\alpha + (\alpha_f-\alpha)B)\,t_1\right)\right) \\
        &\quad + B\exp \left(-(\alpha+\alpha_f)\exp \left(-(2\alpha + (\alpha_f-\alpha)B)\,t_1\right)\right) \\
        &\quad + \exp \left(-(2\alpha + (\alpha_f-\alpha)B)\,t_1\right)
            \left(1 + (2\alpha + (\alpha_f-\alpha)B)\,t_1\right) - 1.
        \end{aligned}
    \end{align}
\end{corollary}
 
\begin{corollary}\label{cor:matching_size_two}
    The matching rate for the two-sided allocation is
    \[
        \match_{\TS}(B, \alpha, \alpha_f)
        = 1 - \max_{t_1, t_2 \in [0,1]} F_{\TS}(t_1, t_2),
    \]
    where
    \begin{align}
        \begin{aligned}\label{eq:FTSdef}
            F_{\TS}(t_1, t_2)
            &= \Big(1-\frac{B}{2}\Big)
                \exp \left(-2\alpha(1-B/2)\,e^{-M_1 t_1}
                      - (\alpha+\alpha_f)(B/2)\,e^{-M_2 t_2}\right) \\
            &\quad + \frac{B}{2}
                \exp \left(-(\alpha+\alpha_f)(1-B/2)\,e^{-M_1 t_1}
                      - \alpha_f B\,e^{-M_2 t_2}\right) \\
            &\quad + \Big(1-\frac{B}{2}\Big) e^{-M_1 t_1}(1+M_1 t_1)
                   + \frac{B}{2}\,e^{-M_2 t_2}(1+M_2 t_2) - 1,
        \end{aligned}
    \end{align}
    with $M_1 =  \left(2-\tfrac{B}{2} \right)\alpha + \tfrac{B}{2}\alpha_f$ and
    $M_2 =  \left(1-\tfrac{B}{2} \right)\alpha +  \left(1+\tfrac{B}{2} \right)\alpha_f$.
\end{corollary}

\subsubsection{Dominance regimes for flexibility allocation}
\label{sec:dominance-theorems}
 
Using Corollaries~\ref{cor:matching_size_one} and~\ref{cor:matching_size_two}, we characterize regimes in which the one-sided or the two-sided allocation dominates. Because both matching rates are given by explicit variational formulas, the comparison reduces to comparing the objectives $F_{\OS}$ and $F_{\TS}$. This is a tractable problem that allows us to study the dominance landscape across the parameter space $(\alpha, \alpha_f, B)$, where $\alpha$ is the baseline connection rate, $\alpha_f$ is the flexibility premium, and $B$ is the total flexibility budget. 
We begin with the case $B = 1$, where we show that the one-sided allocation always dominates.
 
\begin{theorem}[One-sided dominance for $B=1$]\label{thm:B_1}
    For any $\alpha, \alpha_f \ge 0$,
    \[
    \match_{\OS}(1, \alpha, \alpha_f) \ge \match_{\TS}(1, \alpha, \alpha_f)\,,
    \]
    with strict inequality whenever $\alpha \neq \alpha_f$.
\end{theorem}
 
\begin{remark}\label{rmk:full_budget}
    \Cref{thm:B_1} illustrates the tractability of the explicit matching-rate
    formula. For $B = 1$, it yields a short algebraic proof of one-sided
    dominance across the entire $(\alpha, \alpha_f)$ parameter space. This
    strengthens the corresponding results in~\cite{freund2024twov3}, where
    one-sided dominance is established for $\alpha = 0$
    in~\cite[Theorem~2]{freund2024twov3}, and for a subcritical region with
    $10^{-4} < \alpha < 0.77\alpha_f - 0.16$ and $\alpha + \alpha_f < e$
    in~\cite[Theorem~9]{freund2024twov3}, the latter via a computer-aided proof.
\end{remark}
 
For $B < 1$, the one-sided allocation no longer dominates uniformly. Nevertheless, one-sided dominance persists for large $\alpha_f$ when $\alpha = 0$.
 
\begin{theorem}[One-sided dominance for $\alpha = 0$]\label{thm:alpha_0}
    Let $\alpha = 0$. For any $B \in (0,1)$, there exists
    $\overline{\alpha}_f(B) > 0$ such that for all
    $\alpha_f > \overline{\alpha}_f(B)$,
    \[
        \match_{\OS}(B, 0, \alpha_f) > \match_{\TS}(B, 0, \alpha_f).
    \]
\end{theorem}
 
When $B < 1$ and $\alpha > 0$, the asymmetry between the two sides can reverse this picture. The following two theorems establish conditions under which the two-sided allocation dominates for large $\alpha_f$: for all $\alpha > 0$ when $B \ge B^{\star}$, and at both extremes of $\alpha$ for all $B \in (0,1)$.
 
\begin{theorem}[Two-sided dominance for $B \in [ B^{\star}, 1)$]\label{thm:global}
    Let $B^{\star} \triangleq 2 - 2e/3$. For any $B \in [B^{\star}, 1)$ and
    $\alpha > 0$, there exists $\overline{\alpha}_f(\alpha, B) > 0$ such
    that for all $\alpha_f > \overline{\alpha}_f(\alpha, B)$,
    \[
        \match_{\TS}(B, \alpha, \alpha_f) > \match_{\OS}(B, \alpha, \alpha_f).
    \]
\end{theorem}
 
For $B < B^{\star}$, the comparison at intermediate values of $\alpha$ remains open, but two-sided dominance can be established at both extremes.

\begin{theorem}[Two-sided dominance for small or large $\alpha$]
\label{thm:small_large_alpha}
    Fix $B \in (0,1)$.
    \begin{enumerate}[label=$\mathrm{(\roman*)}$]
        \item \textup{(Small-$\alpha$ regime.)}
            There exists $\alpha_{\mathrm{low}}(B)> 0$ such that for any
            $0< \alpha < \alpha_{\mathrm{low}}(B)$, there exists
            $\overline{\alpha}_f(\alpha, B) > 0$ such that for all
            $\alpha_f > \overline{\alpha}_f(\alpha, B)$,
            \[
                \match_{\TS}(B, \alpha, \alpha_f)
                > \match_{\OS}(B, \alpha, \alpha_f).
            \]
        \item \textup{(Large-$\alpha$ regime.)}
            There exists $\alpha(B) > 0$ such that for any
            $\alpha > \alpha(B)$, there exists
            $\overline{\alpha}_f(\alpha, B) > 0$ such that for all
            $\alpha_f > \overline{\alpha}_f(\alpha, B)$,
            \[
                \match_{\TS}(B, \alpha, \alpha_f)
                > \match_{\OS}(B, \alpha, \alpha_f).
            \]
    \end{enumerate}
\end{theorem}
 
Together, parts~(i) and~(ii) of \Cref{thm:small_large_alpha} show that for any $B \in (0,1)$, two-sided dominance holds for sufficiently large $\alpha_f$ whenever $\alpha$ is either sufficiently small or sufficiently large. For $B < B^{\star}$, the intermediate regime $\alpha \in [\alpha_{\mathrm{low}}(B),\, \alpha(B)]$, if non-empty, remains analytically open.
 
\begin{remark}\label{rmk:Thm3}
\Cref{thm:small_large_alpha}(i) parallels \cite[Theorem~3]{freund2024twov3}, which establishes two-sided dominance for $\alpha < \alpha_\star(B)$ and $\alpha_f > \alpha^f_\star(B,\alpha)$ (see \prettyref{eq:fmz}). Their theorem provides explicit expressions for $\alpha^f_\star(B,\alpha)$ and $\alpha_\star(B)$. For completeness, we detail in \prettyref{app:comp} how these results can be recovered directly from the variational formula. Furthermore, our approach extends this admissible regime. While the threshold $\alpha_\star(B)$ vanishes as $B \to 0$, whereas  $\alpha_{\mathrm{low}}(B)$ in Theorem~\ref{thm:small_large_alpha}(i) remains bounded away from zero uniformly in $B$. 
\end{remark}

\subsection{Related Literature} \label{sec:related}
 
\paragraph{Flexibility design.}
Our paper contributes to the operations literature on flexibility design, where a central question is how limited flexibility should be allocated to improve throughput and robustness. Classical work studies this through structural design choices: the chaining principle and sparse process-flexibility in manufacturing~\cite{jordan1995principles,ChouChuaTeoZheng2010,bassamboo2010flexibility,ChenMaZhangZhou2019}, cross-training in workforce management~\cite{wallace2004resource,IravaniVanOyenSims2005}, and resource pooling in service systems~\cite{KesavanStaatsGilland2014,ChodMarkakisTrichakis2021}. A shared insight running through this work is that the \emph{pattern} of flexibility, and not just its total amount, shapes system performance. In our setting, flexibility is not encoded as a deterministic compatibility structure chosen edge by edge; instead, it is allocated across the two sides of a marketplace, modulating the probabilities with which compatibility edges form.
 
Two papers study flexibility allocation directly in random matching markets. In a companion work~\cite{ameen2026uniformity}, flexibility is modeled spatially: supply nodes embedded in a metric space are assigned service ranges ex~ante, and compatibility edges form when demand falls within a supply node's range. That paper establishes a uniformity principle: among all supply-side allocations with the same total flexibility budget, a more uniform distribution of service ranges yields a larger expected maximum matching. The present paper addresses a complementary question: rather than how to distribute flexibility \emph{within} one side, we ask how to split a fixed flexibility budget \emph{between} the two sides. Our primary reference is~\cite{freund2024twov3}, which identifies two competing effects governing this split (flexibility cannibalization and flexibility asymmetry), through an analysis of the Karp--Sipser algorithm, and establishes parameter regimes in which each effect dominates.

\paragraph{Matchings on random graphs.}
A large literature studies maximum matchings in sparse random graphs, beginning with the Karp--Sipser (KS) algorithm and its asymptotic analysis in Erd\H{o}s--R\'enyi (ER) graphs~\cite{karp1981maximum,aronson1998maximum} and extending to belief-propagation techniques from statistical physics~\cite{mezard2009information,zdeborova2006number}. Our compatibility graph is naturally viewed as a sparse bipartite stochastic block model~\cite{holland1983stochastic,abbe2018community}, with block memberships determined by the flexibility types of supply and demand nodes. While the KS algorithm exactly finds a maximum matching in the subcritical regime~\cite{karp1981maximum,aronson1998maximum}, it leaves an uncharacterized residual in the supercritical regime even for ER graphs. Dominance comparisons based on KS therefore rely on bounds rather than exact matching rates~\cite{freund2024twov3}, a limitation that our variational formula overcomes directly.

\paragraph{Local weak convergence.}
Our main technical tool is the framework of local weak convergence~\cite{benjamini2001recurrence,aldous2004objective,aldous2001zeta,vanderhofstad2016random}, which replaces the finite graph by its random tree limit and characterizes macroscopic quantities through recursions on that limit tree. For maximum matchings, this program is carried out by~\cite{bordenave2013matchings}, who derive the asymptotic matching size via a recursive distributional equation on single-type unimodular Galton--Watson trees. We adapt that framework to the bipartite stochastic block model, whose limiting tree is a multi-type Poisson Galton--Watson tree. The single-type recursion of~\cite{bordenave2013matchings} generalizes to the type-dependent operator $\Theta$ (\Cref{def:Theta}), whose fixed points yield our variational formula. Local weak convergence framework has found applications ranging from epidemic forecasting on networks~\cite{alimohammadi2023epidemic} to dense subgraphs for load balancing~\cite{anantharam2016densest} and planted spanning trees~\cite{moharrami2025planted}.

\paragraph{Message passing (belief propagation).}
Closely related to our work is the body of literature on message passing, or belief propagation: on tree graphs belief propagation is exact, and on sparse locally tree-like graphs it yields asymptotically correct fixed-point equations~\cite{mezard2009information,mezard2002analytic}. Belief propagation has also been applied in artificial intelligence~\cite{pearl2014probabilistic,pearl2022fusion} and computer vision~\cite{freeman2000learning}. In operations, message-passing and belief-propagation techniques have been used to study information deadlocks in matching markets with costly compatibility inspections~\cite{immorlica2022matching}, the role of signaling and interviews in random matching markets~\cite{allman2025signaling} and community detection in stochastic block models~\cite{abbe2018community}.

\section{Analysis for the matching rate formula}
\label{sec:matching}

To avoid degenerate notation, we prove the formula under the assumption that all node types have positive mean degree. The degenerate cases follow by replacing $C$ with $C+\varepsilon$ (elementwise), applying the nondegenerate result, and letting $\varepsilon\to 0 $; the normalized matching size changes by at most the number of added edges divided by $n$, while the variational objective converges uniformly. 
Our analysis adapts the local weak convergence framework of~\cite{bordenave2013matchings} from single-type to multi-type unimodular graphs, covering in particular bipartite stochastic block models. We begin by identifying the local weak limit of $(G_n)$.
 
\begin{proposition}[Local weak limit]
\label{prop:BSBM-LWC}
    The sequence $G_n$ of compatibility graphs converges locally weakly in probability to a $2$-type hierarchical Poisson Galton--Watson tree $\mathcal{T}$, defined as follows for each $x, y \in \{1,2\}$.
    \begin{itemize}
        \item The root is a supply node (resp.\ demand node) with probability $1/2$. Conditional on being supply (resp.\ demand), it is of type~$1$ with probability $p_1$ (resp.\ $q_1$).
        \item A supply node of type $x$ has an independent $\mathrm{Pois}(c_{xy}\,q_y)$ number of demand children of type~$y$.
        \item A demand node of type $y$ has an independent $\mathrm{Pois}(c_{xy}\,p_x)$ number of supply children of type~$x$.
    \end{itemize}
\end{proposition}
 
\begin{proof}
    This follows directly from~\cite[Theorem~3.14]{VdH-IRG}.
\end{proof}
 
\paragraph{Proof strategy.}
The key tool linking maximum matchings to the local weak limit is a recursion that requires bounded node degrees. We therefore proceed in two steps: (i)~we establish a variational formula for the matching rate in a $2$-type unimodular bounded-degree Galton--Watson tree (see~\Cref{def:UHGW-2type} and~\Cref{thm:bounded-degree}); (ii)~we recover~\Cref{thm:matching-rate} by approximating the Poisson model with a sequence of such bounded-support models via a truncation argument. Supporting lemmas introduced in this section are proved in Appendix~\ref{app:matching}.
 
\paragraph{Notation.}
For a graph $G$ and node $u$, we write $G - u$ for the graph obtained by deleting $u$ and all its incident edges, and $\mathrm{deg}_G(u)$ for the degree of $u$ in $G$. For an integer $d \ge 1$, let $\mathcal{G}_d$ denote the class of all graphs with maximum degree at most $d$. We adopt the conventions $0^{-1} = \infty$ and $\infty^{-1} = 0$, and write $\mathcal{P}([0,1])$ for the set of probability measures on $[0,1]$.
 
For a distribution $\pi$ on the non-negative integers, its \emph{excess distribution} $\widehat\pi$ is defined by
\[
    \widehat\pi(k)
    = \frac{(k+1)\,\pi(k+1)}{\sum_i i\,\pi(i)},
    \qquad k = 0, 1, \ldots
\]
For $x, y \in \{1,2\}$, we denote by $\piL_x$ and $\piR_y$ the degree distributions of supply nodes of type~$x$ and demand nodes of type~$y$, with PGFs $\phi_x$ and $\psi_y$, respectively. The corresponding excess degree distributions are $\pihatL_x$ and $\pihatR_y$, with PGFs satisfying
\begin{align}\label{eq:pgf-hats}
    \widehat\phi_x(s)
    \triangleq \mathbb{E}_{X\sim\pihatL_x} \left[s^X\right]
    = \frac{\phi_x'(s)}{\phi_x'(1)},
    \qquad
    \widehat\psi_y(s)
    \triangleq \mathbb{E}_{X\sim\pihatR_y} \left[s^X\right]
    = \frac{\psi_y'(s)}{\psi_y'(1)},
\end{align}
where $\phi_x'$ and $\psi'_y$ denote the first derivatives of $\phi_x$ and $\psi_y$ respectively. A summary of all notation is provided in Table~\ref{tab:notation} in Appendix~\ref{apx-table}.

\subsection{Proof of Theorem~\ref{thm:matching-rate}} \label{sec:proof-main}

We show how the variational formula extends from bounded-degree graph sequences to the Poisson setting, following the truncation machinery of~\cite[Section~4.2]{bordenave2013matchings}. For a compatibility graph $G_n = (V_n, E_n)$, define its $d$-truncation $G_n^d$ by isolating every vertex of degree greater than $d$, i.e.\ removing all edges incident to such vertices. Since isolating a vertex reduces the matching size by at most one,
\begin{equation}\label{eq:boundgap}
    \left|\frac{\matchnum(G_n)}{n} - \frac{\matchnum(G_n^d)}{n}\right|
    \le \frac{|\{v \in V_n : \mathrm{deg}(v) > d\}|}{n}.
\end{equation}
To identify the local weak limit of $(G_n^d)$ and apply the appropriate matching rate formula, we introduce the following bounded-degree tree model.

\begin{definition}[$2$-type unimodular bounded-degree Galton--Watson tree]
\label{def:UHGW-2type}
    Fix type proportions $\mathbf{p} = (p_1, p_2)$ for supply and $\mathbf{q} = (q_1, q_2)$ for demand. For each supply type $x \in \{1,2\}$, let $\piL_x$ be an integer-valued distribution with bounded support and PGF $\phi_x$; for each demand type $y \in \{1,2\}$, let $\piR_y$ be an integer-valued distribution with bounded support and PGF $\psi_y$. Fix type-mixing proportions $(\aL_{x1}, \aL_{x2})$ and $(\aR_{y1}, \aR_{y2})$ satisfying the unimodularity constraint
    \begin{align}\label{eq:unimodularity-constraint}
        p_x\,\aL_{xy}\,\phi_x'(1)
        = q_y\,\aR_{yx}\,\psi_y'(1)
        \qquad \text{for all } x, y \in \{1,2\},
    \end{align}
    where $\phi_x'(1) = \mathbb{E}_{D\sim\piL_x}[D]$ and
    $\psi_y'(1) = \mathbb{E}_{D\sim\piR_y}[D]$.
    The associated rooted $2$-type hierarchical Galton--Watson tree $T$, is constructed as follows.
    \begin{itemize}
        \item The root $\circ$ is a supply node with probability $1/2$ and a demand node with probability $1/2$; conditional on being supply (resp.\ demand), it is of type~$1$ with probability $p_1$ (resp.\ $q_1$).
        \item A root supply node of type $x$ (resp.\ demand node of type $y$) produces a number of offspring drawn independently from $\piL_x$ (resp.\ $\piR_y$). Each offspring is independently a demand node of type~$y$ (resp.\ supply node of type~$x$) with probability $\aL_{xy}$ (resp.\ $\aR_{yx}$).
        \item For non-root nodes, the offspring count is drawn from the excess distribution $\pihatL_x$ (resp.\ $\pihatR_y$) in place of $\piL_x$ (resp.\ $\piR_y$), with offspring types assigned identically.
    \end{itemize}
\end{definition}

When $\pi$ is Poisson with rate $\lambda$, its excess distribution $\widehat\pi$ is again Poisson with the same rate, and Definition~\ref{def:UHGW-2type} recovers the tree $\mathcal{T}$ of Proposition~\ref{prop:BSBM-LWC}. For bounded degree distributions, the following theorem establishes the following variational formula, whose proof is given in Section~\ref{sec:bounded-degree}.

\begin{theorem}\label{thm:bounded-degree}
    Consider the $2$-type bounded-degree hierarchical Galton--Watson tree $T$ from Definition~\ref{def:UHGW-2type}, and suppose that for each $x, y \in \{1,2\}$, the support of $\piL_x$ and $\piR_y$ is contained in $\{0, \ldots, d\}$. For any sequence $G_n^d \in \mathcal{G}_d$ of graphs that converges locally weakly in probability to $T$, 
    \begin{align}\label{eq:matched-bdd}
         \frac{ \matchnum(G_n^d) } {n}
        \, \xrightarrow[]{~\mathrm{p.}~} \, 1 - \max_{t_1,\,t_2 \in [0,1]} F(t_1, t_2),
    \end{align}
    as $n\to\infty$, where
    \begin{align}\label{eq:definition-of-F-general}
        F(t_1, t_2)
        \triangleq
        \sum_{x=1}^2 p_x\,\phi_x \bigg(
            \sum_{y=1}^2 \aL_{xy} \bigg(
                1 - \frac{\psi_y'(1-t_y)}{\psi_y'(1)}
            \bigg)
        \bigg)
        + \sum_{y=1}^2 q_y\left(\psi_y(1-t_y) + t_y\,\psi_y'(1-t_y)\right)
        - 1,
    \end{align}
    with $\phi_x$ and $\psi_y$ the PGFs of $\piL_x$ and $\piR_y$,
    respectively.
\end{theorem}

By~\Cref{prop:BSBM-LWC}, $(G_n)$ converges locally weakly in probability to the tree $\mathcal{T}$ with
\[
    \piL_x = \mathrm{Pois}(\lambda_x), \quad
    \piR_y = \mathrm{Pois}(M_y), \quad
    \lambda_x \triangleq \sum_{y=1}^2 c_{xy}\,q_y, \quad
    M_y \triangleq \sum_{x=1}^2 c_{xy}\,p_x,
\]
and limiting neighbor-type mixing probabilities
$\aL_{xy}=c_{xy}q_y/\lambda_x$ and $\aR_{yx}=c_{xy}p_x/M_y$. We first identify the bounded-degree limit.

For a vertex reached through one already present edge, define the endpoint survival probabilities
\[
    \rho_x^{\mathrm{L},d} \triangleq \P(\Pois(\lambda_x)\le d-1),
    \qquad
    \rho_y^{\mathrm{R},d} \triangleq \P(\Pois(M_y)\le d-1).
\]
For a supply root of type $x$, set
\[
    \Lambda_x^{(d)} \triangleq \sum_{y=1}^2 c_{xy}q_y\rho_y^{\mathrm{R},d},
    \qquad
    \overline\Lambda_x^{(d)} \triangleq \lambda_x-\Lambda_x^{(d)}.
\]
Let $K_x^{\mathrm{L},d}\sim\Pois(\Lambda_x^{(d)})$ and
$U_x^{\mathrm{L},d}\sim\Pois(\overline\Lambda_x^{(d)})$ be independent, and define
\[
    D_x^{\mathrm{L},d}
    \triangleq
    K_x^{\mathrm{L},d}\,\mathbf{1}\{K_x^{\mathrm{L},d}+U_x^{\mathrm{L},d}\le d\}.
\]
Thus $K_x^{\mathrm{L},d}$ counts neighbors whose other endpoint survives the truncation, whereas $U_x^{\mathrm{L},d}$ counts neighbors killed because the other endpoint has original degree larger than $d$.  Let
\(
    \phi_x^{(d)}(s) \triangleq \E\big[s^{D_x^{\mathrm{L},d}}\big].
\)
Equivalently,
\[
    \phi_x^{(d)}(s)
    = \P(K_x^{\mathrm{L},d}+U_x^{\mathrm{L},d}>d)
      + \sum_{k=0}^{d} s^k\,
        \P(K_x^{\mathrm{L},d}=k)\,
        \P(U_x^{\mathrm{L},d}\le d-k).
\]
Similarly, for a demand root of type $y$, set
\[
    \Gamma_y^{(d)} \triangleq \sum_{x=1}^2 c_{xy}p_x\rho_x^{\mathrm{L},d},
    \qquad
    \overline\Gamma_y^{(d)} \triangleq M_y-\Gamma_y^{(d)},
\]
let $K_y^{\mathrm{R},d}\sim\Pois(\Gamma_y^{(d)})$ and
$U_y^{\mathrm{R},d}\sim\Pois(\overline\Gamma_y^{(d)})$ be independent, define
\[
    D_y^{\mathrm{R},d}
    \triangleq
    K_y^{\mathrm{R},d}\,\mathbf{1}\{K_y^{\mathrm{R},d}+U_y^{\mathrm{R},d}\le d\},
    \qquad
    \psi_y^{(d)}(s) \triangleq \E\big[s^{D_y^{\mathrm{R},d}}\big].
\]
The corresponding type-mixing probabilities are
\[
    a_{xy}^{\mathrm{L},d}
    \triangleq
    \frac{c_{xy}q_y\rho_y^{\mathrm{R},d}}{\Lambda_x^{(d)}},
    \qquad
    a_{yx}^{\mathrm{R},d}
    \triangleq
    \frac{c_{xy}p_x\rho_x^{\mathrm{L},d}}{\Gamma_y^{(d)}} \, .
\]
Conditional on the total retained degree, the retained neighbor types are multinomial with these probabilities, since the truncation depends only on the total original degree.  The means satisfy
\[
    (\phi_x^{(d)})'(1)=\Lambda_x^{(d)}\rho_x^{\mathrm{L},d},
    \qquad
    (\psi_y^{(d)})'(1)=\Gamma_y^{(d)}\rho_y^{\mathrm{R},d},
\]
and hence unimodularity holds, i.e.
\(
    p_x\,a_{xy}^{\mathrm{L},d}(\phi_x^{(d)})'(1)
    = p_xq_yc_{xy}\rho_x^{\mathrm{L},d}\rho_y^{\mathrm{R},d}
    = q_y\,a_{yx}^{\mathrm{R},d}(\psi_y^{(d)})'(1).
\)

For each fixed $d$, $(G_n^d)$ converges locally weakly in probability to the $2$-type unimodular Galton--Watson tree of Definition~\ref{def:UHGW-2type} with root PGFs $\phi_x^{(d)}$ and $\psi_y^{(d)}$ and mixing probabilities $a_{xy}^{\mathrm{L},d}$ and $a_{yx}^{\mathrm{R},d}$.  To see why the excess laws are the correct ones, consider for example a retained edge leading to a demand vertex of type $y$.  Conditional on this edge being retained, the other incident edges of that demand vertex are retained with total count distributed as
\[
    K_y^{\mathrm{R},d}
    \, \big|\,
    K_y^{\mathrm{R},d}+U_y^{\mathrm{R},d}\le d-1,
\]
whose PGF is exactly
\(
    (\psi_y^{(d)})'(s) / {(\psi_y^{(d)})'(1)}.
\)
The same argument applies on the supply side.  This verifies that the non-root offspring laws are the excess distributions required in Definition~\ref{def:UHGW-2type}. 
Since $G_n^d\in\mathcal G_d$, Theorem~\ref{thm:bounded-degree} therefore yields
\[
    \frac{\matchnum(G_n^d)}{n}
    \xrightarrow[]{~\mathrm{p.}~}
    a_d
    \triangleq
    1 - \max_{(t_1,t_2)\in[0,1]^2} F^{(d)}(t_1,t_2),
\]
where $F^{(d)}$ denotes the objective in~\eqref{eq:definition-of-F-general} instantiated with
$\phi_x^{(d)},\psi_y^{(d)},a_{xy}^{\mathrm{L},d}$.

It remains to pass $d\to\infty$.  Since $\rho_x^{\mathrm{L},d}\to 1$ and $\rho_y^{\mathrm{R},d}\to 1$, we have 
\[
    \Lambda_x^{(d)}\to\lambda_x\, , 
    ~~~~
    \Gamma_y^{(d)}\to M_y \, ,
    ~~~~ 
    a_{xy}^{\mathrm{L},d}\to c_{xy}q_y/\lambda_x \, ,
    ~~~~
    a_{yx}^{\mathrm{R},d}\to c_{xy}p_x/M_y \, .
\]
%
Thus, $\phi_x^{(d)}$ and $\psi_y^{(d)}$, together with their first two derivatives converge uniformly on $[0,1]$ to
$\exp(\lambda_x(s-1))$ and $\exp(M_y(s-1))$ respectively. Consequently, $F^{(d)}$ converges uniformly on $[0,1]^2$ to the limiting function $F^{\mathrm{lim}}$ obtained by instantiating~\eqref{eq:definition-of-F-general} with these Poisson PGFs and limiting mixing probabilities.  Hence,
\begin{align}\label{eq:Flim}
    \lim_{d\to\infty}
    \max_{(t_1,t_2)\in[0,1]^2} F^{(d)}(t_1,t_2)
    = \max_{(t_1,t_2)\in[0,1]^2} F^{\mathrm{lim}}(t_1,t_2).
\end{align}
Fix $\varepsilon>0$, and define
\[
    a \triangleq 1 - \max_{(t_1,t_2)\in[0,1]^2} F^{\mathrm{lim}}(t_1,t_2).
\]
By~\eqref{eq:Flim}, choose $d$ large enough that $|a_d-a|<\varepsilon/3$.  Then, by~\eqref{eq:boundgap},
\begin{align}
    \mathbb{P} \left(\left|\frac{ \matchnum(G_n) }{n} - a\right| > \varepsilon\right)
    &\le
    \mathbb{P} \left(\left|\frac{ \matchnum(G_n) }{n} - a_d\right| > \frac{2\varepsilon}{3}\right) \notag\\
    & \le
    \mathbb{P} \left(\left|\frac{ \matchnum(G_n) }{n} - \frac{\matchnum(G_n^d)}{n}\right| > \frac{\varepsilon}{3}\right)
    +
    \mathbb{P} \left(\left|\frac{ \matchnum(G_n^d) }{n} - a_d\right| > \frac{\varepsilon}{3}\right) \notag\\
    &\le
    \mathbb{P} \left(\frac{|\{v \in V_n : \deg(v) > d\}|}{n} > \frac{\varepsilon}{3}\right)
    +
    \mathbb{P} \left(\left|\frac{ \matchnum(G_n^d) }{n} - a_d\right| > \frac{\varepsilon}{3}\right).
    \label{eq:boundgap2}
\end{align}
The second term in~\eqref{eq:boundgap2} converges to $0$ as $n\to\infty$ by Theorem~\ref{thm:bounded-degree}.  For the first term, the quantity
\(
    {|\{v \in V_n : \deg(v) > d\}|}/{n}
\)
is a bounded local statistic, and therefore converges in probability to
\[
    \tau_d \triangleq \sum_{x=1}^2 p_x\,\mathbb{P}(\mathrm{Pois}(\lambda_x)>d)
    + \sum_{y=1}^2 q_y\,\mathbb{P}(\mathrm{Pois}(M_y)>d).
\]
Since all degree distributions are Poisson, $\tau_d \to 0$ as $d\to\infty$, so by enlarging $d$ if needed we may assume $\tau_d < \varepsilon/3$.  Hence the first term in~\eqref{eq:boundgap2} also converges to $0$, and therefore
\[
    \frac{ \matchnum(G_n) }{n}
    \xrightarrow[]{~\mathrm{p.}~}
    1 - \max_{(t_1,t_2)\in[0,1]^2} F^{\mathrm{lim}}(t_1,t_2).
\]
Finally, for Poisson PGFs,
\[
    \frac{\psi_y'(1-t_y)}{\psi_y'(1)}=e^{-M_y t_y},
    \qquad
    \psi_y(1-t_y)+t_y\psi_y'(1-t_y)=e^{-M_y t_y}(1+M_y t_y),
\]
and
\[
    \phi_x\left(\sum_{y=1}^2 \frac{c_{xy}q_y}{\lambda_x}(1-e^{-M_y t_y})\right)
    =
    \exp\left(-\sum_{y=1}^2 c_{xy}q_y e^{-M_y t_y}\right).
\]
Thus $F^{\mathrm{lim}}$ coincides with $F$ in Theorem~\ref{thm:matching-rate}, completing the proof.

\subsection{Proof of~\texorpdfstring{\Cref{thm:bounded-degree}}{}} \label{sec:bounded-degree}

We start by introducing the operator $\Theta$, which is used throughout the proof.

\begin{definition}[The operator $\Theta$]\label{def:Theta}
    Let $\boldsymbol{\mu} = (\mu_1, \mu_2) \in \mathcal{P}([0,1])^2$ be a pair of probability measures.
    Let $\boldsymbol{\Lambda} = (\Lambda_1, \Lambda_2)$ and $\boldsymbol{\Lambda}' = (\Lambda'_1, \Lambda'_2)$ be pairs of integer-valued random variables. 
    For each index $x \in \{1,2\}$, we define $\Theta_{x;\,\boldsymbol{\Lambda},\boldsymbol{\Lambda}'}(\boldsymbol{\mu})$ as the law of the random variable $Y_x$:
    \begin{equation}\label{eq:Thetadef}
        Y_x \, = \, \bigg( 1 + \sum_{i=1}^{N_x} \bigg( \sum_{j=1}^{N'_{i}} X_{ij} \bigg)^{-1} \, \bigg)^{-1},
    \end{equation}
    where the variables are constructed according to the following hierarchical process:
    \begin{itemize}
        \item $N_x$ is a random variable distributed according to $\Lambda_x$.
        \item For each $i \in \{1, \dots, N_x\}$, the variable $N'_i$ is independently drawn from a mixture distribution where an intermediate index $y \in \{1,2\}$ is selected with probability $\aL_{xy}$, and conditional on $y$, the variable $N'_i$ is independently distributed as $\Lambda'_y$.

        \item For each $j \in \{1, \dots, N'_i\}$, the term $X_{ij}$ is independently drawn from the mixture distribution $\sum_{x'} \aR_{yx'} \, \mu_{x'}$.
    \end{itemize}

    The operator $\Theta$ is the mapping on $\mathcal{P}([0,1])^2$ given by
    \[
        \Theta_{\boldsymbol{\Lambda},\boldsymbol{\Lambda}'}(\boldsymbol{\mu})
        \,\triangleq\,
        \left(
            \Theta_{1;\,\boldsymbol{\Lambda},\boldsymbol{\Lambda}'}(\boldsymbol{\mu}),\,
            \Theta_{2;\,\boldsymbol{\Lambda},\boldsymbol{\Lambda}'}(\boldsymbol{\mu})
        \right).
    \]
\end{definition}

We note that $\Theta_{\boldsymbol{\Lambda},\boldsymbol{\Lambda}'}$ also depends on the type-mixing matrices $[\aL_{xy}]$ and $[\aR_{yx}]$ from Definition~\ref{def:UHGW-2type}; this dependence is suppressed in the notation as these quantities are fixed throughout.

\begin{proposition}\label{prop:BLS-11}
    Consider a sequence of finite graphs $G_n^d \in \mathcal{G}_d$ converging locally weakly in probability to the tree $T$  of~\Cref{thm:bounded-degree}. Then,
    \begin{align} \label{eq:FPE-to-matching-fraction}
       \frac{\matchnum(G_n^d)}{n}  \, \xrightarrow[]{~\mathrm{p.}~} \, 1 - \mathbb{E} \left[\mathcal{R}_{\star}\right]
        ~~~~ \mbox{ as } n \to \infty \, ,
    \end{align}
    where $\mathcal{R}_{\star}$ is drawn from the mixture $\sum_{x}p_x \, \mu_x$ with $\boldsymbol{\mu} = \Theta_{\boldsymbol{\piL},\boldsymbol{ \pihatR }}(\boldsymbol{ \mu^\star})$, and $\boldsymbol{\mu^\star}$ is a largest solution to the random recursive distributional equation $\boldsymbol{{\mu^\star}}= \Theta_{\boldsymbol{ \pihatL }, \boldsymbol{ \pihatR}}(\boldsymbol{{\mu^\star}})$\footnote{The ``largest" solution is with respect to stochastic dominance on the space of probability measures $\mathcal{P}([0,1])^2$. While this space does not have a total order in general, this result guarantees the existence of a greatest solution. }.
\end{proposition}
 \begin{proof}
    The proof follows from \cite[Theorems 10 and 11]{bordenave2013matchings}. 
\end{proof}

By~\Cref{prop:BLS-11}, the rest of the proof  is dedicated to obtaining a precise characterization of  $\mu^{\star}$. To this end, for any  $(\mu_1,\mu_2) \in \mathcal{P}([0,1])^2$, 
let $(t_1, t_2)$ denote its induced positivity vector where 
\[
    t_y \triangleq \sum_{x=1}^{2} \aR_{yx}\, \mu_x\left( (0,1] \right) \, ,
\]
for $y \in \{1,2\}$. 
We first present the following lemma that establishes a relationship between the fixed points of $\Theta_{\boldsymbol{\pihatL}, \boldsymbol{\pihatR}}$ and the function $F$ defined in~\eqref{eq:definition-of-F-general}.

\begin{lemma}\label{lmm:root-rep-formula}
    Let $\boldsymbol{\widehat{\mu}} \triangleq (\widehat\mu_1, \widehat\mu_2)$ be any fixed point of $\Theta_{\boldsymbol{ \pihatL},\boldsymbol{ \pihatR }}$, with induced positivity vector $(\hat t_1, \hat t_2)$. Then
    \[
        \sum_{x=1}^2 p_x  \, \mathbb{E}\left[\Theta_{x; \boldsymbol{ \piL},\boldsymbol{ \pihatR }}(\boldsymbol{\widehat{\mu}})\right] =F(\hat t_1, \hat t_2)
    \]
    where $F:[0,1]^2\to\mathbb{R}$ is defined in~\eqref{eq:definition-of-F-general}.
\end{lemma}

Next, we identify the positivity vector of the relevant fixed point as the maximizer of $F$.

\begin{lemma} \label{lmm:variational-characterization-new}
    The induced positivity vector $(t_1^{\star},t_2^{\star})$ of the largest fixed point $\boldsymbol{\mu}^\star$ of  $\Theta_{\boldsymbol{ \pihatL },\boldsymbol{ \pihatR }}$ satisfies
    \[
        F(t_1^{\star}, t_2^{\star}) = \max_{t_1,t_2 \in [0,1]} \, F(t_1, t_2) \, ,
    \]
    where $F:[0,1]^2\to\mathbb{R}$ is defined in~\eqref{eq:definition-of-F-general}.
\end{lemma}

Putting it all together, we have that the matching rate is given by
\begin{align} \label{eq:nu-at-t-star}
    \lim_{n\rightarrow \infty}\frac{\matchnum(G_n)}{n} \stepa{={}} 1 - \sum_{x=1}^2 p_x  \, \mathbb{E}\left[\Theta_{x; \boldsymbol{ \piL},\boldsymbol{ \pihatR }}(\boldsymbol{\mu}^{\star})\right]
     \stepb{=}  
    1 - F(t_1^{\star}, t_2^{\star})  
     \stepc{=}  
    1 - \max_{t_1,t_2} F(t_1,t_2)\, ,
\end{align}
where (a) uses~\Cref{prop:BLS-11}, (b) uses~\Cref{lmm:root-rep-formula} and (c) uses~\Cref{lmm:variational-characterization-new}. This concludes the proof. We are left to prove \prettyref{lmm:root-rep-formula} and \prettyref{lmm:variational-characterization-new}.

\subsubsection{Proof of~\texorpdfstring{\Cref{lmm:root-rep-formula}}{}}

Consider any type $x \in \{1,2\}$ for a root supply node, and a distribution $\boldsymbol{\widehat{\mu}} = (\widehat\mu_1,\widehat\mu_2)$. Define the random variable $Z_x$ as 
\begin{align} \label{eq:Z_x}
    Z_x\triangleq \sum_{i=1}^{N_x} \frac{1}{S_i}, ~~~~ \mbox{ where }  S_i = \sum_{j=1}^{N_i^{\mathrm{R}}}X_{ij}
\end{align}
and the random variables $N_x$, $N_i^{\mathrm{R}}$ and $X_{ij}$ are obtained as follows. First, $N_x \sim \piL_x$ is the number of demand offspring of the root supply node. Next, for each demand offspring $i$, its type is independently chosen as $y$ with probability $\aL_{xy}$. Conditional on $y$, its offspring count is chosen as $N_i^{\mathrm R} \sim \pihatR_y$ and each $X_{ij}$ is independently drawn from the mixture $\sum_x \aR_{yx} \, \widehat{\mu}_x$. Comparing against~\Cref{def:Theta} with $\boldsymbol{\Lambda} \equiv \boldsymbol{\piL}$ and $\boldsymbol{\Lambda}' \equiv \boldsymbol{\pihatR}$ with the conventions $0^{-1} = \infty$ and $\infty^{-1} = 0$, we have
\[
        \Theta_{x; \boldsymbol{\piL}, \boldsymbol{\pihatR}}(\boldsymbol{\widehat{\mu}}) \, \stackrel{\mathrm{d.}}{=} \, \frac{1}{1+Z_x} \, .
\]
Using the identity 
\[
    (1+Z)^{-1}  = 
    \begin{cases}
        1 - Z/(1+Z), & Z<\infty \\
        0, & Z = \infty
    \end{cases} 
\]
along with the definition of $Z_x$, we have
\begin{align} 
    \E\left[ \Theta_{x; \boldsymbol{\piL}, \boldsymbol{\pihatR}} (\boldsymbol{\widehat{\mu}})\right] 
    &= 
    \E\left[ \frac{1}{1+Z_x}\right] 
    =
    \E \left[ \left( 1 - \frac{Z_x}{1+Z_x} \right) \cdot  \mathbf{1}_{ \{ Z_x < \infty \} } \,+\, 0 \cdot \mathbf{1}_{ \{ Z_x = \infty\} } \right]
    \nonumber \\[3pt]
    & =
    \P(Z_x < \infty) - \E\left[ \sum_{i=1}^{N_x} \frac{S_i^{-1}}{1+Z_x} \cdot \mathbf{1}_{ \{ Z_x < \infty \} }\right] \, . \label{eq:two-term-decomposition}
\end{align}
where the equality holds by setting $Z_x = \sum_{i=1}^{N_x} S_i^{-1}$ in the numerator. We now rewrite 
\begin{align} 
    \frac{S_i^{-1}}{1+Z_x} 
    & =  \frac{S_i^{-1}}{1+S_i^{-1}+\sum_{i'=1, i'\neq i}^{N_x} S_{i'}^{-1}} 
    = \frac{1}{1+S_i\left[1+\sum_{i'=1, i'\neq i}^{N_x} S_{i'}^{-1} \right]}  = \frac{Y_{x,i}}{Y_{x,i}+S_i}\, , 
    \label{eq:S_i_inverse_Z_x}
\end{align}%
where the last equality holds by defining 
\[
    Y_{x,i} \triangleq\frac{1}{1+\sum_{i'=1, i'\neq i}^{N_x} S_{i'}^{-1}} \,.
\]
Substituting~\eqref{eq:S_i_inverse_Z_x} in~\eqref{eq:two-term-decomposition}, we obtain
\begin{align*} 
    \E\left[ \Theta_{x; \boldsymbol{\piL}, \boldsymbol{\pihatR}}(\boldsymbol{\widehat{\mu}}) \right] 
    & =
    \P(Z_x < \infty) - \E\left[ \sum_{i=1}^{N_x} \frac{Y_{x,i}}{Y_{x,i}+S_i}\cdot \mathbf{1}_{ \{ Z_x < \infty \} }\right] \, .
\end{align*}
Finally, we have from~\eqref{eq:Z_x} that $Z_x < \infty$ iff $S_i>0$ for all $i \leq N_x$. Therefore,
\begin{align} \label{eq:ThetaxwithZ}
    \E\left[ \Theta_{x; \boldsymbol{\piL}, \boldsymbol{\pihatR}}(\boldsymbol{\widehat{\mu}})\right] 
    & =
    \P(Z_x < \infty) - \E\left[ \sum_{i=1}^{N_x} \frac{Y_{x,i}}{Y_{x,i}+S_i}\cdot \mathbf{1}_{ \{ S_{1}, \ldots, S_{N_x}>0  \} }\right] \, .
\end{align}

We now present the following lemma about each of the two terms in~\eqref{eq:ThetaxwithZ}. In what follows, $S^{y}$ is the random variable $S_1$ conditioned on its type being $y$.

\begin{lemma} \label{lmm:Term-1-2}
    Let $X_x \sim \widehat\mu_x$ and $r_x \triangleq \sum_{y=1}^2 \aL_{xy} \left( 1 - \psihat_y(1-t_y)\right)$. Then,
    \begin{align} 
        \P(Z_x < \infty) &= \phi_x(r_x) \, \label{eq:Term-1}
        \\
        \E\left[  \sum_{i=1}^{N_x}  \frac{Y_{x,i}}{Y_{x,i} + S_i} \cdot \mathbf{1}_{\{ S_{1}, \ldots, S_{N_x} > 0 \}}\right]  &= \phi'_x(1)  \sum_{y=1}^2 \aL_{xy}\, \E \left[ \frac{X_x}{X_x+S^{y}} \cdot \mathbf{1}_{\{S^{y}>0\}} \right] \, .
        \label{eq:Term-2}
    \end{align}
\end{lemma}

Combining~\Cref{lmm:Term-1-2} with \eqref{eq:ThetaxwithZ}, we get:
\[
    \sum_{x=1}^2 p_x  \,  \mathbb{E}\left[\Theta_{x; \boldsymbol{\piL}, \boldsymbol{\pihatR}}(\boldsymbol{\widehat{\mu}})\right] =
    \sum_{x=1}^2 p_x  \, \phi_x(r_x)-\underbrace{\sum_{y=1}^2\sum_{x=1}^2 p_x \, \phi'_x(1) \, \aL_{xy} \, \E \left[ \frac{X_x}{X_x+S^y} \cdot \mathbf{1}_{\{S^y>0\}} \right]}_{ \mathrm{(ii)} } \, .
\]
Injecting the definition of $r_x$, we have:
\[
    \sum_{x=1}^2 p_x \, \phi_x(r_x)= \sum_{x=1}^2 p_x\,\phi_x \bigg(\sum_{y=1}^2 \aL_{xy} \bigg( 1- \frac{\psi'_y(1-\hat{t}_y)}{\psi'_y(1)}\bigg)\bigg),
\]
which matches the first term in the definition of $F(\hat{t}_1,\hat{t}_2)$. 
    
It remains to show that the term (ii) above matches the second term in the definition of $F(\hat{t}_1,\hat{t}_2)$. Using the unimodularity constraint~\eqref{eq:unimodularity-constraint}, i.e. $p_x\,\aL_{xy} \, \phi'_x(1) = q_y \,\aR_{yx}\, \psi_y'(1)$ for all $x,y \in \{1,2\}$, we have
\begin{align} \label{eq:U_D-formula-2}
    \mathrm{(ii)} =  \sum_{y=1}^2 q_y   \, \psi'_y(1)  \sum_{x=1}^2 \aR_{yx} \, \E\left[ \frac{X_x}{X_x + S^y} \cdot  \mathbf{1}_{\{ S^y > 0\}} \right].
\end{align}
The following lemma now gives an explicit expression for the above quantity.
\begin{lemma}\label{clm:formula-for-Ky}
    Consider a fixed point $(\widehat\mu_1,\widehat\mu_2)$ of $\Theta$ with positivity vector $(\hat{t}_1,\hat{t}_2)$. For each $x\in \{1, 2\}$, generate the independent random variable $X_x \sim \widehat\mu_x$.
    Then,
    \begin{align}\label{eq:formula-for-Ky}
     \psi'_y(1) \cdot  \sum_{x=1}^2 \aR_{yx} \, \E\left[ \frac{X_x}{X_x + S^y} \cdot  \mathbf{1}_{\{ S^y > 0\}} \right]  \, = \,  1-\psi_y(1-\hat t_y)- \hat t_y\,\psi_y'(1-\hat t_y) \, .
    \end{align}
\end{lemma}
Applying~\eqref{eq:formula-for-Ky} to~\eqref{eq:U_D-formula-2}, the second term in our decomposition matches the second term in $F(\hat{t}_1,\hat{t}_2)$. This completes the proof.

\subsubsection{Proof of~\texorpdfstring{\Cref{lmm:variational-characterization-new}}{}}

For $\boldsymbol\mu\in\mathcal P([0,1])^2$ and $y\in\{1,2\}$, recall that the positivity vector of $\boldsymbol\mu$ is defined as $t_y(\boldsymbol\mu)\triangleq \sum_{x=1}^2 \aR_{yx}\mu_x((0,1])$. Further, define
\(
    \mathcal R(\boldsymbol\mu)
    \triangleq
    \sum_{x=1}^2 p_x\,
    \mathbb E\!\left[
        \Theta_{x;\boldsymbol{\piL},\boldsymbol{\pihatR}}(\boldsymbol\mu)
    \right] \, .
\)
By~\Cref{lmm:root-rep-formula}, if $\boldsymbol\mu$ is a fixed point of
$\Theta_{\boldsymbol{\pihatL},\boldsymbol{\pihatR}}$, then
\(
    \mathcal R(\boldsymbol\mu)=F(t_1(\boldsymbol\mu),t_2(\boldsymbol\mu)).
\)
Finally, consider the positivity map $H:[0,1]^2\to[0,1]^2$, defined by
\[
    H_y(t_1,t_2)
    \triangleq
    \sum_{x=1}^2 \aR_{yx}\,
    \widehat\phi_x(r_x(t_1,t_2)),
    \qquad
    r_x(t_1,t_2)
    \triangleq
    \sum_{y=1}^2 \aL_{xy}\bigl(1-\widehat\psi_y(1-t_y)\bigr).
\]
A direct differentiation of the function $F$ gives
\begin{align} \label{eq:derivative-formula}
    \frac{\partial F}{\partial t_y}(t_1,t_2)
    =
    q_y\psi_y''(1-t_y)\bigl(H_y(t_1,t_2)-t_y\bigr).
\end{align}
We use the following two elementary facts, proved in Appendix~\ref{app:matching}. 

\begin{lemma}\label{clm:max}
    There exists a maximizer $(\tmax_1,\tmax_2)$ of $F$ satisfying
    $H(\tmax_1,\tmax_2)=(\tmax_1,\tmax_2)$.
\end{lemma}

Fix such a maximizer and set for $x \in \{1,2\}$,
\(
    \smax_x\triangleq \widehat\phi_x(r_x(\tmax_1,\tmax_2)).
\)
Then, initialize
\[
    \boldsymbol\mu^{(0)}
    \triangleq
    \bigl(\operatorname{Bernoulli}(\smax_1),
          \operatorname{Bernoulli}(\smax_2)\bigr),
    \qquad
    \boldsymbol\mu^{(k)}
    \triangleq
    \Theta_{\boldsymbol{\pihatL},\boldsymbol{\pihatR}}^k
    (\boldsymbol\mu^{(0)}).
\]

\begin{lemma}\label{clm:const}
    For every $k\ge0$, $\mu_x^{(k)}((0,1])=\smax_x$ for $x\in\{1,2\}$, and hence
    \[
        (t_1^{(k)}, t_2^{(k)}) \triangleq (t_1(\boldsymbol\mu^{(k)}),t_2(\boldsymbol\mu^{(k)}))
        =
        (\tmax_1,\tmax_2).
    \]
\end{lemma}
Among all laws $\nu$ on $[0,1]$ such that $\nu((0,1]) = \smax_x$, the Bernoulli law $\mu_x^{(0)}$ defined above is stochastically largest, since it concentrates all positive mass at the maximal value $1$. Therefore,
$\boldsymbol\mu^{(0)}\succeq\boldsymbol\mu^{(1)}$. By monotonicity of
$\Theta_{\boldsymbol{\pihatL},\boldsymbol{\pihatR}}$,
\[
    \boldsymbol\mu^{(0)}
    \succeq
    \boldsymbol\mu^{(1)}
    \succeq
    \boldsymbol\mu^{(2)}
    \succeq\cdots .
\]
Therefore $\boldsymbol\mu^{(k)}$ converges weakly to some
$\boldsymbol\mu^{(\infty)}\in\mathcal P([0,1])^2$. Since the degree distributions are bounded,
$\Theta_{\boldsymbol{\pihatL},\boldsymbol{\pihatR}}$ is weakly continuous, and so
\(
    \boldsymbol\mu^{(\infty)}
    =
    \Theta_{\boldsymbol{\pihatL},\boldsymbol{\pihatR}}
    (\boldsymbol\mu^{(\infty)}).
\)

It remains to lower-bound $\mathcal R(\boldsymbol\mu^{(\infty)})$. To that end, we use the following comparison
inequality: If $\boldsymbol\nu\in\mathcal P([0,1])^2$ has positivity vector $t=(t_1,t_2)$ and
\(
    \widetilde{\boldsymbol\nu}
    \triangleq
    \Theta_{\boldsymbol{\pihatL},\boldsymbol{\pihatR}}(\boldsymbol\nu)
    \preceq
    \boldsymbol\nu,
\)
then
\begin{equation}\tag{CI}\label{eq:comparison-ineq}
    \mathcal R(\boldsymbol\nu)\ge F(t_1,t_2).
\end{equation}
Indeed, let
\[
    \xi_y\triangleq \sum_{x=1}^2 \aR_{yx}\nu_x,
    \qquad
    \widetilde\xi_y\triangleq \sum_{x=1}^2 \aR_{yx}\widetilde\nu_x .
\]
Repeating the decomposition in the proof of~\Cref{lmm:root-rep-formula} gives
\[
    \mathcal R(\boldsymbol\nu)
    =
    \sum_{x=1}^2 p_x\phi_x(r_x(t))
    -
    \sum_{y=1}^2 q_y\psi_y'(1)\,
    \mathbb E\!\left[
        \frac{\widetilde X_y}{\widetilde X_y+S_y}
        \mathbf 1_{\{S_y>0\}}
    \right],
\]
where $\widetilde X_y\sim\widetilde\xi_y$ and
$S_y=\sum_{j=1}^{\widehat N_y}X_{y,j}$, with
$\widehat N_y\sim\pihatR_y$ and $X_{y,j}\stackrel{\rm iid}{\sim}\xi_y$.
Since $\widetilde{\boldsymbol\nu}\preceq\boldsymbol\nu$, we have
$\widetilde\xi_y\preceq\xi_y$, and hence the expectation above is at most the same expression
with $\widetilde X_y$ replaced by $X_y\sim\xi_y$. The exchangeability calculation from
\Cref{lmm:root-rep-formula} then gives
\[
    \psi_y'(1) \cdot 
    \mathbb E\!\left[
        \frac{X_y}{X_y+S_y}\mathbf 1_{\{S_y>0\}}
    \right]
    =
    1-\psi_y(1-t_y)-t_y\psi_y'(1-t_y),
\]
which proves~\eqref{eq:comparison-ineq}.
Applying~\eqref{eq:comparison-ineq} with
$\boldsymbol\nu=\boldsymbol\mu^{(k)}$, and using
$\boldsymbol\mu^{(k+1)}\preceq\boldsymbol\mu^{(k)}$ with \Cref{clm:const}, yields
\[
    \mathcal R(\boldsymbol\mu^{(k)})
    \ge
    F(\tmax_1,\tmax_2)
    \qquad\text{for all }k\ge0.
\]
The root operator $\Theta_{\boldsymbol{\piL},\boldsymbol{\pihatR}}$ is weakly continuous, so
$\mathcal R(\boldsymbol\mu^{(k)})\to\mathcal R(\boldsymbol\mu^{(\infty)})$. Thus
\(
    \mathcal R(\boldsymbol\mu^{(\infty)})
    \ge
    F(\tmax_1,\tmax_2).
\)

Finally, let $\boldsymbol\mu^\star$ be the largest fixed point of
$\Theta_{\boldsymbol{\pihatL},\boldsymbol{\pihatR}}$, and let
$(t_1^\star,t_2^\star)$ be its positivity vector. Since
$\boldsymbol\mu^\star\succeq\boldsymbol\mu^{(\infty)}$ and
$\Theta_{\boldsymbol{\piL},\boldsymbol{\pihatR}}$ is increasing,
\[
    F(t_1^\star,t_2^\star)
    =
    \mathcal R(\boldsymbol\mu^\star)
    \ge
    \mathcal R(\boldsymbol\mu^{(\infty)})
    \ge
    F(\tmax_1,\tmax_2),
\]
where the first equality uses~\Cref{lmm:root-rep-formula}. Since
$(\tmax_1,\tmax_2)$ maximizes $F$, the reverse inequality
$F(t_1^\star,t_2^\star)\le F(\tmax_1,\tmax_2)$ is immediate. Hence equality holds, proving
\Cref{lmm:variational-characterization-new}.

\section{Analysis for dominance regimes for flexibility allocation} \label{sec:dominance}

In this section, we use~\Cref{thm:matching-rate} to prove the theorems of~\Cref{sec:dominance-theorems}, i.e. results on dominance of one-sided versus two-sided allocation for various regimes.

\subsection{Analysis for one-sided dominance}

\subsubsection{Proof of \prettyref{thm:B_1}}
Define $\beta \triangleq (\alpha + \alpha_f)/2$. Setting $B=1$ in~\Cref{cor:matching_size_one}, the function $F_{\OS}$ in~\eqref{eq:FOSdef} simplifies to
\begin{align*}
    F_{\OS}(t_1)
    = \exp\left(-2\beta\, e^{-2\beta t_1}\right)
      + e^{-2\beta t_1}(1 + 2\beta t_1) - 1 \, ,
\end{align*}
for $t_1 \in [0,1]$. Further, setting $B = 1$ in~\Cref{cor:matching_size_two}, the constants $M_1$ and $M_2$ reduce to
\[
    M_1 = \frac{3\alpha + \alpha_f}{2} \, , 
    ~~~~
    M_2 = \frac{\alpha + 3\alpha_f}{2} \, .
\]
Thus, setting $u_1 \triangleq e^{-M_1 t_1}$ and $u_2 \triangleq e^{-M_2 t_2}$, the function $F_{\TS}$ in~\eqref{eq:FTSdef} becomes
\begin{align}
    F_{\TS}(t_1, t_2)
    = \frac{1}{2}\,e^{-\alpha u_1 - \beta u_2}
      + \frac{1}{2}\,e^{-\beta u_1 - \alpha_f u_2}
      - 1 + \frac{1}{2}(u_1 + u_2)
      + \frac{1}{2}(M_1 t_1 u_1 + M_2 t_2 u_2) \, ,
    \label{eq:FTS}
\end{align}
for $(t_1, t_2) \in [0,1]^2$. We start by establishing that
\begin{align}\label{eq:domainextension}
    \max_{t_1, t_2 \in [0,1]} F_{\TS}(t_1, t_2)
    \,=\,
    \sup_{t_1, t_2 \ge 0} F_{\TS}(t_1, t_2).
\end{align}
Differentiating~\eqref{eq:FTS} with respect to $t_1$ (using $\partial u_1/\partial t_1 = -M_1 u_1$) gives
\[
    \frac{\partial F_{\TS}}{\partial t_1}
    = \frac{M_1 u_1}{2}
      \left(
        \alpha\, e^{-\alpha u_1 - \beta u_2}
        + \beta\, e^{-\beta u_1 - \alpha_f u_2}
        - M_1 t_1
      \right) \, .
\]
Since $e^{-z} \le 1$ for all $z \ge 0$, we have
\[
    \alpha\, e^{-\alpha u_1 - \beta u_2}
    + \beta\, e^{-\beta u_1 - \alpha_f u_2}
    \le \alpha + \beta
    = \frac{3\alpha + \alpha_f}{2} 
    = M_1 \, ,
\]
where the last equality uses the definition $\beta = (\alpha+\alpha_f)/2$.
Therefore, for $t_1 \ge 1$,
\[
    \frac{\partial F_{\TS}}{\partial t_1}
    \,\le\, \frac{M_1 u_1}{2}(M_1 - M_1 t_1)
    \,=\, \frac{M_1^2 u_1}{2}(1 - t_1)
    \,\le\, 0.
\]
By an identical argument, $\partial F_{\TS}/\partial t_2 \le 0$ for all $t_2 \ge 1$. Hence $F_{\TS}$ is non-increasing in each coordinate outside $[0,1]$, so we may extend the optimization domain from $[0,1]^2$ to $[0,\infty)^2$ without loss. This establishes~\eqref{eq:domainextension}.

Next, let $t^{\star}$ denote a maximizer of $F_{\OS}$ over $[0,1]$, and set $U \triangleq e^{-2\beta t^{\star}}$. Define the coupled point
\begin{equation}\label{eq:coupled}
    t_1^\dagger \,\triangleq\, \frac{2\beta\, t^{\star}}{M_1},
    \qquad
    t_2^\dagger \,\triangleq\, \frac{2\beta\, t^{\star}}{M_2},
    \qquad \text{so that} \quad u_1 = u_2 = U.
\end{equation}
Note that $t_1^\dagger, t_2^\dagger \ge 0$; they need not lie in $[0,1]$, but this is immaterial since Step~1 allows us to work with the supremum over $[0,\infty)^2$.
Evaluating~\eqref{eq:FTS} at $(t_1^\dagger, t_2^\dagger)$ and using $u_1 = u_2 = U$ (and noting $M_i t_i^\dagger = 2\beta t^{\star}$ for $i=1,2$), we have
\begin{align}
    F_{\TS}(t_1^\dagger, t_2^\dagger)
    = \frac{1}{2}\,e^{-(\alpha+\beta)U}
      + \frac{1}{2}\,e^{-(\beta+\alpha_f)U}
      - 1 + U + 2\beta t^{\star} U \, .
    \label{eq:FTSeval}
\end{align}
We apply the AM-GM inequality\footnote{(AM-GM inequality) For two non-negative reals $a$ and $b$, we have $\frac 1 2(a+b) \geq \sqrt{ab}$, with equality iff $a=b$.} to the two exponential terms in~\eqref{eq:FTSeval}. 
Since $(\alpha+\beta)+(\beta+\alpha_f) = 4\beta$, AM-GM applied to $a = e^{-(\alpha+\beta)U}$ and $b = e^{-(\beta+\alpha_f)U}$ gives
\[
    \frac{1}{2} \, \left(
      e^{-(\alpha+\beta)U} + e^{-(\beta+\alpha_f)U}
    \right)
    \,\ge\,
    \sqrt{\,e^{-(\alpha+\beta)U}\cdot e^{-(\beta+\alpha_f)U}\,}
    \,=\,
    \sqrt{e^{-4\beta U}}
    \,=\, e^{-2\beta U} \, .
\]
Substituting this bound into~\eqref{eq:FTSeval} and recalling that $F_{\OS}(t^{\star}) = e^{-2\beta U} - 1 + U + 2\beta t^{\star} U$, we obtain
\begin{align}
    F_{\TS}(t_1^\dagger, t_2^\dagger)
    \,\ge\, e^{-2\beta U} - 1 + U + 2\beta t^{\star} U
    \,=\, F_{\OS}(t^{\star}) \, .
    \label{eq:bridge}
\end{align}
In summary, chaining the results of both steps yields
\[
    \max_{t_1,\, t_2 \in [0,1]} F_{\TS}
    \, = \, \sup_{t_1,\, t_2 \ge 0} F_{\TS}(t_1, t_2)
    \, \ge \, F_{\TS}(t_1^\dagger, t_2^\dagger)
    \, \ge \, F_{\OS}(t^{\star})
    \, = \, \max_{t \in [0,1]} F_{\OS}(t) \, .
\]
The first equality follows from~\eqref{eq:domainextension}, the last equality holds by definition of $t^{\star}$. Since $\match_{\OS} = 1 - \max F_{\OS}$ and $\match_{\TS} = 1 - \max F_{\TS}$, we conclude $\match_{\OS}(1,\alpha,\alpha_f) \ge \match_{\TS}(1,\alpha,\alpha_f)$.

The AM-GM inequality is strict iff \ $e^{-(\alpha+\beta)U} \neq e^{-(\beta+\alpha_f)U}$, which holds (since $U > 0$) iff $\alpha \neq \alpha_f$. Hence the above chain of inequalities is strict whenever $\alpha \neq \alpha_f$. When $\alpha = \alpha_f$, all entries of the connection matrix $\bfC$ equal $2\alpha$, so the compatibility graph $G_n$ has the same distribution regardless of the flexibility allocation, i.e. $\match_{\OS}(1,\alpha,\alpha) = \match_{\TS}(1,\alpha,\alpha)$.

\subsubsection{Proof of~\texorpdfstring{\Cref{thm:alpha_0}}{}}
Substituting $\alpha = 0$ in the expressions of~\Cref{cor:matching_size_one} and~\Cref{cor:matching_size_two}, the objective functions become
\begin{align} \label{eq:F-OS-alpha-0}
    F_{\OS}(t_1) 
    &= B \exp  \left(   -\alpha_f \,  e^{-\alpha_f B t_1} \right) 
       + e^{-\alpha_f B t_1} (1 + \alpha_f B t_1) - B \, ,
\end{align}
and
\begin{align}
    \begin{aligned} \label{eq:F-TS-alpha-0}
    F_{\TS}(t_1, t_2) 
    &= \left(1-\tfrac{B}{2}\right)
       \exp   \left(   -\tfrac{\alpha_f B}{2} \cdot 
       e^{-\alpha_f(1+B/2)t_2}\right) \\[0.5 em]
    &\quad + \tfrac{B}{2}
       \exp   \left(   -\alpha_f \left(1-\tfrac{B}{2}\right) \cdot 
       e^{- \alpha_f B t_1 / 2}
       -\alpha_f B\,e^{-\alpha_f(1+B/2)t_2}\right) \\[0.5em]
    &\quad + \left(1-\tfrac{B}{2}\right)  
       e^{- \alpha_f B t_1 / 2}  
       \left(1+\tfrac{\alpha_f B}{2}t_1\right)  + \tfrac{B}{2}\,
       e^{-\alpha_f(1+B/2)t_2} \cdot 
       \left(1+\alpha_f \left(1+\tfrac{B}{2} \right ) \,  t_2 \right)
       -  1  .
    \end{aligned}
\end{align}

First, we upper bound $F_{\OS}$. Let $y = e^{-\alpha_f B t_1}$. Since $t_1\in[0,1]$, it follows that $y\in[e^{-\alpha_f B},1] \subset (0,1]$.  Define the function $g(y)$ from $F_{\OS}(t_1)$ via this change of variable, i.e.
\begin{align}
        g(y) = Be^{-\alpha_f y} + y(1-\ln y) - B \, . \label{eq:g_y} 
\end{align}
We make the following claim.
\begin{claim}\label{clm:maximizer}
    For sufficiently large $\alpha_f$, the global maximizer $y^{\star}$ of $g(y)$ on $[e^{-\alpha_f B},1]$ satisfies 
    \[
        1 > y^{\star} > c_B \geq e^{-\alpha_f B} \, ,
    \]
    where $c_B\in(0,1)$ is the unique solution to $x(1-\ln x)=1-B$.
\end{claim}
Claim~\ref{clm:maximizer} implies that $y^{\star}$ is an interior point satisfying $g'(y^{\star})=0$, i.e. $-\alpha_f B e^{-\alpha_f y^{\star}}-\ln(y^{\star}) = 0$. We therefore have
\[
    y^{\star} = \exp\left(   -\alpha_f Be^{-\alpha_f y^{\star}}\right)
    \stepa{\geq} \exp\left(-\alpha_f Be^{-\alpha_f c_B}\right)
    \stepb{\geq} 1-\alpha_f Be^{-\alpha_f c_B} \, ,
\]
where (a) uses $y^{\star}>c_B$ and (b) uses $e^{-z}\geq 1-z$. Applying $y(1-\ln y)\leq 1$ for $y\in(0,1]$ yields
\[
    g(y^{\star}) \leq 1 - B + Be^{-\alpha_f y^{\star}}
    \leq 1 - B + B \exp\left(  -\alpha_f(1-\alpha_f Be^{-\alpha_f c_B}) \right) \, .
\]
For sufficiently large $\alpha_f$, using $e^z\leq 1+2z$ for $z\in[0,\ln 2]$ yields the upper bound
\begin{align} \label{eq:UB-gy}
    g(y^{\star}) \leq 1-B + Be^{-\alpha_f}
    \left(1+2\alpha_f^2 \,B \, e^{-\alpha_f \, c_B} \right).
\end{align} 
   
\noindent Next, we lower bound $F_{\TS}$, by evaluating $F_{\TS}$ at $(t_1,t_2)=(0,1)$. We have from~\eqref{eq:F-TS-alpha-0} that
\begin{align}
    \max_{t_1,t_2\in[0,1]}F_{\TS} 
    & \geq F_{\TS}(0,1) \nonumber
    \\
    & = \left(1-\tfrac{B}{2}\right) \cdot 
       \exp \left(-\tfrac{\alpha_f B}{2}\,
       e^{-\alpha_f(1+B/2)}\right) + \tfrac{B}{2}
       \exp\left(-\alpha_f \left(1-\tfrac{B}{2}\right)
       -\alpha_f B\,e^{-\alpha_f(1+B/2)}\right) \nonumber \\[0.4 em]
    & ~~~~~~ - \tfrac{B}{2}
       + \tfrac{B}{2}\,e^{-\alpha_f(1+B/2)}
       \left(1+\alpha_f \left(1+\tfrac{B}{2}\right)\right)\nonumber \\[0.6 em]
    & \begin{aligned} \label{eq:LB-FTS}
    & \geq 1-B
      -\left(1-\tfrac{B}{2}\right) \left(\tfrac{\alpha_f B}{2}
      e^{-\alpha_f(1+B/2)}\right) + \tfrac{B}{2}e^{-\alpha_f(1-B/2)}
      \left(1-\alpha_f Be^{-\alpha_f(1+B/2)}\right) \\[0.4em]
    & ~~~~~~
      +\tfrac{B}{2} \, e^{-\alpha_f(1+B/2)} \, 
      \left(1+\alpha_f \left(1+\tfrac{B}{2}\right)\right) \, ,
      \end{aligned} 
\end{align}
where the second inequality holds by applying the inequality $e^{-x}\geq 1-x$. 

When taking the difference between~\eqref{eq:LB-FTS} and~\eqref{eq:UB-gy}, the $1-B$ term cancels.
Further, since $B\in(0,1)$, the exponent $-\alpha_f(1-B/2)$ is strictly greater than $-\alpha_f$, $-\alpha_f(1+B/2)$, and $-\alpha_f(1+c_B)$. Therefore, for $\alpha_f$ sufficiently large, the positive term $\frac{B}{2}e^{-\alpha_f(1-B/2)}$ strictly dominates all remaining terms, which are of order at most $e^{-\alpha_f}$ or $e^{-\alpha_f(1+c_B)}$, giving $\max F_{\TS}-\max F_{\OS}>0$ for all sufficiently large $\alpha_f$. Since $\match=1-\max F$, this gives $\match_{\OS}(B,0,\alpha_f)> \match_{\TS}(B,0,\alpha_f)$.

\begin{proof}[Proof of~\texorpdfstring{\Cref{clm:maximizer}}{}]
    Set $\alpha_f\geq  {-\ln (c_B)}/{B}$, which guarantees $e^{-\alpha_f B}\leq c_B$, so $c_B$ lies within the domain. By~\eqref{eq:g_y}, $g(y)\leq y(1-\ln y)$ for all $y\in(0,1]$. At the upper boundary, $g(1)=1-B+Be^{-\alpha_f}$. Since $y^{\star}$ is the global maximizer, $g(y^{\star})\geq g(1)$.
    
    Let $x^{\star}$ be the unique solution to $x(1-\ln x)=1-B+Be^{-\alpha_f}$. Since $Be^{-\alpha_f}>0$, we have $x^{\star}(1-\ln x^{\star})>c_B(1-\ln c_B)$, and by strict monotonicity of $y(1-\ln y)$ on $(0,1]$, $x^{\star}>c_B$. If $y^{\star}<x^{\star}$, then $g(y^{\star})\leq y^{\star}(1-\ln y^{\star})<x^{\star}(1-\ln x^{\star})=g(1)$, contradicting $g(y^{\star})\geq g(1)$. Hence $y^{\star}\geq x^{\star}>c_B$.
    
    To see $y^{\star}<1$: $g'(y)=-\alpha_f Be^{-\alpha_f y}-\ln y$, so $g'(1)=-\alpha_f Be^{-\alpha_f}<0$, meaning $g$ is strictly decreasing at $y=1$. Hence the maximizer cannot be at $y=1$, giving $y^{\star}<1$.
\end{proof}

\subsection{Analysis of two-sided dominance}

Our proofs in this subsection utilize the following formulas for the limiting matching rate as $\alpha_f \to \infty$ -- they are proved in Appendix~\ref{app:limiting}.

\begin{proposition}\label{prop:oneside}
    Let $B\in(0,1)$ and $\alpha\geq 0$. The one-sided matching rate satisfies
    \[
        \lim_{\alpha_f \rightarrow \infty}
        \match_{\OS}(B, \alpha, \alpha_f) 
        = 1 - \max_{y\in(0,1]} \Phi_{\OS}(y) \, ,
    \]
    where
    \begin{align}\label{eq:Phi_os}
        \Phi_{\OS}(y) 
        =
        (1-B)e^{-2\alpha y} + y(1-\ln y) - 1 \, .
    \end{align}
\end{proposition}

\begin{proposition}\label{prop:twoside}
    Let $B\in(0,1)$ and $\alpha\geq 0$. The two-sided matching rate satisfies 
    \[
        \lim_{\alpha_f \rightarrow \infty}\match_{\TS}(B, \alpha,\alpha_f) 
        = 1 - \max\Big(0,\,\max_{y\in(0,1]} \Phi_{\TS}(y) \Big),
    \]
    where
    \begin{align} \label{eq:Phi_ts}
       \Phi_{\TS}(y) 
       = \Big(1-\frac{B}{2}\Big)
         \left[e^{-2\alpha(1-B/2)y} + y(1-\ln y)\right] - 1 \, .
    \end{align}
\end{proposition}

By~\Cref{prop:oneside} and~\Cref{prop:twoside}, the asymptotic unmatched fractions as $\alpha_f \to \infty$ are
\begin{align}
    U_{\OS}(\alpha) 
    &\,\triangleq \, \lim_{\alpha_f\to\infty}
             \left(1-\match_{\OS}(B,\alpha,\alpha_f) \right)
    \,=\, \max_{y\in(0,1]}\Phi_{\OS}(y) \, ,  \label{eq:M_os}\\
    U_{\TS}(\alpha) 
    &\,\triangleq \, \lim_{\alpha_f\to\infty}
             \left(1-\match_{\TS}(B,\alpha,\alpha_f) \right)
    \,=\, \max \Big(0,\,\max_{y\in(0,1]}\Phi_{\TS}(y)\Big) \label{eq:L_ts} \, ,
\end{align}
where $\Phi_{\OS}$ and $\Phi_{\TS}$ are as defined in~\Cref{prop:oneside} and~\Cref{prop:twoside} respectively.
We now prove Theorem~\ref{thm:global} and~\Cref{thm:small_large_alpha}.

\subsubsection{Proof of~\texorpdfstring{\Cref{thm:global}}{}}

We begin with the following claim.

\begin{claim}\label{clm:global}
    For any $B\in(0,1)$, there exists $\alpha(B)>0$ such that
    $U_{\TS}(\alpha)=0$ for all $\alpha\geq\alpha(B)$.
    If $B\geq B^*=2-2e/3$, then $\alpha(B)\leq\frac{e}{2-B}$.
\end{claim}

By Claim~\ref{clm:global}, for $\alpha\geq\alpha(B)$, $U_{\TS}(\alpha)=0$.
By \eqref{eq:M_os} and \eqref{eq:L_ts}, \footnote{This demonstrates
that in the $\alpha_f\to\infty$ regime with sufficiently large
$\alpha$, two-sided allocation achieves an almost perfect matching,
while one-sided allocation leaves a constant fraction unmatched.}
\[
    U_{\OS}(\alpha)
    \, \geq \, \Phi_{\OS}(1)
    \, = \,  (1-B)e^{-2\alpha}
    \, > \,  0 = U_{\TS}(\alpha).
\]
It thus suffices to prove $U_{\OS}(\alpha)>U_{\TS}(\alpha)$ for all
$\alpha\in(0,\alpha(B))$. 
Let $y_2^\star$ be the global maximizer of $\Phi_{\TS}$.
By Claim~\ref{clm:global}, $\alpha(B)\leq\frac{e}{2-B}$ for
$B\geq B^*$, so $(2-B)\alpha<e$ for all $\alpha\in(0,\alpha(B))$.
Hence $\Phi_{\TS}$ has a unique critical point satisfying
$\Phi_{\TS}'(y_2^\star)=0$:
\[
    -\ln y_2^\star
    = (2-B)\alpha\,e^{-(2-B)\alpha y_2^\star}.
\]
Indeed, $w\mapsto we^w$ is strictly increasing from $[0,1]$ onto
$[0,e]$, and $(2-B)\alpha\in(0,e)$, so there is a unique
$z\in(0,1)$ with $ze^z=(2-B)\alpha$, and $y_2^\star=e^{-z}$
satisfies the equation:
\[
    (2-B)\alpha\,e^{-(2-B)\alpha e^{-z}}
    = ze^z\cdot e^{-ze^z\cdot e^{-z}}
    = ze^z\cdot e^{-z}
    = z
    = -\ln(e^{-z}).
\]
Since $U_{\OS}(\alpha)\geq\Phi_{\OS}(y_2^\star)$:
\[
    U_{\OS}(\alpha)-U_{\TS}(\alpha)
    \,\geq\, \Phi_{\OS}(y_2^\star)-\Phi_{\TS}(y_2^\star).
\]
Using $y_2^\star=e^{-z}$, $(2-B)\alpha e^{-z}=z$, and $-\ln e^{-z}=z$:
\begin{align*}
    \Phi_{\OS}(e^{-z})
    &= (1-B)e^{-\frac{2}{2-B}z}+e^{-z}(1+z)-1,\\
    \Phi_{\TS}(e^{-z})
    &= \Big(1-\frac{B}{2}\Big)\bigl[e^{-z}+e^{-z}(1+z)\bigr]-1
    = \Big(1-\frac{B}{2}\Big)e^{-z}(2+z)-1.
\end{align*}
Using $(1+z)-(1-B/2)(2+z)=-(1-B)+\frac{B}{2}z$, subtracting gives:
\[
    U_{\OS}(\alpha)-U_{\TS}(\alpha)
    \,\geq\,
    (1-B)\Bigl[e^{-\frac{2}{2-B}z}-e^{-z}\Bigr]+\frac{B}{2}ze^{-z}.
\]
Dividing by $\frac{B}{2}e^{-z}>0$, this is equivalent to $R(z)<z$,
where
\[
    R(z)\,\triangleq\,\frac{2(1-B)}{B}\Bigl[1-e^{-\frac{B}{2-B}z}\Bigr].
\]
We have $R(0)=0$ and $R'(z)=\frac{2(1-B)}{2-B}e^{-\frac{B}{2-B}z}$,
which is strictly decreasing, so
\[
    R'(z)\leq R'(0)=\frac{2(1-B)}{2-B}=1-\frac{B}{2-B}<1
\]
for all $B\in(0,1)$. Hence $R(z)<z$ for all $z>0$, giving
$U_{\OS}(\alpha)>U_{\TS}(\alpha)$.

To transfer this to finite $\alpha_f$, set
$H(\alpha)\triangleq U_{\OS}(\alpha)-U_{\TS}(\alpha)>0$ and
$\varepsilon\triangleq H(\alpha)/3>0$. By
Propositions~\ref{prop:oneside} and~\ref{prop:twoside},
$\max F_{\OS}\to U_{\OS}(\alpha)$ and $\max F_{\TS}\to U_{\TS}(\alpha)$
as $\alpha_f\to\infty$, so there exists $\bar\alpha_f(\alpha,B)$ such
that for all $\alpha_f>\bar\alpha_f(\alpha,B)$:
\[
    \max F_{\OS}-\max F_{\TS}
    \,>\, \bigl(U_{\OS}(\alpha)-\varepsilon\bigr)
         -\bigl(U_{\TS}(\alpha)+\varepsilon\bigr)
    \,=\, \frac{H(\alpha)}{3}>0.
\]

\begin{proof}[Proof of Claim~\ref{clm:global}]
By Proposition~\ref{prop:twoside}, $U_{\TS}(\alpha)=\max(0,\Psi(\alpha))$
where $\Psi(\alpha)\triangleq\max_{y\in(0,1]}\Phi_{\TS}(y)$ is
strictly decreasing in $\alpha$ (since
$\frac{\partial}{\partial\alpha}\Phi_{\TS}(y,\alpha)=
-(1-B/2)(2-B)ye^{-(2-B)\alpha y}<0$ for all $y\in(0,1]$).
At $\alpha=0$, $y=1$ maximizes $\Phi_{\TS}(\cdot,0)$ giving
$\Psi(0)=1-B>0$. As $\alpha\to\infty$, $\Psi(\alpha)\to -B/2<0$.
By the intermediate value theorem, there exists a unique
$\alpha(B)\in(0,\infty)$ with $\Psi(\alpha(B))=0$, so
$U_{\TS}(\alpha)=0$ for all $\alpha\geq\alpha(B)$.

For $B\geq B^*=2-2e/3$, we evaluate $\Psi$ at $\frac{e}{2-B}$.
The inner function $g(y)\triangleq e^{-ey}+y(1-\ln y)$ satisfies
$g'(y)=-ee^{-ey}-\ln y=0$ at $y=1/e$, with $g'>0$ on $(0,1/e)$
and $g'<0$ on $(1/e,1]$, so $y=1/e$ is the unique maximizer,
giving $g(1/e)=\frac{3}{e}$. Thus
\[
    \Psi \Big(\frac{e}{2-B}\Big)
    = \frac{3(2-B)}{2e}-1
    \,\leq\, 0 = \Psi(\alpha(B)),
\]
where the inequality uses $3(2-B)\leq 2e$. Since $\Psi$ is
strictly decreasing, $\alpha(B)\leq\frac{e}{2-B}$.
\end{proof}

\subsubsection{Proof of~\texorpdfstring{\Cref{thm:small_large_alpha}}{}}
\label{sec:small_large_alpha}
Fix any $B\in(0,1)$. For both regimes, we establish
$U_{\OS}(\alpha)>U_{\TS}(\alpha)$ for the relevant range of
$\alpha$; the transfer to finite $\alpha_f$ is then analogous
to the proof of~\Cref{thm:global} and hence omitted.

\paragraph{(i) (Small-$\alpha$ regime.)} We first introduce the following claim. 
        \begin{claim}\label{clm:uniquemaxphiTS}
         For $2\alpha(1-B/2)<e$, $\Phi_{\TS}(y)$ admits a unique maximizer
            $y^*\in(0,1)$ over $[0,1]$, satisfying
            \[
                \ln y^* = -2\alpha\Big(1-\frac{B}{2}\Big)
                e^{-2\alpha(1-B/2)y^*}
                \qquad\text{and}\qquad
                y^* = e^{-2\alpha(1-B/2)y^*}.
            \]
If $\max_{y\in(0,1]}\Phi_{\rm TS}(y)\le0$, then $U_{\rm TS}(\alpha)=0$, while
$U_{\rm OS}(\alpha)\ge \Phi_{\rm OS}(1)=(1-B)e^{-2\alpha}>0$, and we are done. We may therefore assume $\max_y\Phi_{\rm TS}(y)>0$, so that $U_{\rm TS}(\alpha)=\Phi_{\rm TS}(y^*)$. 

By~\Cref{eq:M_os} and \prettyref{eq:L_ts}, given that $y^*$ is the unique maximizer of $\Phi_{\TS}(y)$ over  $[0,1]$ by~\Cref{clm:uniquemaxphiTS}, we obtain 
    \begin{align*}
        U_{\OS}(\alpha) \!- \! U_{\TS}(\alpha)
        &\geq \Phi_{\OS}(y^*)-\Phi_{\TS}(y^*)\\
        &= (1-B)e^{-2\alpha y^*}+y^*(1 \!-\! \ln y^*) \!-\! 1
          \!-\! \Big[\Big(1-\frac{B}{2}\Big)\big(e^{-2\alpha(1-B/2)y^*}
          +y^*(1 \!-\! \ln y^*)\big) \!-\! 1\Big] \\
          & =  (1-B)e^{-2\alpha y^*}
        -\Big(1-\frac{B}{2}\Big)e^{-2\alpha(1-B/2)y^*}
        +\frac{B}{2}y^*(1-\ln y^*) \\
          &= y^*\Bigl[(1-B)e^{\frac{B/2}{1-B/2}\ln y^*}
      -(1-B/2)+\frac{B}{2}(1-\ln y^*)\Bigr]\\
    &= y^*\Bigl[(1-B)\bigl(e^{\frac{B/2}{1-B/2}\ln y^*}-1\bigr)
      -\frac{B}{2}\ln y^*\Bigr]
    \,\triangleq\, y^*\Delta(y^*) \,,
    \end{align*}
    where the first equality holds by \prettyref{eq:Phi_os} and \prettyref{eq:Phi_ts}, the third equality holds by $(1-B/2)e^{-2\alpha(1-B/2)y^*} =(1-B/2)y^*$ and $e^{-2\alpha y^*}=(y^*)^{1/(1-B/2)} =y^*e^{\frac{B/2}{1-B/2}\ln y^*}$ by \Cref{clm:uniquemaxphiTS}.

    It remains to show $\Delta(y^*)>0$.
    Using $e^x-1>x$ for $x\neq 0$ with $x=\frac{B}{2-B}\ln y^*$:
    \[
        \Delta(y^*)
        > \Bigl(\frac{B(1-B)}{2-B}-\frac{B}{2}\Bigr)\ln y^*
        = -\frac{B^2}{2(2-B)}\ln y^*
        > 0,
    \]
    since $y^*\in(0,1)$ gives $\ln y^*<0$.
    Hence, we have $U_{\OS}(\alpha)-U_{\TS}(\alpha)>0$ for all $\alpha<\frac{e}{2-B}\triangleq \alpha_{\mathrm{low}}(B)$.
    \end{claim}

    \paragraph{(ii) Large $\alpha$ regime.} Evaluating $\Phi_{\OS}$
    at $y=1$:
    \[
        U_{\OS}(\alpha)
        \geq \Phi_{\OS}(1)
        = (1-B)e^{-2\alpha}
        > 0
        \qquad\text{for all }\alpha\geq 0.
    \]
    By~\Cref{clm:global}, there exists $\alpha(B)>0$ such that
    $U_{\TS}(\alpha)=0$ for all $\alpha>\alpha(B)$. Hence for any
    $\alpha>\alpha(B)$, $U_{\OS}(\alpha)\geq(1-B)e^{-2\alpha}>0$
    while $U_{\TS}(\alpha)=0$. The transfer argument applies with
    $U_{\OS}(\alpha)-U_{\TS}(\alpha)=U_{\OS}(\alpha)>0$,
    completing the proof.
    We are left to prove our claim. 
    \begin{proof}[Proof of~\Cref{clm:uniquemaxphiTS}]
    Direct differentiation gives
    \begin{align*}
        \Phi_{\TS}'(y)
        &= -\Big(1-\frac{B}{2}\Big)\Big[2\alpha\Big(1-\frac{B}{2}\Big)
           e^{-2\alpha(1-B/2)y}+\ln y\Big],\\
        \Phi_{\TS}''(y)
        &= \Big(1-\frac{B}{2}\Big)\Big[4\alpha^2\Big(1-\frac{B}{2}\Big)^2
           e^{-2\alpha(1-B/2)y}-\frac{1}{y}\Big].
    \end{align*}
    Setting $x=2\alpha(1-B/2)y$ gives
    $\Phi_{\TS}''(y)=\frac{1-B/2}{y}[2\alpha(1-B/2)xe^{-x}-1]$.
    Since $xe^{-x}\leq 1/e$ for all $x\geq 0$ and $2\alpha(1-B/2)<e$,
    we have $2\alpha(1-B/2)xe^{-x}<1$, so $\Phi_{\TS}''(y)<0$ and
    $\Phi_{\TS}$ is strictly concave on $(0,1]$.
    Since $\lim_{y\to 0^+}\Phi_{\TS}'(y)=+\infty$ and
    $\Phi_{\TS}'(1)=-2\alpha(1-B/2)^2e^{-2\alpha(1-B/2)}<0$,
    the intermediate value theorem gives a unique $y^*\in(0,1)$
    with $\Phi_{\TS}'(y^*)=0$, i.e.\
    \[
        \ln y^* = -2\alpha\Big(1-\frac{B}{2}\Big)e^{-2\alpha(1-B/2)y^*}.
    \]
    It remains to show $y^*=e^{-2\alpha(1-B/2)y^*}$.
    The function $g(y)\triangleq y-e^{-2\alpha(1-B/2)y}$ satisfies
    $g(0)=-1<0$, $g(1)=1-e^{-2\alpha(1-B/2)}>0$, and
    $g'(y)=1+2\alpha(1-B/2)e^{-2\alpha(1-B/2)y}>0$,
    so by the intermediate value theorem there exists a unique
    $\tilde{y}\in(0,1)$ with $\tilde{y}=e^{-2\alpha(1-B/2)\tilde{y}}$.
    This $\tilde{y}$ satisfies $\Phi_{\TS}'(\tilde{y})=0$ since
    \[
        \ln\tilde{y}
        = -2\alpha\Big(1-\frac{B}{2}\Big)\tilde{y}
        = -2\alpha\Big(1-\frac{B}{2}\Big)e^{-2\alpha(1-B/2)\tilde{y}},
    \]
    so by uniqueness of the zero of $\Phi_{\TS}'$, we have $\tilde{y}=y^*$.
\end{proof}

\section{Simulations}\label{sec:simulation}

\begin{figure}[t]
    \centering
    \includegraphics[width=0.7\linewidth]{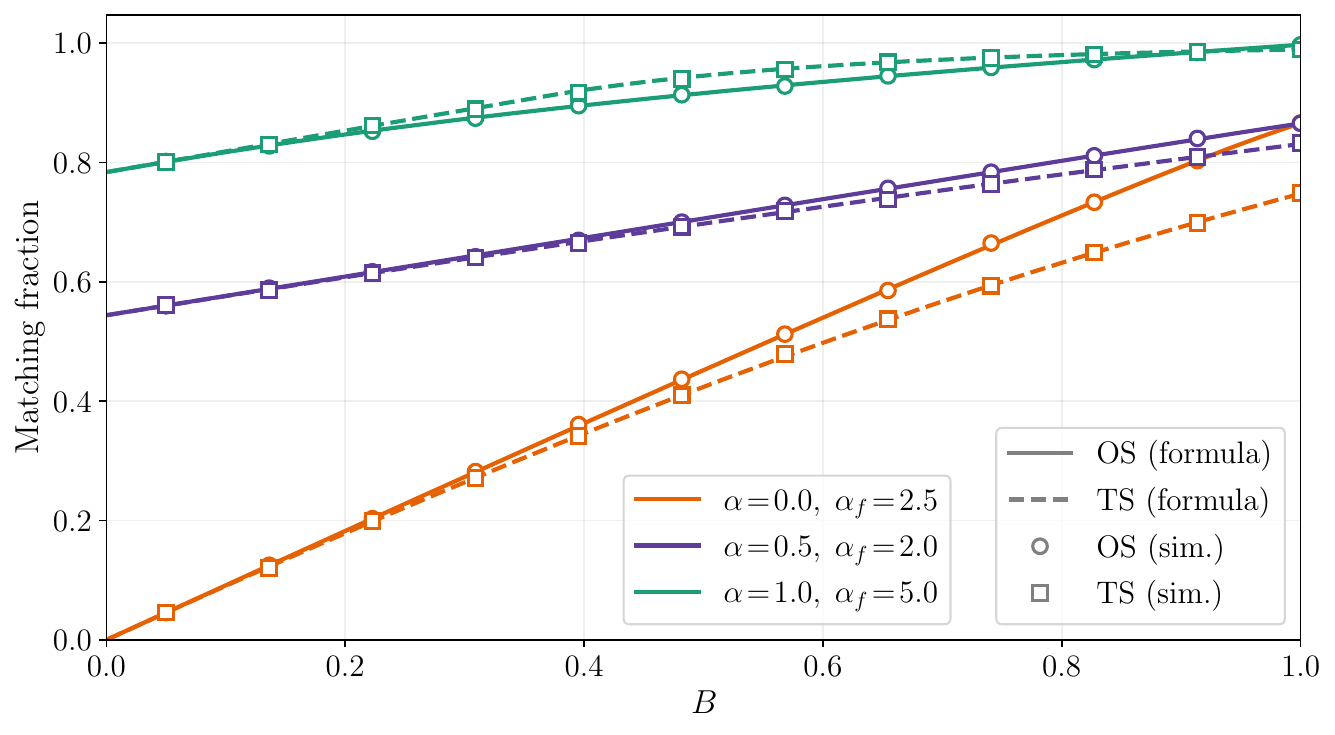}
    \caption{Matching fraction $\match_{\OS/\TS}(B, \alpha, \alpha_f)$ against $B$ (formula vs simulation) for various $(\alpha,\alpha_f)$.}
    \label{fig-comparison}
\end{figure}

This section validates the closed-form formulas of Corollaries~\ref{cor:matching_size_one}--\ref{cor:matching_size_two} by simulating the dominance landscape across the full parameter space $(\alpha, \alpha_f, B)$. 
The variational characterization makes this computationally tractable: rather than running a matching algorithm for each parameter configuration, we numerically solve the low-dimensional optimization problems
\[
    \match_{\OS}(B,\alpha,\alpha_f) = 1 - \max_{t_1\in[0,1]} F_{\OS}(t_1) \, ,
    \qquad
    \match_{\TS}(B,\alpha,\alpha_f) = 1 - \max_{t_1,t_2\in[0,1]^2} F_{\TS}(t_1,t_2) \, ,
\]
and compute the advantage ratio
\[
    \Adv(\alpha, \alpha_f, B)  \, \triangleq \,
    \frac{\match_{\OS}(B,\alpha,\alpha_f) }{ \match_{\TS}(B,\alpha,\alpha_f) } \, ,
\]
$\Adv > 1$ indicates one-sided dominance and $\Adv < 1$ indicates two-sided dominance. This formulation allows us to sweep efficiently over the parameter space and locate the boundaries between dominance regimes. We organize the results as follows: Figure~\ref{fig-comparison} validates the formulas against simulation; Figure~\ref{fig-B} traces $\Adv$ as a function of~$\alpha_f$ for varying~$\alpha$, at flexibility budgets $B=1$ and $B=0.6$; Figure~\ref{fig-Adv-vs-B} traces $\Adv$ as a function of~$B$; Figure~\ref{fig-Adv-alpha0} traces $\Adv$ when $\alpha = 0$, and Figure~\ref{fig-heatmap} maps the full dominance regions in the $(\alpha_f - \alpha,\,\alpha)$ plane across different flexibility budgets $B$. Code for all simulations is available online\footnote{\url{https://github.com/taha-ameen/flexibility-in-bipartite-sbm}}.

\begin{figure}[t]
    \centering
    \subfigure[$B=1$, $\alpha \leq 1$]{
    \includegraphics[width = 0.47\textwidth]{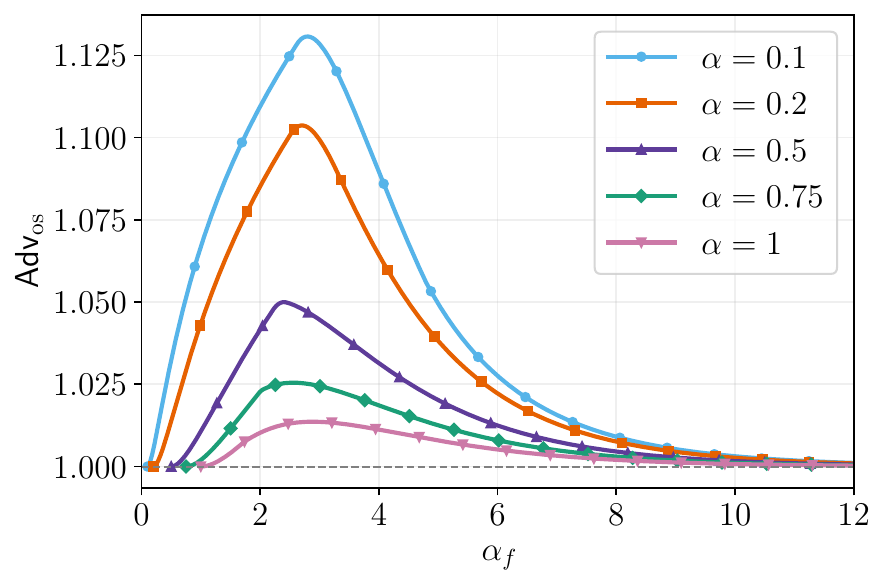}
    }
    \hfill
    \subfigure[$B=1$, $\alpha \geq 1$]{
    \includegraphics[width = 0.47\textwidth]{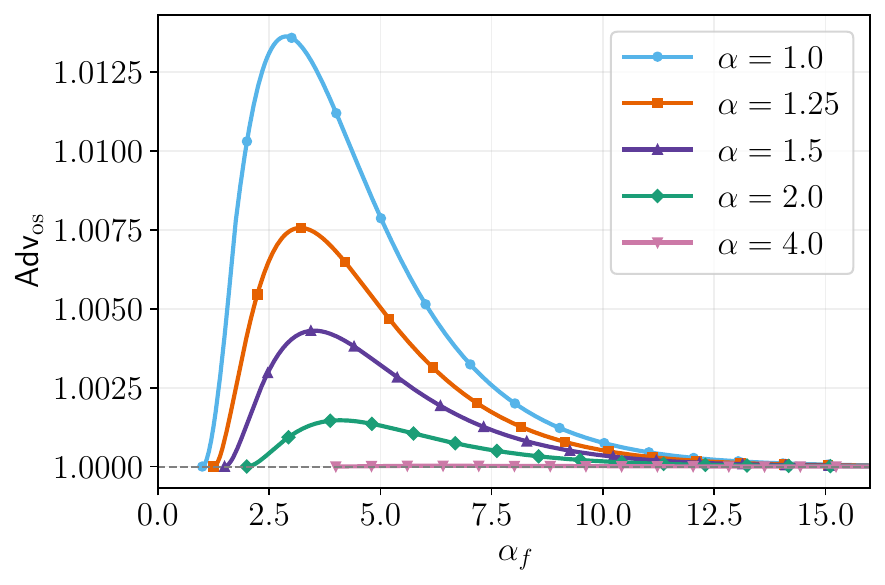}
    }
    \vfill
    \subfigure[$B=0.6$, $\alpha \leq 1$]{
    \includegraphics[width = 0.47\textwidth]{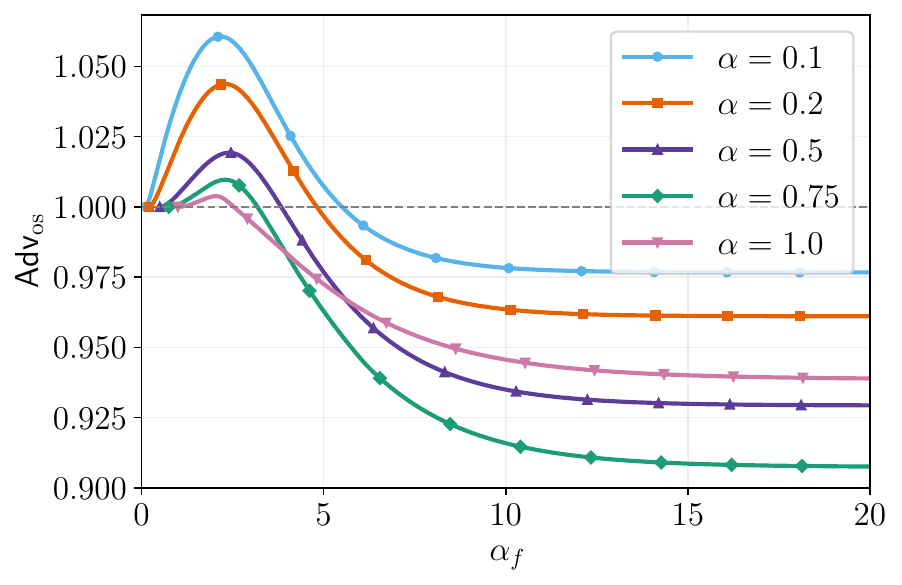}
    }
    \hfill
    \subfigure[$B=0.6$, $\alpha \geq 1$]{
    \includegraphics[width = 0.47\textwidth]{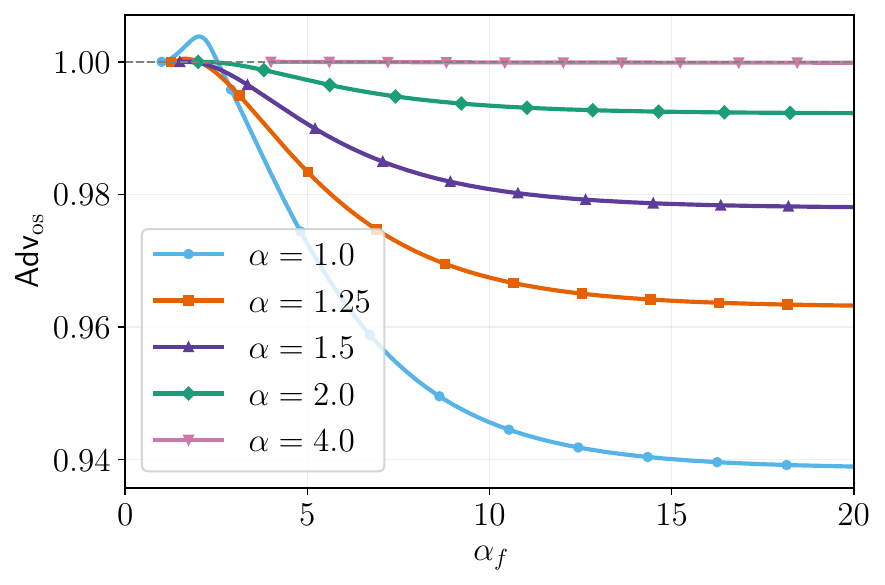}
    }
    \caption{$\Adv$ against $\alpha_f$ for various $\alpha$, at $B=1$ (top) and $B=0.6$ (bottom).}
    \label{fig-B}
\end{figure}

\paragraph{Empirical validation.}
Figure~\ref{fig-comparison} validates the analytical matching rate of~\Cref{thm:matching-rate} against simulation. For three different values of $(\alpha, \alpha_f)$, we plot the matching fraction under both the one-sided allocation (solid lines) and the two-sided allocation (dashed lines) as a function of the flexibility budget $B \in [0,1]$. Simulation points are obtained via the Hopcroft--Karp algorithm on random instances of the bipartite stochastic block model with $n = 10^4$ nodes per side, averaged over 10 trials. The agreement between formula and simulation across all cases confirms the validity of the formulas in~\Cref{cor:matching_size_one} and~\Cref{cor:matching_size_two}.

The three cases in~\Cref{fig-comparison} illustrate qualitatively distinct dominance regimes. When $\alpha = 0$ and $\alpha_f = 2.5$, the one-sided allocation strictly outperforms the two-sided allocation for all~$B$, consistent with~\Cref{thm:alpha_0}. The moderate case $\alpha = 0.5$, $\alpha_f = 2$ exhibits a similar pattern, though the gap between the two allocations is considerably smaller. In contrast, for $\alpha = 1$ and $\alpha_f = 5$ (large flexibility premium), the two-sided allocation dominates over most of the flexibility budget range, with the one-sided allocation recovering only near $B = 1$. This reversal reflects the interplay identified in~\Cref{thm:small_large_alpha}: when $\alpha > 0$ and $\alpha_f$ is sufficiently large, the asymmetry effect outweighs flexibility cannibalization, favoring the two-sided allocation; the recovery at $B = 1$ is guaranteed by~\Cref{thm:B_1}.

\paragraph{Plotting $\Adv$ against $\alpha_f$.}

Figure~\ref{fig-B} plots $\Adv$ as a function of $\alpha_f$ for several values of $\alpha$, at $B = 1$ (top row) and $B = 0.6$ (bottom row). Each row splits the parameter range into two panels: $\alpha \leq 1$ (left) and $\alpha \geq 1$ (right).

When $B = 1$ (top row), the one-sided allocation dominates uniformly: $\Adv \geq 1$ for all tested values of $\alpha$ and $\alpha_f$, consistent with~\Cref{thm:B_1}. Two further patterns emerge. First, for each fixed $\alpha$, $\Adv$ is non-monotone in $\alpha_f$: it rises to a peak and then decays back toward~$1$ as $\alpha_f \to \infty$. The peak reflects the regime where the cannibalization effect is strongest relative to total matching volume; as $\alpha_f$ grows further, both allocations approach a perfect matching of the flexible nodes and the advantage vanishes. Second, the peak value of $\Adv$ is decreasing in $\alpha$. This is intuitive: a larger baseline rate $\alpha$ means that even regular--regular pairs connect at a reasonable rate, reducing the marginal value of concentrating flexibility on one side. Note also the difference in scale between the two panels: the advantage is an order of magnitude smaller for $\alpha \geq 1$, where the baseline connectivity already ensures a high matching rate.

When $B = 0.6$ (bottom row), the picture changes qualitatively. In the left panel ($\alpha \leq 1$), $\Adv$ again starts above~$1$ for small $\alpha_f$, but now crosses below~$1$ as $\alpha_f$ grows, indicating a transition to two-sided dominance. The crossover occurs at smaller $\alpha_f$ for larger $\alpha$, consistent with the intuition that a higher baseline rate amplifies the asymmetry effect. This sign change is predicted by~\Cref{thm:global,thm:small_large_alpha}. In the right panel ($\alpha \geq 1$), $\Adv$ is below~$1$ for most of the regime $\alpha_f > \alpha$, confirming that for larger baseline rates the asymmetry effect dominates even at moderate flexibility premiums. 

In summary,the comparison between the two rows underscores the role of the flexibility budget~$B$: at full budget, one-sided dominance prevails (\Cref{thm:B_1}); below full budget, the asymmetry effect can reverse dominance, with the reversal occurring more readily when $\alpha$ is larger.

\paragraph{Plotting $\Adv$ against $B$.}

\begin{figure}[t]
    \subfigure[Moderate flexibility premium ($\alpha_f/\alpha \leq 2$)]{
    \includegraphics[width = 0.47\textwidth]{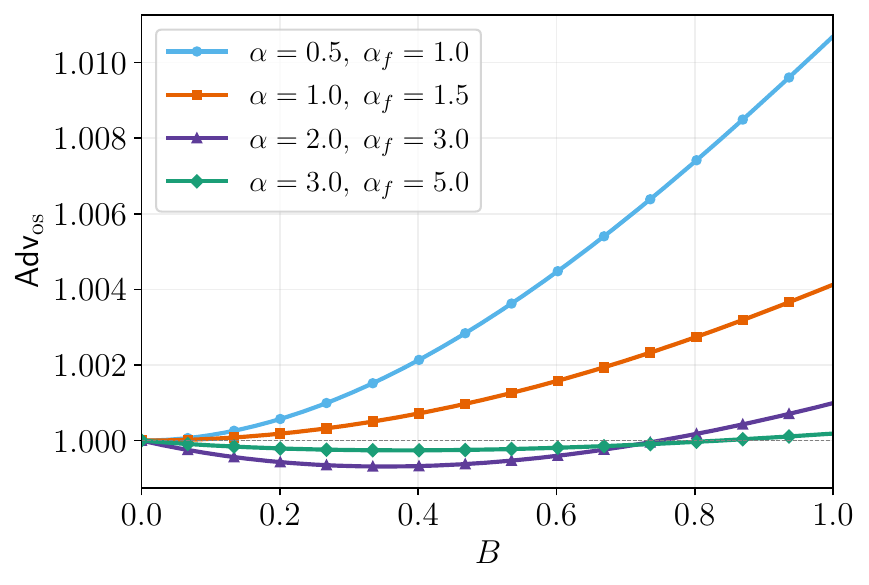}
    \label{fig-Adv-vs-B-1}
    }
    \hfill
    \subfigure[Large flexibility premium ($\alpha_f/\alpha \geq 4$)]{
    \includegraphics[width = 0.47\textwidth]{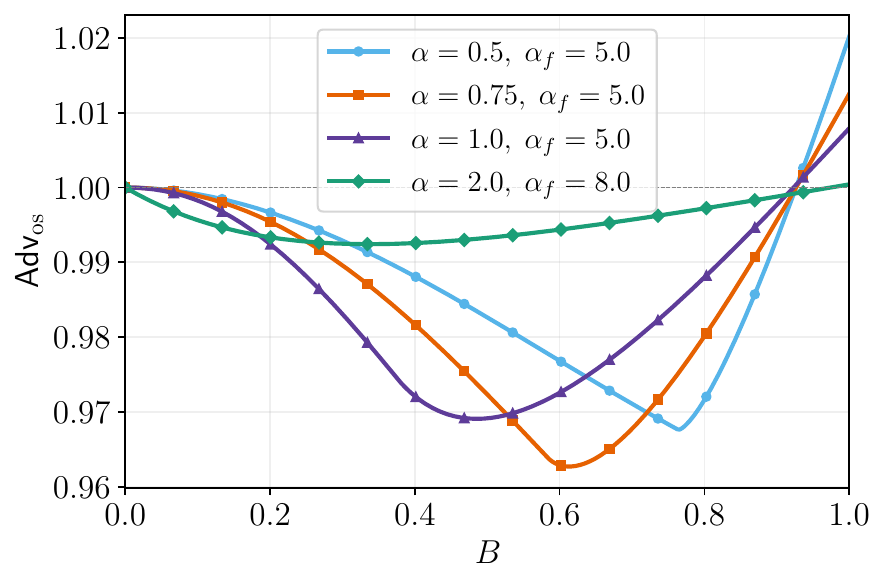}
    }
    \caption{Plotting $\Adv$ against $B$ for various $\alpha,\alpha_f$}
    \label{fig-Adv-vs-B}
\end{figure}

\begin{figure}[t]
    \subfigure[Plotting against $\alpha_f$]{
    \includegraphics[width = 0.47\textwidth]{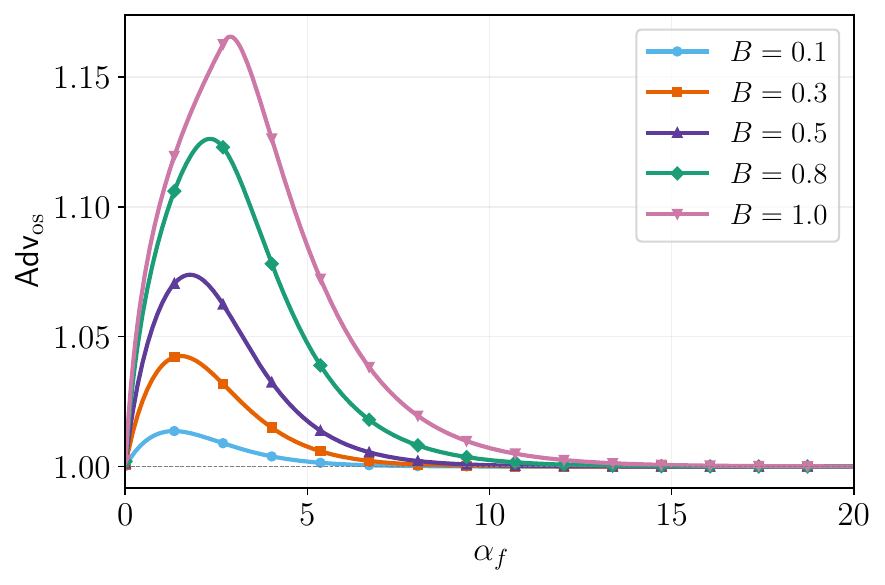}
    \label{fig-alpha0-1}
    }
    \hfill
    \subfigure[Plotting against $B$]{
    \includegraphics[width = 0.47\textwidth]{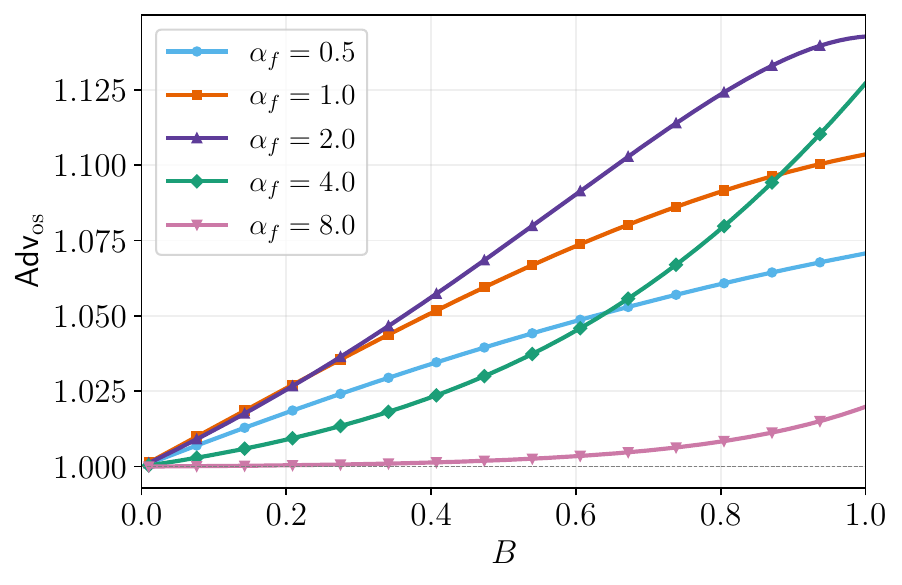}
    \label{fig-alpha0-2}
    }
    \caption{Plotting $\Adv$ against $\alpha_f$ and $B$ when $\alpha = 0$}
    \label{fig-Adv-alpha0}
\end{figure}

\Cref{fig-Adv-vs-B} plots $\Adv$ against the flexibility budget~$B$. The two panels share a common feature: for every parameter pair tested, $\Adv$ is increasing near $B = 1$ and satisfies $\Adv \geq 1$ at $B = 1$, consistent with~\Cref{thm:B_1}. Away from full flexibility budget, however, the comparison depends on both the flexibility premium and the budget level. Even at moderate ratios $\alpha_f/\alpha$ (panel~(a)), some parameter pairs exhibit $\Adv < 1$ at small~$B$ before crossing back above~$1$, showing that one-sided dominance need not hold uniformly across flexibility budgets. This effect is amplified at larger ratios (panel~(b)), where the two-sided allocation dominates over most of the flexibility budget range. Together, the two panels highlight that the optimal allocation is not determined by the flexibility premium $\alpha_f$ alone, and the available flexibility budget~$B$ plays an important role.

\paragraph{Plotting $\Adv$ when $\alpha = 0$.}

Figure~\ref{fig-Adv-alpha0} plots $\Adv$ when $\alpha = 0$ for various values of $B$ and $\alpha_f$. In every case $\Adv > 1$, confirming that the one-sided allocation dominates~(\Cref{thm:alpha_0}). Panel~(a) shows that the advantage is non-monotone in $\alpha_f$, peaking at an intermediate value before decaying back to~$1$. Panel~(b) shows that the advantage is monotonically increasing in the flexibility budget~$B$: a larger budget amplifies the benefit of concentrating flexibility on one side. Together, the two panels show that the one-sided advantage can be substantial (exceeding $15\%$ at $B = 1$) but is transient in $\alpha_f$, concentrated in the regime where flexibility is valuable but not yet abundant enough to make matching trivial.

\paragraph{Dominance regions.}

\begin{figure}[t]
    \subfigure[$B = 0.1$]{
    \includegraphics[width = 0.31\textwidth]{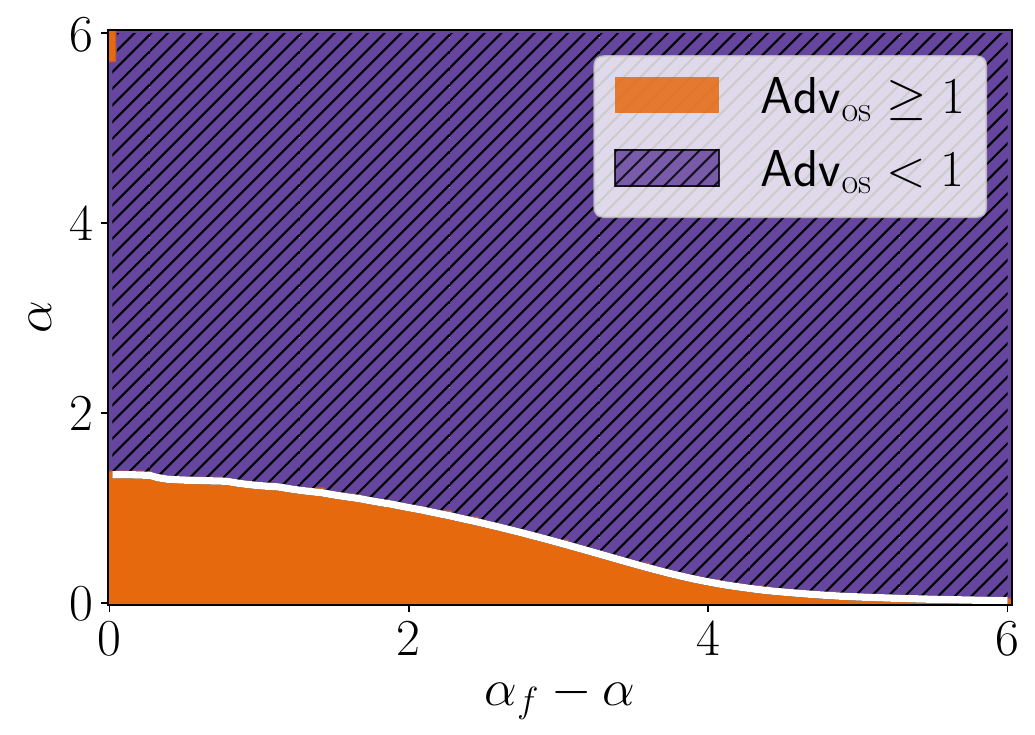}
    }
    \hfill
    \subfigure[$B = 0.2$]{
    \includegraphics[width = 0.31\textwidth]{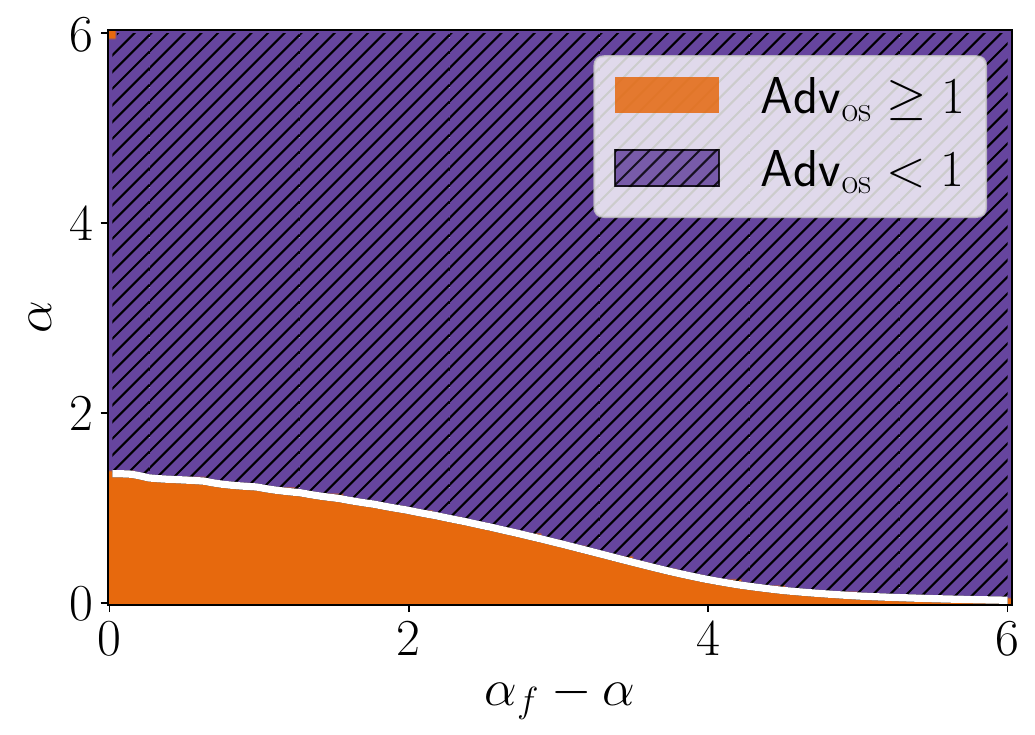}
    }
    \hfill
    \subfigure[$B = 0.4$]{
    \includegraphics[width = 0.31\textwidth]{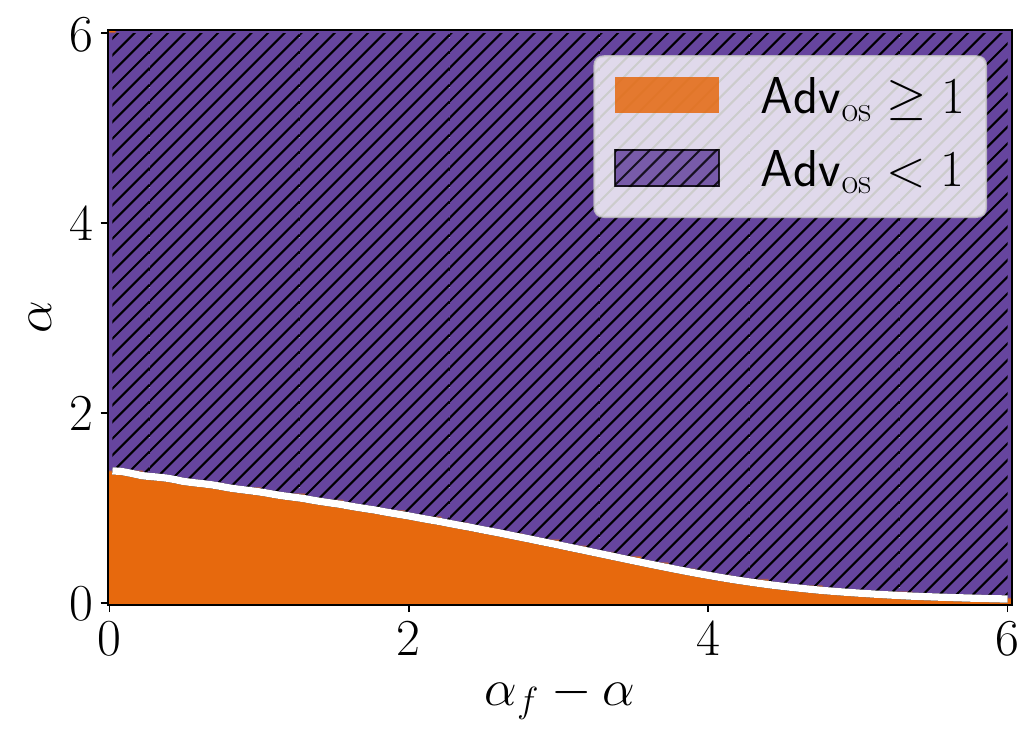}
    }
    \vfill
    \subfigure[$B = 0.5$]{
    \includegraphics[width = 0.31\textwidth]{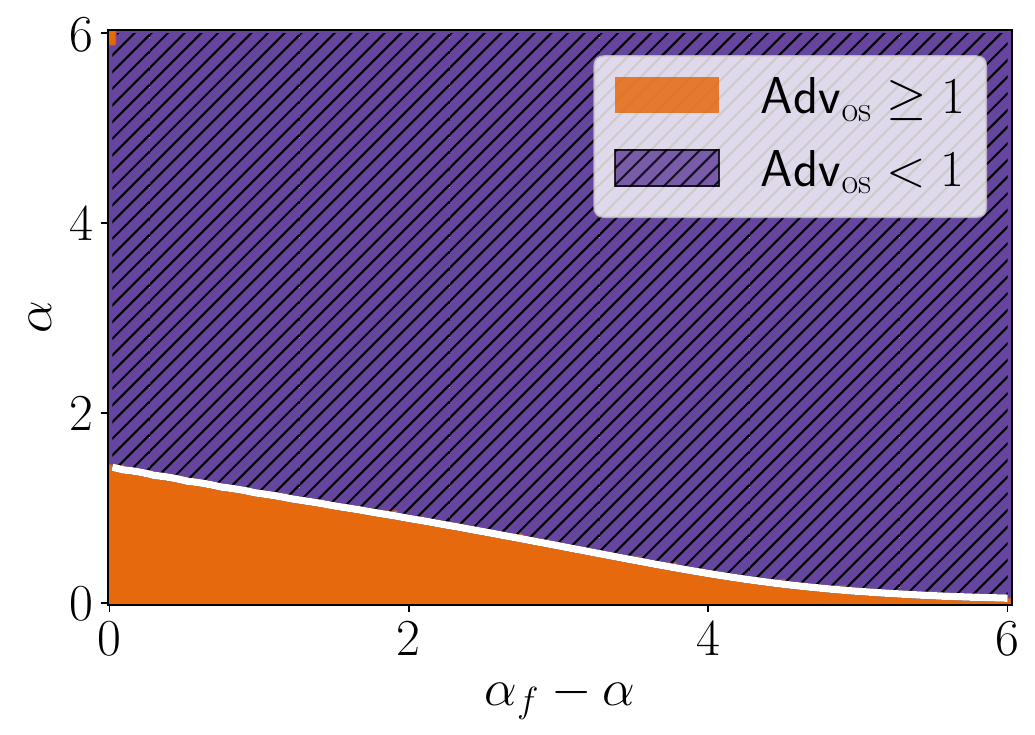}
    }
    \hfill    
    \subfigure[$B = 0.6$]{
    \includegraphics[width = 0.31\textwidth]{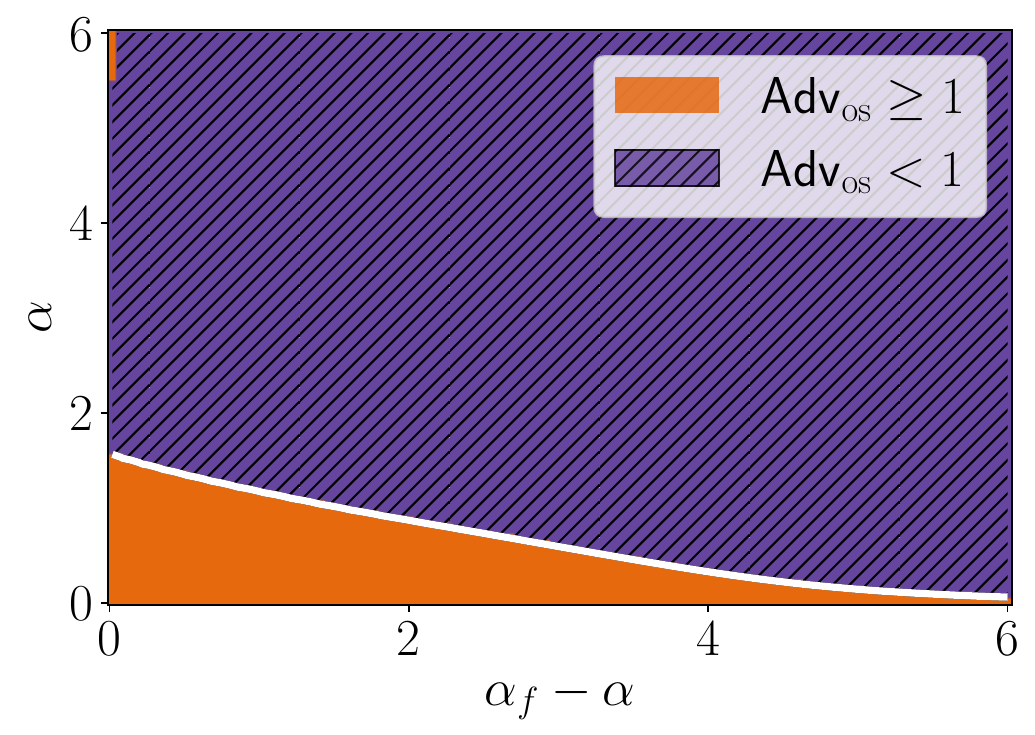}
    }
    \hfill
    \subfigure[$B = 0.7$]{
    \includegraphics[width = 0.31\textwidth]{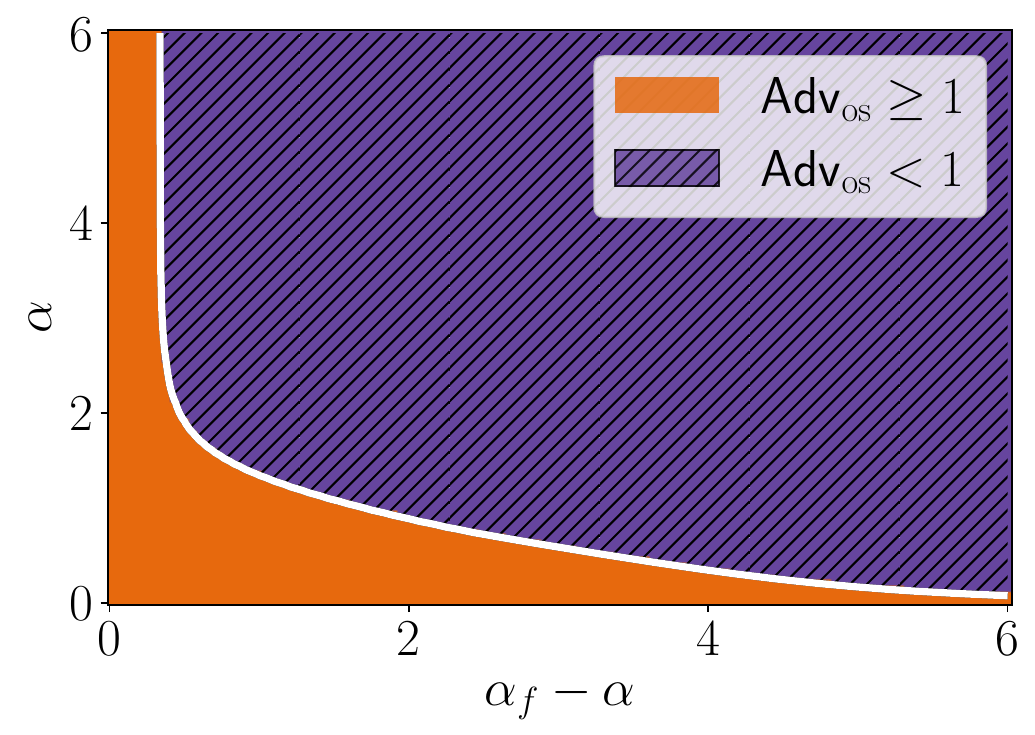}
    }
    \vfill
    \subfigure[$B = 0.8$]{
    \includegraphics[width = 0.31\textwidth]{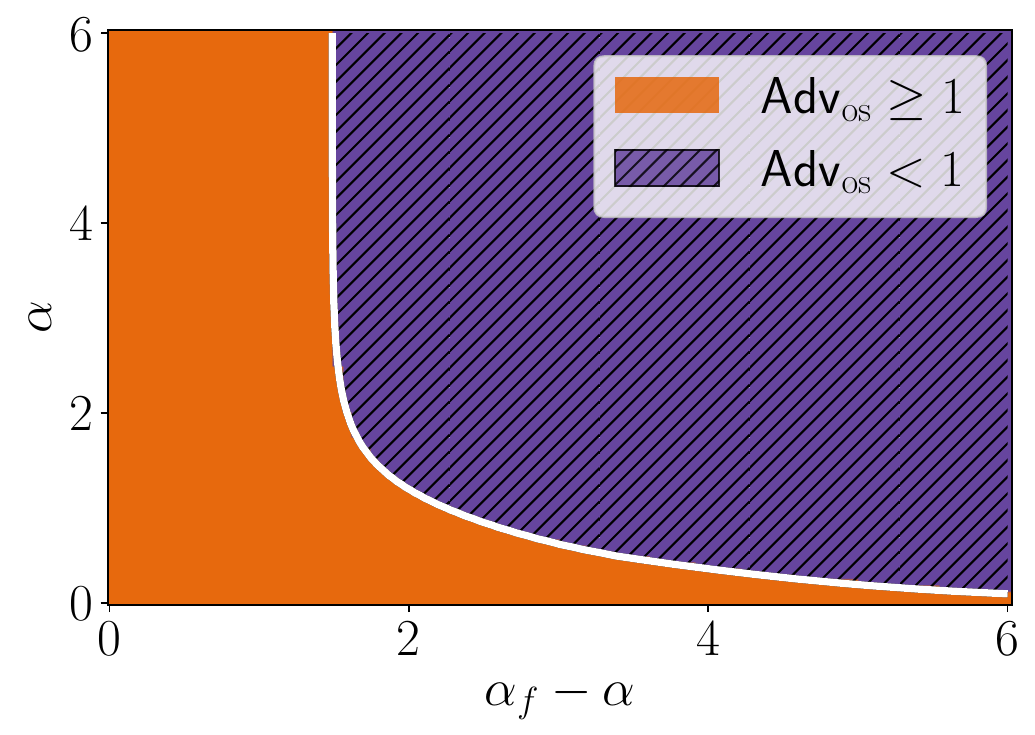}
    }
    \hfill    
    \subfigure[$B = 0.9$]{
    \includegraphics[width = 0.31\textwidth]{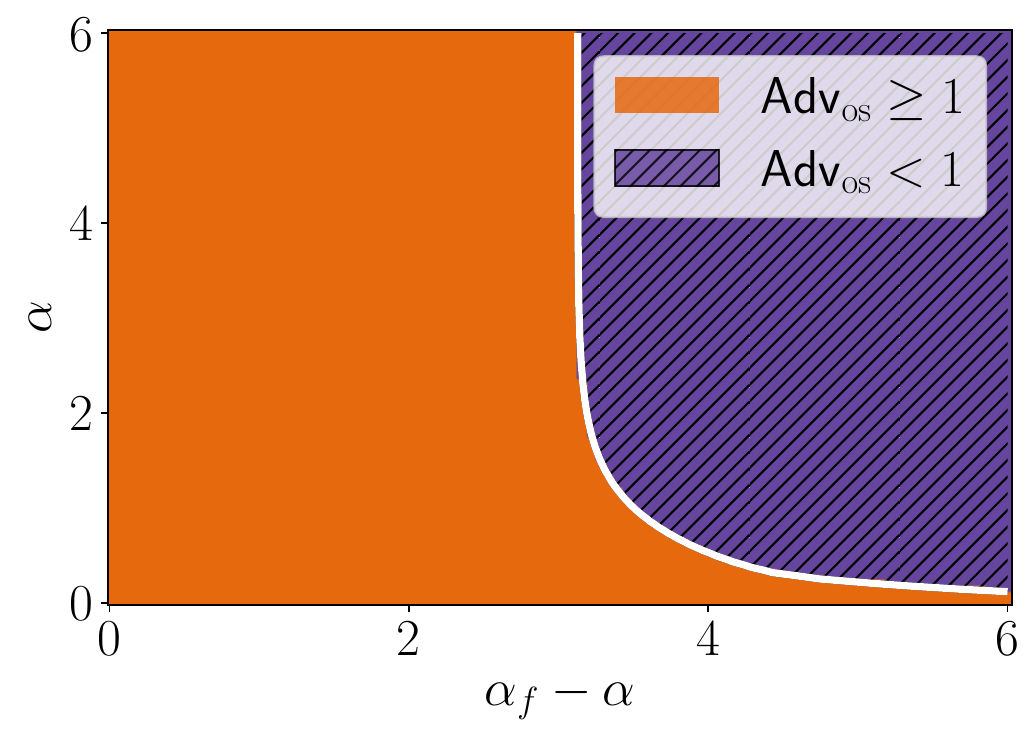}
    }
    \hfill
    \subfigure[$B = 1.0$]{
    \includegraphics[width = 0.31\textwidth]{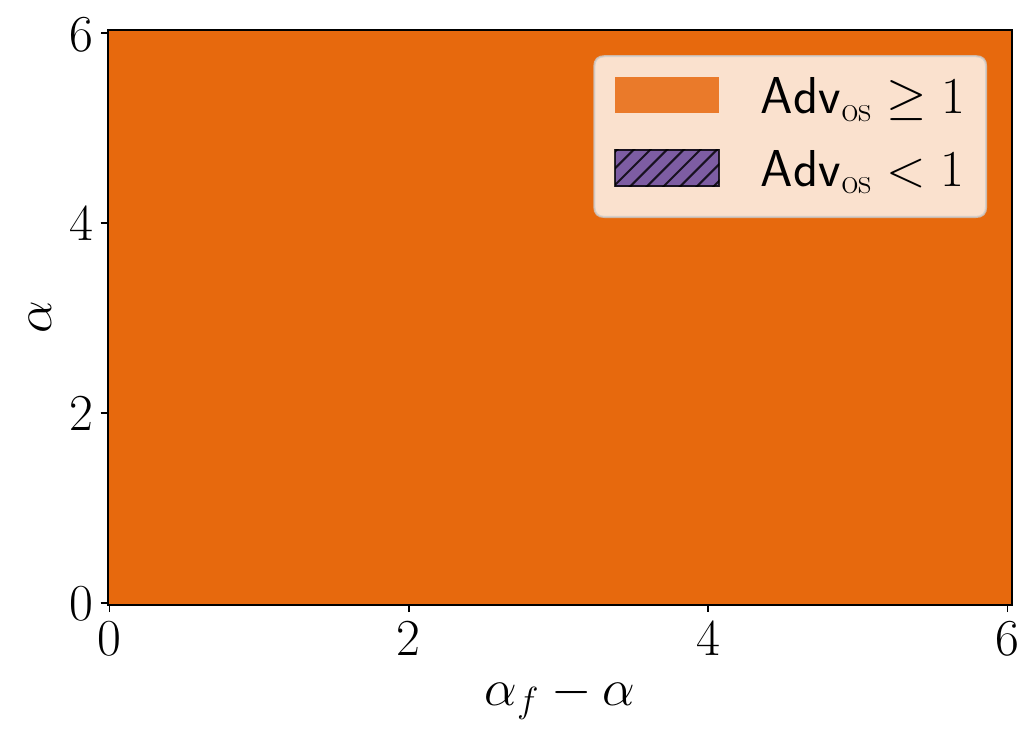}
    }    
    \caption{Dominance regions for various values of $B$}
    \label{fig-heatmap}
\end{figure}

Figure~\ref{fig-heatmap} displays the dominance regions in the $(\alpha_f - \alpha,\,\alpha)$ plane for nine values of~$B$ ranging from $0.1$ to~$1$. The panels reveal a clear monotone progression: as the flexibility budget~$B$ increases, the one-sided dominant region (orange) expands at the expense of the two-sided dominant region (purple), culminating in panel~(i) where one-sided dominance covers the entire parameter space, in agreement with~\Cref{thm:B_1}.

At small flexibility budgets (panels (a)--(c)), the two-sided allocation dominates over nearly all of the parameter space; one-sided dominance is confined to a thin region near $\alpha = 0$, precisely the regime identified by~\Cref{thm:alpha_0}. As $B$ increases through intermediate values (panels (d)--(f)), the one-sided region expands upward along the $\alpha$~axis, and the dominance boundary takes on a characteristic shape: for each fixed $\alpha > 0$, there is a value of $\alpha_f - \alpha$ beyond which the two-sided allocation dominates. This boundary shifts to larger values of $\alpha_f - \alpha$ as $B$ grows, reflecting the fact that a larger flexibility budget mitigates the asymmetry effect and requires a stronger flexibility premium to tip the balance toward two-sided allocation. At large flexibility budgets (panels (g)--(h)), the two-sided region retreats before disappearing entirely at $B = 1$.

Taken together, the panels provide a complete picture of how the three model parameters, i.e. baseline rate~$\alpha$, flexibility premium~$\alpha_f$, and flexibility budget~$B$, jointly determine the optimal allocation, complementing the analytical regimes established in~\Cref{thm:B_1,thm:alpha_0,thm:global,thm:small_large_alpha}.

\bibliographystyle{alpha}
\bibliography{references.bib}

\appendix

\section{Table of Notation} \label{apx-table}

\begin{table}[H]
\centering
\caption{Table of notation used in the paper}\label{tab:notation}
\begin{tabular}{ll}
    \hline
    \textbf{Notation} & \textbf{Description} \\[0.1em] \hline
    \noalign{\vspace{4pt}}
    \multicolumn{2}{l}{\textit{Model Parameters \& Graph Structure}} \\[0.3em]
        $n$ & Number of nodes per side in the balanced marketplace  \\[0.3em]
        $\bL, \bR$ & Probability that a supply or demand node is flexible  \\[0.3em]
        $B$ & Total flexibility budget ($\bL + \bR$)  \\[0.3em]
        $c_{xy}$ & Connection rate between supply type $x$ and demand type $y$  \\[0.3em]
        $\bfC = [c_{xy}]$ & $2 \times 2$ connection rate matrix  \\[0.3em]
        $\alpha, \alpha_f$ & Baseline rate and flexibility premium  \\[0.3em]
        $(p_1,p_2), (q_1,q_2)$ & Type probabilities for supply and demand sides  \\[0.3em]
        $M_y$ & $\triangleq \sum_x c_{xy} p_x$, aggregate demand rate for type $y$  \\[0.3em]
        $\lambda_x$ & $\triangleq \sum_y c_{xy} q_y$, aggregate supply rate for type $x$  \\[0.3em]
        \hline
        \noalign{\vspace{4pt}}
        \multicolumn{2}{l}{\textit{Analytical Framework}} \\[0.3em]
        $\pi_x^\mathrm{L}, \pi_y^\mathrm{R}$ & Degree distributions for supply and demand nodes  \\[0.3em]
        $\widehat{\pi}_x^\mathrm{L}, \widehat{\pi}_y^\mathrm{R}$ & Excess degree distributions  \\[0.3em]
        $\phi_x(s), \psi_y(s)$ & PGFs for supply and demand node degrees  \\[0.3em]
        $\widehat{\phi}_x(s), \widehat{\psi}_y(s)$ & PGFs for the excess degree distributions  \\[0.3em]
        $\aL_{xy}, \aR_{yx}$ & Neighbor-type mixing probabilities  \\[0.3em]
        $d$ & Truncation parameter for maximum degree  \\[0.3em]
        $\pi^{(\mathrm{L},d)}_{x}$ & Truncated distribution to degree $d$  \\[0.3em] 
        $F(t_1, t_2)$ & Variational objective function for matching rate  \\[0.3em]
        $t_1, t_2$ & Positivity vector components (Root Exposure Probabilities)  \\[0.3em]
        $\Theta$ & Fixed-point recursion operator for REPs  \\[0.3em]
        \hline
        \noalign{\vspace{4pt}}
        \multicolumn{2}{l}{\textit{Matching Performance \& Comparison}} \\[0.3em]
        $\matchnum(G_n)$ & Size of the maximum matching in graph $G_n$  \\[0.3em]
        $\match(\bfp, \bfq \,;\, \bfC)$ & Asymptotic matching rate \\[0.3em]
        $\match_{\OS}(B, \alpha, \alpha_f)$, $\match_{\TS}(B, \alpha, \alpha_f)$ & Matching rate for one-sided and two-sided allocations \\[0.3em]
        $\Adv(\alpha, \alpha_f, B)$ & Advantage ratio $\match_{\OS}/\match_{\TS}$; $>1$ iff one-sided dominates  \\[0.3em] 
        $\Phi_{\OS}, \Phi_{\TS}$ & Limit matching rate functions as $\alpha_f \to \infty$ \\[0.3em]
        \hline
\end{tabular}
\end{table}

\section{Supplementary proofs for \prettyref{sec:matching}}
\label{app:matching}

\subsection{Proof of~\texorpdfstring{\Cref{lmm:Term-1-2} }{}}

First, we prove~\eqref{eq:Term-1}. Recall that $S_i$ is the sum of offspring REPs of the $i$-th offspring demand node. Since $Z_x = \sum_{i=1}^{N_x} \frac{1}{S_i}$ and since $S_i \geq 0$, we have that $Z_x < \infty$ if and only if $S_i > 0$ for all $i$. For any term $i$, conditioned on the type of the $i$-th demand node $u$ being $y$, we have
\[
    \P(S_i = 0) = \P(\forall j\in [N_i], X_{ij}=0) = \mathbb{E}_{N \, \sim \, \pihatR_y} \left[ (1-t_y)^{N} \right] = \psihat_y(1-t_y) \, ,
\]
and so averaging over the type $y$, we have that each $S_i$ independently satisfies 
\[
    \P(S_i>0) = \sum_{y=1}^2 \aL_{xy} \left( 1 - \psihat_y(1-t_y)\right) = r_x \, .
\] 
Therefore,
\begin{align} \label{eq:first-term-simplification}
    \P(Z_x < \infty) = \E\left[ r_x^{N_x}\right] = \phi_x(r_x) \, .
\end{align}
This concludes the proof.

Next, we prove~\eqref{eq:Term-2}, by analyzing the LHS conditioned on $N_x=k$. For each $i\in\{1,\dots,k\}$, define
\[
    T_i \triangleq \frac{Y_{x,i}}{Y_{x,i}+S_i} \cdot \mathbf{1}_{\{S_1, \ldots, S_k>0\}} \, .
\]
The $k$-tuple $(T_1,\dots,T_k)$ is exchangeable, i.e. their joint law is invariant under any relabeling of the indices $1,\cdots,k$, hence:
\[
    \E \left[\sum_{i=1}^k T_i \, \Big\vert \, N_x=k\right]
    =
    \sum_{i=1}^k \E \left[T_i  \, \big \vert \,  N_x=k \right]
    =
    k \cdot \E\left[T_1 \, \big\vert \, N_x=k \right] \, .
\]
By the tower property of expectation:
\begin{align}\label{eq:towerpropresult}
    \E\left[  \sum_{i=1}^{N_x}\frac{Y_{x,i}}{Y_{x,i} + S_i} \cdot \mathbf{1}_{\{ S_{1}, \ldots, S_{N_x} > 0 \}}\right] = \E \left[N_x\cdot \E\left[T_1 \, \vert \,  N_x \right]\right]. 
\end{align}
We now analyze the inner expectation. First, conditional on $S_1>0$, $\frac{Y_{x,1}}{Y_{x,1} + S_1}=0$ if $\mathbf{1}_{\{ S_{2}, \ldots, S_{N_x} > 0 \}}$ is not satisfied. Hence:
\[
    \E\left[T_1 \, \vert \,  N_x \right] =\E\left[\frac{Y_{x,1}}{Y_{x,1}+S_1} \cdot \mathbf{1}_{\{S_1>0\}} \vert \,  N_x \right]
\]
Recall that, by definition, for any $k>0$, $S_1$ is independent of $Y_{x,1}$ and $N_x$. Also by definition, the type of $S_1$ is $y$ with probability $\aL_{xy}$. Hence, with $Y_x^{(k)}$ the r.v. $Y_{x,1}$ conditioned to $N_x=k$, and $S^y$ the r.v. $S_1$ conditioned to its type being $y$, we have:
\begin{align*}
    \E\left[T_1 \, \vert \,  N_x=k \right]=    \sum_{y=1}^2 \aL_{xy}\,
    \E \bigg[\frac{Y_x^{(k)}}{Y_x^{(k)}+S^y} \cdot \mathbf{1}_{\{S^y>0\}}\bigg] 
\end{align*}
Let us now analyze the outer expectation. Plugging the above in~\eqref{eq:towerpropresult}, we get
\begin{align*}
    E\left[  \sum_{i=1}^{N_x}\frac{Y_{x,i}}{Y_{x,i} + S_i} \cdot \mathbf{1}_{\{ S_{1}, \ldots, S_{N_x} > 0 \}}\right]     &=
    \sum_{k \geq 1} \P(N_x = k) \,k
    \sum_{y=1}^2 \aL_{xy} \,
    \E \bigg[\frac{Y_x^{(k)}}{Y_x^{(k)}+S^y} \cdot \mathbf{1}_{\{S^y>0\}} \bigg]
    \\
    & =
   \sum_{y=1}^2 \aL_{xy}  \sum_{k \geq 1} \P(N_x = k) \,k
    \,
    \E \bigg[\frac{Y_x^{(k)}}{Y_x^{(k)}+S^y} \cdot \mathbf{1}_{\{S^y>0\}} \bigg]
\end{align*}
We introduce  
\[
    X_x \stackrel{\mathrm{d}.}{=} \bigg(1+\sum_{j=1}^{\widehat{N}_x} S_j^{-1}\bigg)^{-1} \, ,
\]
where $\widehat{N}_x\sim \pihatL_x$ and $S_j$ is generated the same way as $S_i$ in the definition of $Z_x$. This yields
\begin{align*} 
    \P(Z_x < \infty) - \E\left[  \sum_{i=1}^{N_x}  \frac{Y_{x,i}}{Y_{x,i} + S_i} \cdot \mathbf{1}_{\{ S_{1}, \ldots, S_{N_x} > 0 \}}\right] 
    = 
    \phi_x(r_x) - \phi'_x(1) \sum_{y=1}^2 \aL_{xy}\, \E \left[ \frac{X_x}{X_x+S^y} \cdot \mathbf{1}_{\{S^y>0\}} \right] .
\end{align*}
Since $\widehat{\mu}$ is a fixed point of $\Theta$, we have $X_x\sim \widehat{\mu}_x$, and the proof is complete.

\subsection{Proof of~\texorpdfstring{\Cref{clm:formula-for-Ky}}{} }

We introduce the following shorthand:
\[
    K_y(\mu_1,\mu_2)
    \, \triangleq \,
     \psi'_y(1) \cdot  \sum_{x=1}^2 \aR_{yx}\,
    \E\left[\frac{X_x}{X_x+S^y}\mathbf{1}_{\{S^y>0\}}\right] \,,
\]
where $X_x \sim \widehat\mu_x$ for each  $x\in \{1, 2\}$ and $S^y$ has the distribution of the sum of a $\pihatR_y$-distributed number of i.i.d. variables drawn from the mixture $\sum_{x} \aR_{yx} \, \widehat\mu_{x}$. 
First, we simplify the sum. Denote $X\sim\sum_{x} \aR_{yx} \,  \widehat\mu_{x}$. We have:
\[
    K_y(\mu_1,\mu_2)
    \, = \,
     \psi'_y(1)  \cdot
    \E\left[\frac{X}{X+S^{y}}\mathbf{1}_{\{S^{y}>0\}}\right]\,,
\]
With $X_i$ independent copies of $X$ and $N'_y\sim \pihatR_y$, the above can be rewritten:
\[
    K_y(\mu_1,\mu_2)
    \, = \,
     \psi'_y(1)  \cdot
    \E\bigg[\frac{X}{X+\sum_{i=1}^{N'_{y}}X_i}\mathbf{1}_{\{\sum_{i=1}^{N'_{y}}X_i>0\}}\bigg]\,.
\]
Defining $N^{\star} =\sum_{i=1}^{N'_{y}}\mathbf{1}_{\{X_i>0\}}$ and $X^+$, $X_i^+$ independent random variable iid from $\sum_{x} \aR_{yx} \,  \widehat\mu_{x}$ conditioned on being strictly positive:
\[
    K_y(\mu_1,\mu_2)
    \, = \,
     \psi'_y(1) \cdot  \mathbb{P}(X>0) \cdot 
    \E\left[\frac{X^+}{X^++\sum_{i=1}^{N^{\star}}X_i^+}\mathbf{1}_{\{N^{\star}>0\}}\right]\,.
\]
By exchangeability, and as $\hat{t}_y = \mathbb{P}(X>0)$:
\[
    K_y(\mu_1,\mu_2)
    \, = \,
     \psi'_y(1) \cdot   \hat{t}_y\,
    \E\left[\frac{1}{1+N^{\star}}\mathbf{1}_{\{N^{\star}>0\}}\right]\,.
\]
The remainder of the proof is algebra.
\begin{align*}
    K_y(\mu_1,\mu_2) 
        & \,=\, \psi'_y(1) \cdot \hat{t}_y \cdot \sum_{n\geq 1} \widehat{\pi}_y(n)\sum_{d=1}^n\binom{n}{d}\frac{\hat{t}_y^d \, (1-\hat{t}_y)^{n-d}}{d+1}  \\
        & \,=\, \psi'_y(1) \cdot \sum_{n\geq 1} \widehat{\pi}_y(n)\sum_{d=1}^n\binom{n}{d} (1-\hat{t}_y)^{n-d}\int_0^{\hat{t}_y}t^{d} \, \diff t \\
        & \,=\, \psi'_y(1) \cdot \int_0^{\hat{t}_y} \widehat{\psi}_y(1+t-\hat{t}_y) \, \diff t \,-\, \psi'_y(1) \cdot \hat{t}_y \cdot \widehat{\psi}_y(1-\hat{t}_y) \\
        & \,=\, \int_{1-\hat{t}_y}^{1} \psi'_y(t) \, \diff t \, -\, \hat{t}_y \cdot \psi'_y(1-\hat{t}_y) \\
        & \,=\, 1- \psi_y(1-\hat{t}_y)-\hat{t}_y \cdot \psi'_y(1-\hat{t}_y) \, .
\end{align*}

\subsection{Proof of~\texorpdfstring{\Cref{clm:max}}{}}

We show $H_y(\tmax_1, \tmax_2) = \tmax_y$ for each $y \in \{1,2\}$. Since $\psi_y$ is a PGF,
\[
    \psi''_y(s) = \sum_{k \geq 2} k(k-1) \cdot \piR(k) \cdot s^{k-2} \, ,
\]
so $\psi_y'' \geq 0$ on $[0,1]$ and $\psi_y''(s) > 0$ for all $s > 0$, provided $\piR(k \geq 2) > 0$. We exclude the degenerate case $\piR(k \geq 2) = 0$. Note also that $H_y$ is a convex combination of values in $[0,1]$, so $H_y \in [0,1]$. We consider two cases based on the positivity of $\psi''_y(1 - \tmax_y)$.
\begin{itemize}
    \item Case 1: $\psi''_y(1 - \tmax_y) > 0$. From~\eqref{eq:derivative-formula}, the sign of $\frac{\partial F}{\partial t_y}$ at $(\tmax_1, \tmax_2)$ equals the sign of $H_y(\tmax_1, \tmax_2) - \tmax_y$. Since $(\tmax_1, \tmax_2)$ maximizes $F$ over $[0,1]^2$, the first-order optimality conditions require: $\frac{\partial F}{\partial t_y} = 0$ if $\tmax_y \in (0,1)$, \ $\frac{\partial F}{\partial t_y} \leq 0$ if $\tmax_y = 0$, and $\frac{\partial F}{\partial t_y} \geq 0$ if $\tmax_y = 1$. In each subcase, since $H_y \in [0,1]$, we conclude $H_y(\tmax_1, \tmax_2) = \tmax_y$.
    \item Case 2: $\psi''_y(1 - \tmax_y) = 0$. Since $\psi_y''(s) > 0$ for $s > 0$, this case requires $\tmax_y = 1$ (and $\psi_y''(0) = 2\piR(2) = 0$). We show that $H_y(\tmax_1, \tmax_2) = 1$. Assume for the sake of contradiction that $H_y(\tmax_1, \tmax_2) = 1 - \delta$ for some $\delta \in (0,1)$. Define for $\varepsilon \in (0,1)$ the vector
    \[
        t^{(\varepsilon)} \triangleq
        \begin{cases}
            (1-\varepsilon, \,  \tmax_2),  & \mbox{if $y=1$} \\
            ( \tmax_1, 1-\varepsilon), & \mbox{if $y=2$}
        \end{cases}
        \, .
    \]
    Since $H_y$ is continuous, there exists $\varepsilon_0\in(0,1)$ such that $H_y(t^{(\varepsilon)}) \leq 1-\delta/2$ for all $\varepsilon\in(0,\varepsilon_0)$.
    Choose $\varepsilon\in(0,\min\{\varepsilon_0,\delta/2\})$. Then
    \[
        H_y(t^{(\varepsilon)})-t^{(\varepsilon)}_y \,\le\, (1-\delta/2)-(1-\varepsilon)
        \, = \,
        \varepsilon-\delta/2
        \, < \,
        0 \, .
    \]
    Moreover, $\psi_y''\big(1-t^{(\varepsilon)}_y\big)=\psi_y''(\varepsilon)>0$ since $\varepsilon > 0$. Plugging into~\eqref{eq:derivative-formula}, we get $\frac{\partial F}{\partial t_y}(t^{(\varepsilon)}) < 0$ for all sufficiently small $\varepsilon$. In particular, increasing the $y$-th coordinate from $1-\varepsilon$ to $1$ strictly decreases $F$, contradicting maximality of $( \tmax_1, \tmax_2)$. Hence $H_y( \tmax_1, \tmax_2)=1$.
\end{itemize}
In both cases, $H_y( \tmax_1, \tmax_2) = \tmax_y$. Since $y$ was arbitrary, $H( \tmax_1, \tmax_2) = (\tmax_1, \tmax_2)$. This completes the proof of~\Cref{clm:max}.

\subsection{Proof of~\texorpdfstring{\Cref{clm:const}}{}}

For the base case $k = 0$, we have by construction,
\[
    t_y^{(0)}
    = \sum_{x=1}^{2} \aR_{yx}\,\mu_x^{(0)}\left((0,1]\right)
    = \sum_{x=1}^{2} \aR_{yx}\,\smax_x
    = H_y(\tmax_1, \tmax_2)
    = \tmax_y \, ,
\]
so $(t_1^{(0)},t_2^{(0)}) = ( \tmax_1, \tmax_2) $. For the inductive step, suppose $(t_1^{(k)},t_2^{(k)}) = (\tmax_1, \tmax_2) $. By definition,
\(
    t_y^{(k+1)} = \sum_{x=1}^{2} \aR_{yx}\,\mu_x^{(k+1)} \left((0,1]\right) \, .
\)
Now $\mu_x^{(k+1)}((0,1]) = \P(Y_x > 0)$, where $Y_x\sim \Theta_{x; \boldsymbol{\pihatL},\boldsymbol{ \pihatR }}(\boldsymbol{\boldsymbol{{\mu}}}^{(k)})$.
With the notations introduced in~\eqref{eq:Thetadef},
\begin{align*}
    \mathbb{P}(Y_x>0) = \mathbb{P} \left( \forall i \in [N_x] , \,   \sum_{j=1}^{N'_{i}} X_{ij}>0
    \right)
    =
    \widehat\phi_x(r_x(\tmax_1, \tmax_2))
    =
    \smax_x \, .
\end{align*}
Therefore,
\[
    t_y^{(k+1)} = \sum_{x=1}^{2} \aR_{yx}\, \smax_x = H_y(\tmax_1, \tmax_2) = \tmax_y \, .
\]
Hence, $(t_1^{(k)},t_2^{(k)}) = ( \tmax_1, \tmax_2)$ for all $k \geq 0$. This concludes the proof.

\section{Supplementary proofs for \prettyref{sec:dominance}} \label{app:limiting}

This section presents the proof of the limiting matching rates as $\alpha_f \to \infty$, for the case of one-sided flexibility (\Cref{prop:oneside}) and two-sided flexibility (\Cref{prop:twoside}).

\subsection{Proof of~\texorpdfstring{\Cref{prop:oneside}}{}}

Let $y = e^{-Mt_1}$ with $M = 2\alpha + (\alpha_f-\alpha)B$. Under this change of variables, $t_1\in[0,1]$ maps bijectively to $y\in[e^{-M},1]$, and the objective becomes
\begin{align*}
    \widetilde{F}_{\OS}(y) 
    = (1-B)e^{-2\alpha y} + y(1-\ln y) - 1 
      + Be^{-(\alpha+\alpha_f)y} \, .
\end{align*}

First,
we show that any maximizer of $\widetilde{F}_{\OS}(y)$ over $[e^{-M},1]$ lies in a compact interval $[\varepsilon,1]$ for some $\varepsilon>0$ independent of $\alpha_f$. Evaluating at $y=1$:
\begin{align} \label{eq:FS-1}
    \widetilde{F}_{\OS}(1) 
    = (1-B)e^{-2\alpha} + Be^{-(\alpha+\alpha_f)} 
    > (1-B)e^{-2\alpha} > 0.
\end{align}
For any $y\in(0,1]$, we have $e^{-2\alpha y}\leq 1$ and $e^{-(\alpha+\alpha_f)y}\leq 1$, and so we have the upper bound
\begin{align} \label{eq:FS-2}
    \widetilde{F}_{\OS}(y) 
    \leq (1-B) + B + y(1-\ln y) - 1 
    = y(1-\ln y) \, .
\end{align}
Now, $h(y) \triangleq y(1-\ln y)$ is continuous and strictly increasing on $(0,1)$ with $h(0)=0$ and $h(1)=1>(1-B)e^{-2\alpha}$. Therefore, there exists a unique $\varepsilon\in(0,1)$ depending only on $\alpha$ and $B$ (and not on $\alpha_f$) such that $h(\varepsilon)=(1-B)e^{-2\alpha}$. By strict monotonicity, $h(y)<(1-B)e^{-2\alpha}$ for all $y\in(0,\varepsilon)$, so $y(1-\ln y)\leq(1-B)e^{-2\alpha}$ for all $y\in(0,\varepsilon]$. Hence, combining with~\eqref{eq:FS-1} and~\eqref{eq:FS-2},
\[
    \widetilde{F}_{\OS}(y)\leq y(1-\ln y)\leq (1-B)e^{-2\alpha}<\widetilde{F}_{\OS}(1) ~~~~ \mbox{ for all } y\in(0,\varepsilon) \,, 
\]
so any maximizer of $\widetilde{F}_{\OS}(y)$ lies in $[\varepsilon,1]$. Since $e^{-M}\to 0$ as $\alpha_f\to\infty$, we have $e^{-M}<\varepsilon$ for all sufficiently large $\alpha_f$, so $[\varepsilon,1]\subset[e^{-M},1]$ and all maximizers lie in the fixed compact set $[\varepsilon,1]$.

Finally, note that the only $\alpha_f$-dependent term of $\widetilde{F}_{\OS}$ is $Be^{-(\alpha+\alpha_f)y}$, which satisfies 
\[
    \Big| Be^{-(\alpha+\alpha_f)y} \Big|
    \leq Be^{-(\alpha+\alpha_f)\varepsilon}
    \to 0
\]
uniformly on $[\varepsilon,1]$ as $\alpha_f\to\infty$. Hence $\widetilde{F}_{\OS}\to\Phi_{\OS}$ uniformly on $[\varepsilon,1]$, and by uniform convergence on a compact set, it follows that
\begin{align} \label{eq:unif-conv}
    \lim_{\alpha_f\to\infty}
    \max_{y\in[\varepsilon,1]} \widetilde{F}_{\OS}(y)
    = \max_{y\in[\varepsilon,1]} \Phi_{\OS}(y) \, .
\end{align}
Chaining the results:
\begin{align*}
    \lim_{\alpha_f\to\infty} \, 
    \max_{t_1\in[0,1]}F_{\OS}(t_1)
    & \stepa{=} \lim_{\alpha_f\to\infty} \, 
       \max_{y\in[e^{-M},1]}\widetilde{F}_{OS}(y) \\
    & \stepb{=} \lim_{\alpha_f\to\infty} \, 
    \max_{y\in[\varepsilon,1]}\widetilde{F}_{\OS}(y) \\
    & \stepc{=} \max_{y\in[\varepsilon,1]}\Phi_{\OS}(y) \\
    & \stepd{=} \max_{y\in(0,1]}\Phi_{\OS}(y) \, .
\end{align*}
Here, (a) uses the bijection $t_1\in[0,1] \leftrightarrow y\in[e^{-M},1]$. Further, (b) holds because, for all sufficiently large $\alpha_f$, the maximizer of $\widetilde{F}_{\OS}$ over $[e^{-M},1]$ lies in $[\varepsilon,1]$, so the two maxima coincide for all large $\alpha_f$ and their limits are equal. The equality (c) uses~\eqref{eq:unif-conv}, and (d) holds because for any $y\in(0,\varepsilon]$,
\begin{align}\label{eq:maxboundedaway0TS}
    \Phi_{\OS}(y) 
    \leq y \,(1-\ln y) 
    \leq (1-B)e^{-2\alpha} 
    = \Phi_{\OS}(1)
    \leq \max_{y\in[\varepsilon,1]}\Phi_{\OS}(y) \, ,
\end{align}
where the second inequality uses $h(y)\leq(1-B)e^{-2\alpha}$ on $(0,\varepsilon]$ as established. Hence the maximum over $(0,\varepsilon]$ cannot exceed the maximum over $[\varepsilon,1]$, giving $\max_{(0,1]}\Phi_{\OS}= \max_{[\varepsilon,1]}\Phi_{\OS}$.

\subsection{Proof of \prettyref{prop:twoside}}

With the change of variable $y_1 = e^{-M_1(\alpha_f) t_1}$ and $y_2 = e^{-M_2(\alpha_f) t_2}$, where $M_1(\alpha_f) = \left(2-\frac{B}{2}\right)\alpha + \frac{B}{2}\alpha_f$ and $M_2(\alpha_f) = \left(1-\frac{B}{2}\right)\alpha + \left(1+\frac{B}{2}\right)\alpha_f$, the two-sided objective becomes:
\begin{align}
    \begin{aligned} \label{eq: FTS-y}
   \widetilde{F}_{\TS}(y_1, y_2) 
   &= \left(1-\frac{B}{2}\right)
      e^{-2\alpha(1-B/2)y_1}\, e^{-(\alpha+\alpha_f)(B/2)y_2} \\
   &\quad + \frac{B}{2}\,
      e^{-(\alpha+\alpha_f)(1-B/2)y_1}\, e^{-\alpha_f B y_2} \\
   &\quad + \left(1-\frac{B}{2}\right)y_1(1-\ln y_1) 
      + \frac{B}{2}y_2(1-\ln y_2) - 1,
      \end{aligned}
\end{align}
on the domain $\mathcal{D}(\alpha_f) \triangleq [e^{-M_1(\alpha_f)},1]\times[e^{-M_2(\alpha_f)},1]$. We define $h(y) \triangleq y(1-\ln y)$, and we use the convention $h(0) \triangleq  \lim_{y\to 0^+} y(1-\ln y) = 0$ so that $h(y)$ is continuous on $[0,1]$. We now establish the limit via $\liminf$ and $\limsup$ bounds.

\noindent \textit{Lower Bound $(\liminf)$:} Fix any $y_1\in(0,1]$  and set $y_2(\alpha_f) = e^{-\alpha_f(1+B/2)}$.
Since $\alpha_f(1+B/2)\leq M_2(\alpha_f)$ for $\alpha\geq 0$ and $B\in(0,1)$, we have $y_2(\alpha_f)\in[e^{-M_2(\alpha_f)},1]$. 
Since $e^{-M_1(\alpha_f)}\to 0$ as $\alpha_f\to\infty$, we have for any fixed $y_1>0$ that $y_1\geq e^{-M_1(\alpha_f)}$ for all sufficiently large $\alpha_f$. Thus, $(y_1,y_2(\alpha_f))\in\mathcal{D}(\alpha_f)$ for sufficiently large $\alpha_f$.

As $\alpha_f\to\infty$, it follows that $y_2(\alpha_f)\to 0$ and $\alpha_f \,y_2(\alpha_f) = \alpha_f \, e^{-\alpha_f(1+B/2)} \to 0$ (since $xe^{-cx}\to 0$ for $c>0$), so the exponential coupling terms in~\eqref{eq: FTS-y} satisfy
\[
    e^{-(\alpha+\alpha_f)(B/2)y_2(\alpha_f)}\to 1 
    ~~~~ \mbox{ and } ~~~~
    e^{-\alpha_f B y_2(\alpha_f)}\to 1 \, .
\]
Further, the second term in~\eqref{eq: FTS-y} vanishes since $(\alpha+\alpha_f)(1-B/2)y_1\to\infty$ for $y_1>0$. The fourth term $y_2(\alpha_f)(1-\ln y_2(\alpha_f))\to 0$ since $\lim_{y\to 0^+}y(1-\ln y) = 0$. Collecting terms,
\[
    \lim_{\alpha_f\to\infty}
    \widetilde{F}_{\TS} \left(y_1,\,y_2(\alpha_f) \right)
    = \Phi_{\TS}(y_1).
\]
Taking the supremum over $y_1\in(0,1]$ gives $\liminf_{\alpha_f\to\infty}\max_{\mathcal{D}(\alpha_f)} \widetilde{F}_{\TS} \geq 
\sup_{y\in(0,1]}\Phi_{\TS}(y)$.
Additionally, evaluating at $y_1=y_2=1$: the two terms bearing $\alpha_f$ in their exponents vanish, since $e^{-(\alpha+\alpha_f)(B/2)}\to 0$ and $e^{-(\alpha+\alpha_f)(1-B/2)}e^{-\alpha_f B}\to 0$, while the factor $e^{-2\alpha(1-B/2)}$ in the first term is a fixed constant. The polynomial terms sum to $(1-B/2)\cdot 1+(B/2)\cdot 1-1=0$. Therefore $\widetilde{F}_{\TS}(1,1)\to 0$. Combined:
\begin{align} \label{eq:liminf-bound}
    \liminf_{\alpha_f\to\infty}
    \max_{\mathcal{D}(\alpha_f)}
    \widetilde{F}_{\TS}
    \, \geq \, 
    \max\left(0,\,\sup_{y\in(0,1]}
    \Phi_{\TS}(y)\right).
\end{align}

\noindent \textit{Upper Bound $(\limsup)$:} Let $M_n = \max_{\mathcal{D}(\alpha_{f,n})} \widetilde{F}_{\TS}(\,\cdot\,, \alpha_{f,n})$ for a sequence $\alpha_{f,n}\to\infty$. Consider any subsequence along which $M_{n_k}$ converges to the $\limsup$, with corresponding maximizers $(y_{1,k},y_{2,k})\in\mathcal{D}(\alpha_{f,n_k}) \subset[0,1]^2$. By compactness of $[0,1]^2$, pass to a further subsequence with $(y_{1,k},y_{2,k})\to(u,v)\in[0,1]^2$.
We examine the limit by cases.
\begin{itemize}
    \item Case 1: $v > 0$.
    Since $y_{2,k}\to v>0$, for large $k$ we have $y_{2,k}\geq v/2>0$, so $\alpha_{f,n_k}y_{2,k}\geq\alpha_{f,n_k}v/2\to\infty$. Hence both coupling exponentials satisfy $e^{-(\alpha+\alpha_f)(B/2)y_{2,k}}\to 0$ and $e^{-\alpha_f B y_{2,k}}\to 0$, so the first two terms of $\widetilde{F}_{\TS}$ vanish. By continuity of $h$ on $[0,1]$:
    \[
        \lim_{k\to\infty}M_{n_k}
        = \left(1-\frac{B}{2}\right) \,u \, (1-\ln u)
          +\frac{B}{2} \, v \, (1-\ln v)-1
        \,\leq\, 
        \left(1-\frac{B}{2}\right)+\frac{B}{2}-1=0,
    \]
    using $x(1-\ln x)\leq 1$ for all $x\in[0,1]$.

    \item Case 2: $v = 0$.
    \begin{itemize}
        \item {Subcase $2$a: $u > 0$.}
        The second term vanishes since $(\alpha+\alpha_f)(1-B/2)y_{1,k}\to\infty$. The fourth term $y_{2,k}(1-\ln y_{2,k})\to 0$ by continuity of $h$ at $0$. Bounding $e^{-(\alpha+\alpha_f)(B/2)y_{2,k}} \leq 1$:
        \[
            \lim_{k\to\infty}M_{n_k}
            \,\leq\, 
            \left(1-\frac{B}{2}\right)
            e^{-2\alpha(1-B/2)u}
            +\left(1-\frac{B}{2}\right) u(1-\ln u)-1
            = \Phi_{\TS}(u)
            \,\leq\,\sup_y\Phi_{\TS}(y).
        \]

        \item {Subcase $2$b: $u = 0$, $v = 0$.}
        All polynomial terms vanish by $h(0)=0$. Bounding the exponentials by their coefficients:
        \[
            \lim_{k\to\infty}M_{n_k}
            \,\leq\, 
            \left(1-\frac{B}{2}\right)
            +\frac{B}{2}-1=0.
        \]
    \end{itemize}
\end{itemize}

In all cases the limit of the subsequence is upper bounded by $\max(0,\sup_{y}\Phi_{\TS}(y))$, and so $\limsup_{n\to\infty}M_n \leq \max(0,\sup_y\Phi_{\TS}(y))$. Combining with the lower bound~\eqref{eq:liminf-bound} establishes the limit. It remains to verify that the supremum over $(0,1]$ is attained. Since
\[
    \lim_{y\to 0^+}\Phi_{\TS}(y) 
    = \left(1-\frac{B}{2}\right)[1 + 0] - 1 
    = -\frac{B}{2} < 0 \, ,
\]
we have for small enough $\varepsilon>0$ that $\sup_{y\in(0,\varepsilon]}\Phi_{\TS}(y)\leq -B/4 < 0$.
If $\sup_{y\in(0,1]}\Phi_{\TS}(y)>0$, then the supremum is not achieved in $(0,\varepsilon]$ and equals the maximum over the compact set $[\varepsilon,1]$, which is attained by continuity. 
If $\sup_{y\in(0,1]}\Phi_{\TS}(y)\leq 0$, then $\max(0,\sup_y\Phi_{\TS}(y))=0$ regardless of  whether the supremum is attained. In both cases the formula holds with $\max$ in place of $\sup$, i.e.
\begin{align*}
    \lim_{\alpha_f \to \infty} 
    \max_{t_1, t_2 \in [0,1]} F_{\TS}(t_1, t_2) 
    = \max\left(0,\,
      \max_{y \in (0,1]} \Phi_{\TS}(y)\right).
\end{align*}
This concludes the proof.

\section{Details on~\texorpdfstring{\Cref{rmk:Thm3}}{}} \label{app:comp}
\cite[Theorem~3]{freund2024twov3} establishes two-sided dominance
for $B\in(0,1)$ whenever $0<\alpha<\alpha_\star(B)$ and
$\alpha_f>\alpha_f^\star(B,\alpha)$, where
\begin{align}
   \begin{aligned}\label{eq:fmz}
           \alpha_\star(B)
    &\triangleq
    \min \left\{
    \frac{B^2}{8(1-B/2)^3},\,
    \frac{1}{2(1-B/2)}\ln\frac{2-B}{B}
    \right\} \\[0.3em]
    \alpha_f^\star(B,\alpha)
    &\triangleq
    \frac{\ln B - \ln \left(2\alpha\bigl[
    (B/2)^2 - 2\alpha(1-B/2)^3\bigr]\right)}
    {(1-B/2)e^{-2\alpha(1-B/2)} - B/2}.
       \end{aligned}
\end{align}
The proof of~\cite[Theorem~3]{freund2024twov3} lower-bounds
$\match_{\TS}$ by
\begin{equation}\label{eq:FMZ-lb}
    L_{\mathrm{FMZ}}(B,\alpha,\alpha_f)
    \,\triangleq\,
    m_{\mathrm{reg}}^{\mathrm{FMZ}}+B-C_{\mathrm{FMZ}},
\end{equation}
where
\begin{align}
     m_{\mathrm{reg}}^{\mathrm{FMZ}}
    & \triangleq 
    2 \left(1-\frac{B}{2} \right)
     \left(1- \left(1-\frac{B}{2} \right)\alpha-e^{-2\alpha(1-B/2)} \right)  \label{eq:reg}\\
    C_{\mathrm{FMZ}}
    & \triangleq
    \frac{2}{\alpha_f}
     \left(e^{\alpha_f B/2}-1 \right)
    e^{-\alpha_f(1-B/2)e^{-2\alpha(1-B/2)}}, \label{eq:C_fmz}
\end{align} 
and upper-bounds $\match_{\OS}$ by
\[
    U_{\mathrm{FMZ}}(B,\alpha)\triangleq 1-(1-B)e^{-2\alpha}.
\]
Two-sided dominance follows whenever
$L_{\mathrm{FMZ}}>U_{\mathrm{FMZ}}$, so the admissible  regime for $\alpha$ and $\alpha_f$ is determined by the strength of $L_{\mathrm{FMZ}}$.

The lower bound in~\cite[Theorem 3]{freund2024twov3} is obtained by lower bounding the number of nodes matched in the following process: first, a maximum matching is constructed among regular nodes. The leftover nodes are then greedily matched to flexible nodes. The following proposition demonstrates how the bound $\match_{\TS}\geq L_{\mathrm{FMZ}}$ can be recovered directly from the variational formula, circumventing the need to analyze algorithmic bounds.
\begin{proposition}\label{prop:TS_lower}
For all $B\in(0,1)$, $\alpha\geq 0$, and $\alpha_f>0$:
$\match_{\TS}(B,\alpha,\alpha_f)\geq L_{\mathrm{FMZ}}(B,\alpha,\alpha_f)$. 
\end{proposition}

The proof of \prettyref{prop:TS_lower} is deferred to Appendix~\ref{sec:TS_lower}. Next, we introduce the following Lemma~\ref{lmm:OS-upper-bound} which retains
the bound $\match_{\OS}\leq U_{\mathrm{FMZ}}$.

\begin{lemma}[Upper bound on $\match_{\OS}$]
\label{lmm:OS-upper-bound}
For any $B\in(0,1)$, $\alpha\geq 0$, and $\alpha_f\geq 0$,
$\match_{\OS}(B,\alpha,\alpha_f)\leq U_{\mathrm{FMZ}}(B,\alpha)
=1-(1-B)e^{-2\alpha}$.
\end{lemma}

\begin{proof}
By Corollary~\ref{cor:matching_size_one},
$\match_{\OS}=1-\max_{t_1\in[0,1]}F_{\OS}(t_1)$.
Evaluating at $t_1=0$:
\[
    \max_{t_1\in[0,1]}F_{\OS}(t_1)
    \,\geq\,F_{\OS}(0)
    \,=\,(1-B)e^{-2\alpha}+Be^{-(\alpha+\alpha_f)}
    \,\geq\,(1-B)e^{-2\alpha},
\]
so $\match_{\OS}\leq 1-(1-B)e^{-2\alpha}$.
\end{proof}

Together with~\Cref{prop:TS_lower} and~\Cref{lmm:OS-upper-bound}, the admissible $\alpha_f$ regime of~\cite[Theorem~3]{freund2024twov3} is recovered; we omit the explicit threshold for brevity.

\subsection{Proof of~\texorpdfstring{\Cref{prop:TS_lower}}{}}\label{sec:TS_lower}
Define
\begin{align}\label{eq:Gammadef}
    \Gamma
    \triangleq
    \Big(1-\frac{B}{2}\Big)\,\max_{y\in(0,1]}
    \Big\{e^{-2\alpha(1-B/2)y}+y(1-\ln y)-1\Big\},
\end{align}
and let $M'_1\triangleq\frac{B}{2}(\alpha+\alpha_f)$,~
$M'_2\triangleq \Gamma \cdot (\alpha+\alpha_f)$, and
\begin{align}\label{eq:FCross_def}
    F_{\mathrm{Cross}}(t_1,t_2)
    &= \Gamma\exp \Big({-M'_1\,e^{-M'_2 t_2}}\Big)
     + \frac{B}{2}\exp \Big({-M'_2\,e^{-M'_1 t_1}}\Big) \nonumber\\
    &\quad + \Gamma\,e^{-M'_1 t_1}(1+M'_1 t_1)
           + \frac{B}{2}\,e^{-M'_2 t_2}(1+M'_2 t_2)
           - \Big(\Gamma+\frac{B}{2}\Big).
\end{align}
Recall also by Theorem~\ref{thm:matching-rate} that
\[
    \match_{\TS}=1-\max_{t_1,t_2\in[0,1]}F_{\mathrm{TS}}(t_1,t_2)\,.
\]

We first prove the following lemma that bounds the two-sided objective function $F_\TS$ using $F_{\mathrm{Cross}}$ at their respective optimizers.

\begin{lemma}\label{lmm:dominationresult}
    We have
    \[
        \max_{t_1,t_2\in[0,1]} F_{\mathrm{TS}}(t_1,t_2)
        \,\leq\,
        \max_{t_1,t_2\in[0,1]} F_{\mathrm{Cross}}(t_1,t_2).
    \]
\end{lemma}

To evaluate the right-hand side, we introduce
\begin{align}\label{eq:Lambdadef}
    \Lambda \triangleq \max_{v\in(0,\Gamma]}
    \Bigl[\frac{B}{2}e^{-(\alpha+\alpha_f)v}
    + v \left(1-\ln\frac{v}{\Gamma} \right)-\Gamma\Bigr],
\end{align}
and establish the following identity.

\begin{lemma}\label{lmm:Lambdaid}
    We have
    \[
        \max_{t_1,t_2\in[0,1]}
        F_{\mathrm{Cross}}(t_1,t_2)
        \,=\,
        \Gamma - \frac{B}{2} + 2\Lambda.
    \]
\end{lemma}

It thus remains to bound $\Gamma$ and $\Lambda$. We first bound
$\Gamma$.

\begin{lemma}\label{lmm:upperboundGamma}
    We have
    \[
        \Big(1-\frac{B}{2}\Big)e^{-2\alpha(1-B/2)}
        \, \leq \,
        \Gamma
        \, \leq \, 
        1-\frac{B}{2}-m_{\mathrm{reg}}^{\mathrm{FMZ}}(B,\alpha),
    \]
    where $m_{\mathrm{reg}}^{\mathrm{FMZ}}(B,\alpha)
   $ is defined in \prettyref{eq:reg}.
\end{lemma}

Leveraging the lower bound on $\Gamma$ from
Lemma~\ref{lmm:upperboundGamma}, we obtain the following upper
bound on $\Lambda$.

\begin{lemma}\label{lmm:tentative}
    We have
    \[
        2\Lambda \,\leq\, C_{\mathrm{FMZ}}(B,\alpha,\alpha_f),
    \]
    where $C_{\mathrm{FMZ}}(B,\alpha,\alpha_f)$ is defined in \prettyref{eq:C_fmz}.
\end{lemma}

Combining Lemmas~\ref{lmm:dominationresult}, \ref{lmm:Lambdaid} and \ref{lmm:tentative}
together with the upper bound on $\Gamma$ from
Lemma~\ref{lmm:upperboundGamma} gives
\[
    \max_{t_1,t_2\in[0,1]} F_{\mathrm{TS}}(t_1,t_2)
    \,\leq\, \Gamma - \frac{B}{2} + C_{\mathrm{FMZ}}
    \,\leq\, 1 - B - m_{\mathrm{reg}}^{\mathrm{FMZ}} + C_{\mathrm{FMZ}}
    \,=\, 1 - L_{\mathrm{FMZ}},
\]
which completes the proof.

\subsubsection{Proof of~\texorpdfstring{\Cref{lmm:dominationresult}}{}}

We introduce the monotonic change of variables $y_i = e^{-M_i t_i}$
for $F_{\TS}$ and $z_i = e^{-M'_i t_i}$ for $F_{\mathrm{Cross}}$.
By a domain-extension argument analogous to~\eqref{eq:domainextension},
we may work on $[0,1]^2$ throughout. The two objectives become
\begin{align}\label{eq:FTS-y}
    F_{\TS}(y_1, y_2)
    &=  \Big(1-\frac{B}{2} \Big)
       \exp  \Big({-2\alpha \Big(1-\frac{B}{2} \Big)y_1
       -(\alpha+\alpha_f)\frac{B}{2}y_2} \Big) \nonumber\\[0.3em]
    &\quad + \frac{B}{2}
       \exp  \Big({-(\alpha+\alpha_f) \Big(1-\frac{B}{2} \Big)y_1
       - \alpha_f B y_2} \Big) \nonumber\\[0.3em]
    &\quad +  \Big(1-\frac{B}{2} \Big)y_1(1-\ln y_1)
           + \frac{B}{2}y_2(1-\ln y_2) - 1,
\end{align}
and
\begin{align}\label{eq:FCross-z}
    F_{\mathrm{Cross}}(z_1, z_2)
    &= \Gamma\exp  \left({-\frac{B}{2}(\alpha+\alpha_f)z_2} \right)
     + \frac{B}{2}\exp  \left({-\Gamma(\alpha+\alpha_f)z_1} \right) \nonumber\\
    &\quad + \Gamma\,z_1(1-\ln z_1)
           + \frac{B}{2}\,z_2(1-\ln z_2)
           -  \left(\Gamma + \frac{B}{2} \right).
\end{align}
The lemma follows from the following claim by optimizing both sides
over $y\in[0,1]$.

\begin{claim}\label{clm:fixydom}
    For every fixed $y\in[0,1]$,
    \[
        \max_{y_1\in[0,1]} F_{\TS}(y_1,y)
        \,\leq\,
        \max_{z_1\in[0,1]} F_{\mathrm{Cross}}(z_1,y).
    \]
\end{claim}

\begin{proof}[Proof of~\Cref{clm:fixydom}]
Introduce $\beta\triangleq e^{-(\alpha+\alpha_f)(B/2)y}$ and
\begin{align}\label{eq:Adef}
    A(y_1)
    &\triangleq
    \beta \left(1-\frac{B}{2} \right)
    \Bigl[e^{-2\alpha(1-B/2)y_1} + y_1(1-\ln y_1) - 1\Bigr]
    + \frac{B}{2}e^{-(\alpha+\alpha_f)(1-B/2)y_1} \nonumber\\
    &\quad + (1-\beta) \left(1-\frac{B}{2} \right)\bigl[y_1(1-\ln y_1)-1\bigr].
\end{align}
Setting $y_2=y$ in~\eqref{eq:FTS-y}, factoring $\beta$ from the
first exponential, bounding $e^{-\alpha_f By}\leq 1$ in the second,
writing $-1=-(1-B/2)-(B/2)$, and splitting
$(1-B/2)[y_1(1-\ln y_1)-1]$ as $\beta(\cdot)+(1-\beta)(\cdot)$,
we obtain
\[
    F_{\TS}(y_1,y) \,\leq\, A(y_1) + \frac{B}{2}y(1-\ln y) - \frac{B}{2}.
\]
Similarly, defining
\begin{align}\label{eq:Bdef}
    B(z_1)
    \,\triangleq\,
    \Gamma\beta
    + \frac{B}{2}\exp  \left({-\Gamma(\alpha+\alpha_f)z_1} \right)
    + \Gamma\bigl[z_1(1-\ln z_1)-1\bigr],
\end{align}
and using $\Gamma e^{-(B/2)(\alpha+\alpha_f)y}=\Gamma\beta$,
we have exactly
$F_{\mathrm{Cross}}(z_1,y)=B(z_1)+\frac{B}{2}y(1-\ln y)-\frac{B}{2}$.
The claim reduces to $\max_{y_1}A(y_1)\leq\max_{z_1}B(z_1)$,
which we establish by splitting into two cases.

\paragraph{Case 1: $y_1\geq\Gamma/(1-B/2)$.}
From~\eqref{eq:Adef}, using the bound $\leq\Gamma$
from~\eqref{eq:Gammadef} on the first bracket, monotonicity of
$e^{-(\alpha+\alpha_f)\cdot}$ with $(1-B/2)y_1\geq\Gamma$ on the
second term, and $1-\beta\geq 0$ with $y_1(1-\ln y_1)-1\leq 0$
for $y_1\in(0,1]$ (since $g(y)\triangleq y(1-\ln y)$ has
$g'(y)=-\ln y\geq 0$ with $g(1)=1$) on the third:
\[
    A(y_1)
    \,\leq\, \beta\Gamma + \frac{B}{2}e^{-(\alpha+\alpha_f)\Gamma}
    \,=\, B(1)
    \,\leq\, \max_{z_1}B(z_1),
\]
where the equality uses $\Gamma[1(1-\ln 1)-1]=0$.

\paragraph{Case 2: $y_1\leq\Gamma/(1-B/2)$.}
Note $\Gamma>0$ from~\eqref{eq:Gammadef} since
$\Gamma\geq(1-B/2)e^{-2\alpha(1-B/2)}>0$.
Set $z_1\triangleq\frac{1-B/2}{\Gamma}y_1\in[0,1]$ and
$H(y_1)\triangleq A(y_1)-B(z_1)$. Since
$\Gamma(\alpha+\alpha_f)z_1=(\alpha+\alpha_f)(1-B/2)y_1$,
the exponential terms $\frac{B}{2}e^{-(\alpha+\alpha_f)(1-B/2)y_1}$
in $A(y_1)$ and $\frac{B}{2}e^{-\Gamma(\alpha+\alpha_f)z_1}$
in $B(z_1)$ cancel. Collecting the remaining terms gives
\[
    H(y_1)
    = H_0(y_1)+\beta\Bigl[(1-B/2)e^{-2\alpha(1-B/2)y_1}-\Gamma\Bigr]
    = (1-\beta)H_0(y_1) + \beta H_1(y_1),
\]
where the second equality uses $H_1-H_0=(1-B/2)e^{-2\alpha(1-B/2)y_1}-\Gamma$, and
\begin{align*}
    H_0(y_1)
    &\triangleq  \Big(1-\frac{B}{2} \Big) \, y_1\ln\frac{1-B/2}{\Gamma}
    + \Gamma -  \Big(1-\frac{B}{2} \Big), \\[0.3em]
    H_1(y_1)
    &\triangleq  \Big(1-\frac{B}{2} \Big) \,
    e^{-2\alpha(1-B/2)y_1}
    +  \Big(1-\frac{B}{2} \Big) \, y_1\ln\frac{1-B/2}{\Gamma}
    -  \Big(1-\frac{B}{2} \Big).
\end{align*}
We show $H_0\leq 0$ and $H_1\leq 0$ separately.

\paragraph{Bound on $H_0$.}
Since $\ln\frac{1-B/2}{\Gamma}\geq 0$, $H_0$ is increasing in
$y_1$. Setting $x\triangleq\Gamma/(1-B/2)\in(0,1]$ and using
$y_1\leq\Gamma/(1-B/2)$:
\[
    H_0(y_1)
    \,\leq\, \Gamma\ln\frac{1-B/2}{\Gamma} + \Gamma -  \left(1-\frac{B}{2} \right)
    \,=\,  \left(1-\frac{B}{2} \right)\Bigl[x\ln\tfrac{1}{x}+x-1\Bigr]
    \,\leq\, 0,
\]
where $f(x)\triangleq x\ln(1/x)+x-1\leq 0$ for $x\in(0,1]$
since $f'(x)=-\ln x\geq 0$ with $f(1)=0$.

\paragraph{Bound on $H_1$.}
If $\alpha=0$, then $\Gamma=(1-B/2)\max_{y\in(0,1]}\{y(1-\ln y)\}
=1-B/2$, so $H_1(y_1)=0$.
For $\alpha>0$, $H_1''(y_1)=4\alpha^2(1-B/2)^3
e^{-2\alpha(1-B/2)y_1}>0$, so $H_1$ is convex and its maximum
on $[0,\Gamma/(1-B/2)]$ is at an endpoint. At $y_1=0$:
$H_1(0)=0$. At $y_1=\Gamma/(1-B/2)$, setting $x=\Gamma/(1-B/2)$:
\[
    H_1  \left(\frac{\Gamma}{1-B/2} \right)
    = \underbrace{ \left(1-\frac{B}{2} \right)
      \Bigl[e^{-2\alpha(1-B/2)x}+x(1-\ln x)-1\Bigr]}_{\leq\,\Gamma
      \text{ by~\eqref{eq:Gammadef}}}
      - \underbrace{ \left(1-\frac{B}{2} \right)x}_{=\,\Gamma}
    \, \leq \, 0.
\]
Since $H(y_1)\leq 0$ for all $y_1\in[0,\Gamma/(1-B/2)]$,
we have $A(y_1)\leq B \left(\tfrac{1-B/2}{\Gamma}y_1\right)
\leq\max_{z_1}B(z_1)$. Combining with Case~1 completes the proof.
\end{proof}

\subsubsection{Proof of~\texorpdfstring{\Cref{lmm:Lambdaid}}{}}

With the change of variables $v_1=\Gamma e^{-M'_1 t_1}\in(0,\Gamma]$
and $v_2=\frac{B}{2}e^{-M'_2 t_2}\in(0,B/2]$ (where the domains
follow from the same extension argument as~\eqref{eq:domainextension}),
substituting $z_1=v_1/\Gamma$ and $z_2=v_2/(B/2)$
into~\eqref{eq:FCross-z} gives
\begin{align*}
    F_{\mathrm{Cross}}(v_1,v_2)
    &= \Gamma e^{-(\alpha+\alpha_f)v_2}
       + v_2 \left(1-\ln\frac{v_2}{B/2} \right)
       + \frac{B}{2}e^{-(\alpha+\alpha_f)v_1}
       + v_1 \left(1-\ln\frac{v_1}{\Gamma} \right)
       - \Gamma - \frac{B}{2}.
\end{align*}
Since there are no cross terms in $v_1$ and $v_2$, the maximum
decouples as
\begin{align*}
    \max_{t_1,t_2} F_{\mathrm{Cross}}(t_1,t_2)
    &= \underbrace{\max_{v_1\in(0,\Gamma]}
       \Bigl\{\frac{B}{2}e^{-(\alpha+\alpha_f)v_1}
       + v_1 \left(1-\ln\frac{v_1}{\Gamma} \right)-\Gamma\Bigr\}}_{=\,\Lambda
       \text{ by~\eqref{eq:Lambdadef}}} \\
    &\quad + \underbrace{\max_{v_2\in(0,B/2]}
       \Bigl\{\Gamma e^{-(\alpha+\alpha_f)v_2}
       + v_2 \left(1-\ln\frac{v_2}{B/2} \right)-\frac{B}{2}\Bigr\}}_{\triangleq\,\Lambda'}.
\end{align*}
It remains to show $\Lambda'=\Gamma-B/2+\Lambda$.
To this end, we use the identity: for any $A>0$, $c>0$, $y\geq 0$,
\begin{equation}\label{eq:trickeq}
    Ae^{-cy} = \max_{x\in(0,A]}\Bigl[-cxy+x \left(1-\ln\frac{x}{A} \right)\Bigr],
\end{equation}
where the maximum is attained at $x=Ae^{-cy}\in(0,A]$.
Applying~\eqref{eq:trickeq} to $\Lambda$ with $A=B/2$,
$c=\alpha+\alpha_f$, $y=v_1$, and to $\Lambda'$ with $A=\Gamma$,
$c=\alpha+\alpha_f$, $y=v_2$:
\begin{align*}
    \Lambda
    &= \max_{v_1\in(0,\Gamma]}\max_{v_2\in(0,B/2]}
    \Bigl[-(\alpha+\alpha_f)v_1v_2
    + v_2 \left(1-\ln\frac{v_2}{B/2} \right)
    + v_1 \left(1-\ln\frac{v_1}{\Gamma} \right)-\Gamma\Bigr], \\
    \Lambda'
    &= \max_{v_2\in(0,B/2]}\max_{v_1\in(0,\Gamma]}
    \Bigl[-(\alpha+\alpha_f)v_1v_2
    + v_1 \left(1-\ln\frac{v_1}{\Gamma} \right)
    + v_2 \left(1-\ln\frac{v_2}{B/2} \right)-\frac{B}{2}\Bigr].
\end{align*}
The two objectives are identical except for the constants $-\Gamma$
and $-B/2$, over the same domain, so $\Lambda-\Lambda'=B/2-\Gamma$,
giving $\Lambda'=\Gamma-B/2+\Lambda$ and hence
\[
    \max_{t_1,t_2}F_{\mathrm{Cross}}(t_1,t_2)
    = \Lambda + \Lambda'
    = \Gamma - \frac{B}{2} + 2\Lambda,
\]
which completes the proof.

\subsubsection{Proof of~\texorpdfstring{\Cref{lmm:upperboundGamma}}{}}

Recall
\begin{align*}
    \Gamma
    \,\triangleq\,
     \left(1-\frac{B}{2} \right)\,\max_{y\in(0,1]}
    \Bigl\{e^{-2\alpha(1-B/2)y}+y(1-\ln y)-1\Bigr\}.
\end{align*}

\paragraph{Lower bound.}
Evaluating at $y=1$ gives $\Gamma\geq(1-B/2)e^{-2\alpha(1-B/2)}$.

\paragraph{Upper bound.}
By convexity of $y\mapsto e^{-\lambda y}$ and the inequality
$e^{-\lambda y}\leq 1-y(1-e^{-\lambda})$ for $y\in[0,1]$, we have
\[
    e^{-2\alpha(1-B/2)y}+y(1-\ln y)-1
    \,\leq\, y \left(e^{-2\alpha(1-B/2)}-\ln y \right)
    \,\triangleq\, g(y).
\]
Since $g''(y)=-1/y<0$, $g$ is strictly concave. The unique maximizer
on $(0,1]$ is 
\[
y^*=\exp(e^{-2\alpha(1-B/2)}-1)\in(0,1]\,,
\]
with
\[
    g(y^*) = \exp  \left(e^{-2\alpha(1-B/2)}-1 \right).
\]
We apply $e^{x-1}\leq 2x-\ln x-1$ for $x\in(0,1]$. To verify,
let $h(x)\triangleq 2x-\ln x-1-e^{x-1}$. One checks
$h(1)=h'(1)=h''(1)=0$ and $h'''(x)<0$, so
$h''(x)>0$, $h'(x)<0$, and $h(x)>0$ on $(0,1)$.
Setting $x=e^{-2\alpha(1-B/2)}\in(0,1]$ so that $-\ln x=2\alpha(1-B/2)$:
\[
    g(y^*)
    \,\leq\, 2e^{-2\alpha(1-B/2)}+2\alpha \left(1-\frac{B}{2} \right)-1.
\]
Multiplying by $(1-B/2)$ and rearranging gives
\begin{align*}
    \Gamma
    \,\leq\, 1-\frac{B}{2}
    -\underbrace{2 \left(1-\frac{B}{2} \right)
     \left(1- \left(1-\frac{B}{2} \right)\alpha-e^{-2\alpha(1-B/2)} \right)}_{
    =\,m_{\mathrm{reg}}^{\mathrm{FMZ}}(B,\alpha)},
\end{align*}
which completes the proof.

\subsubsection{Proof of Lemma~\ref{lmm:tentative}}
\label{sec:pf-tentative}

First, since $\alpha\geq 0$: $e^{-(\alpha+\alpha_f)v}\leq e^{-\alpha_f v}$, therefore:
\[
    \Lambda\leq\max_{v\in(0,\Gamma]}
    \Bigl[\tfrac{B}{2}e^{-\alpha_f v}+v(1-\ln(v/\Gamma))-\Gamma\Bigr].
\]
Substituting $u=\alpha_f v$ 
and multiplying by $\alpha_f$:
\begin{equation}\label{eq:lambda-norm-cmp}
    \alpha_f\Lambda
    \,\leq\,
    \max_{u\in(0,\alpha_f\Gamma]}
    \Bigl[\tfrac{\alpha_f B}{2}e^{-u}
    +u \left(1-\ln\tfrac{u}{\alpha_f\Gamma} \right)
    -\alpha_f\Gamma\Bigr].
\end{equation}
We now prove the following auxiliary inequality.
\begin{claim}\label{clm:aux-ineq}
For $Y\geq 0$, $Z>0$, and $u\in(0,Z]$:
\begin{equation}\label{eq:aux-cmp}
    Ye^{-u}+u \left(1-\ln\tfrac{u}{Z} \right)-Z
    \,\leq\,
    e^{-Z}(e^Y-1).
\end{equation}
\end{claim}

Apply Claim~\ref{clm:aux-ineq} with $Y=\alpha_f B/2$
and $Z=\alpha_f\Gamma$ to bound~\eqref{eq:lambda-norm-cmp}:
\[
    \alpha_f\Lambda
    \,\leq\,
    e^{-\alpha_f\Gamma}(e^{\alpha_f B/2}-1).
\]
Dividing by $\alpha_f/2$ and using the lower bound on $\Gamma$
(Lemma~\ref{lmm:upperboundGamma}):
\[
    2\Lambda
    \,\leq\,
    \frac{2}{\alpha_f}(e^{\alpha_f B/2}-1)e^{-\alpha_f\Gamma}
    \,\leq\,C_{\mathrm{FMZ}}. 
\]
\begin{proof}[Proof of \Cref{clm:aux-ineq}]
Applying the fundamental inequality $e^x \geq 1+x$ with $x = Y-Z+u$, we obtain
$e^{Y-Z+u} \geq 1+Y-Z+u$. Multiplying both sides by $e^{-u}$ yields:
\[
    e^{Y-Z} \,\geq\, e^{-u} + Ye^{-u} - (Z-u)e^{-u}.
\]
Subtracting $e^{-Z}$ provides a lower bound for the right-hand side of~\eqref{eq:aux-cmp}:
\[
    e^{-Z}(e^Y-1) \,\geq\, Ye^{-u} + e^{-u} - (Z-u)e^{-u} - e^{-Z}.
\]
To prove~\eqref{eq:aux-cmp}, it therefore suffices to show that
\[
    e^{-u} - (Z-u)e^{-u} - e^{-Z} \,\geq\, u \left(1-\ln\tfrac{u}{Z} \right) - Z,
\]
which reorganizes as
\[
    Z - u + u\ln(u/Z) \, \geq \, (Z-u)e^{-u} + e^{-Z} - e^{-u}.
\]
Both sides are integrals over $[u,Z]$:
\[
    \int_u^Z  \left(1-\tfrac{u}{t} \right)\mathrm{d}t
    \,\geq\,
    \int_u^Z (e^{-u}-e^{-t})\,\mathrm{d}t,
\]
so it suffices to show $1-u/t \geq e^{-u}-e^{-t}$ for all $t \geq u > 0$, or equivalently $1-e^{-u} \geq u/t - e^{-t}$. 
To that end, consider the strictly concave function $f(x)=1-e^{-x}$ (since $f''(x)=-e^{-x}<0$). Because $f(0)=0$, concavity implies $f(\theta t) \geq \theta f(t)$ for any $\theta\in[0,1]$. Choosing $\theta=u/t \leq 1$ yields $f(u) \geq \frac{u}{t}f(t)$, meaning:
\[
    1-e^{-u} \,\geq\, \frac{u}{t}(1-e^{-t}) \,=\, \frac{u}{t} - \frac{u}{t}e^{-t}.
\]
Since $u/t \leq 1$ and $e^{-t} > 0$, we have $-\frac{u}{t}e^{-t} \geq -e^{-t}$, which simplifies the bound to $1-e^{-u} \geq u/t - e^{-t}$, completing the proof.
\end{proof}

\end{document}